\documentclass[a4paper,12pt]{amsart}
\usepackage{txfonts}

\usepackage{url}
\usepackage{amssymb}
\usepackage{latexsym}
\usepackage{amsfonts}
\usepackage{amsmath}
\usepackage{eucal}
\usepackage{bm}
\usepackage{bbm}
\usepackage{graphicx}
\usepackage[english]{varioref}
\usepackage[nice]{nicefrac}
\usepackage[all]{xy}
\usepackage{amsthm}
\usepackage[colorlinks,linkcolor=blue,urlcolor=cyan,citecolor=red]{hyperref}

\makeatletter
\newcommand*{\rom}[1]{\expandafter\@slowromancap\romannumeral #1@}
\makeatother

\newtheorem{thm}{Theorem}[section]
\newtheorem{prop}[thm]{Proposition}
\newtheorem{lem}[thm]{Lemma}
\newtheorem{defn}[thm]{Definition}
\newtheorem{cor}[thm]{Corollary}
\newtheorem{rmk}[thm]{Remark}
\newtheorem{conj}[thm]{Conjecture}

\begin{document}
\author{Zhongwei Yang}
\address{Current: Beijing International Center For Mathematical Research, Peking University, No.5 Yiheyuan Road, Haidian District, Beijing 100871, China}
%\address{Department of Mathematics, The Hong Kong University of Science and Technology, Clear Water Bay, Kowloon, Hong Kong}
\email{yzw@connect.ust.hk}
\title[Class polynomials for some affine Hecke algebras]{Class polynomials for some affine Hecke algebras}

%\begin{document}

%\title[Class polynomials for some affine Hecke algebras]{Class polynomials for some affine Hecke algebras}

%\dedicatory{Dedicated to }

%\author[Z. Yang]{Zhongwei Yang}
%\address{Department of Mathematics, The Hong Kong University of Science and Technology, Hong Kong}
%\email{yzw@connect.ust.hk}
\begin{abstract}
Class polynomials attached to affine Hecke algebras were first introduced by X.~He in \cite{He1}. They play an important role in the study of affine Deligne-Lusztig varieties. Motivated by \cite{He2}, we compute the class polynomials attached to an affine Hecke algebra of type (twisted) $\widetilde{A}_2$. Using these class polynomials we prove a conjecture of G\"{o}rtz-Haines-Kottwitz-Reuman for the general linear group, unitary group and division algebra of semisimple rank 2. Furthermore, we discuss some interesting patterns on affine Deligne-Lusztig varieties.
\end{abstract}

%\date{\today}

\maketitle
%\tableofcontents

\section*{Introduction}

In this paper, we study class polynomials of affine Hecke algebras and apply them to the study of affine Deligne-Lusztig varieties in some affine flag varieties.

Let's recall the classical Deligne-Lusztig variety first. It was introduced by Deligne and Lusztig in 1976, which was used to construct linear representations of finite groups of Lie type (see \cite{DL}). Let $\mathbb{F}_q$ be a finite field and $\mathbf{k}$ be its algebraic closure. Let $G$ be a reductive group defined over $\mathbb{F}_q$ with a Frobenius automorphism $\sigma$ and let $B$ be a Borel subgroup defined over $\mathbb{F}_q$. We have the Bruhat decomposition $G=\bigsqcup_{w\in W}BwB$, where $W$ is the Weyl group. The \emph{Deligne-Lusztig variety} associated with $w\in W$ is defined by\[X_w=\{gB\in G/B\mid g^{-1}\sigma(g)\in BwB\}.\] It is a locally closed subvariety of the flag variety $G/B$ of dimension $\ell(w)$.

The notion of an affine Deligne-Lusztig variety was first introduced by Rapoport in \cite{R}, which is an analogue of Deligne and Lusztig's classical construction. For simplicity, let $G$ be as above and let $L=\mathbf{k}((\epsilon))$ be the field of the Laurent series. Again we denote $\sigma$ by the automorphism on the loop group $G(L)$. Let $I$ be a $\sigma$-stable Iwahori subgroup of $G(L)$. We have the Iwahori-Bruhat decomposition $G(L)=\bigsqcup_{\tilde{w}\in\widetilde{W}}I\tilde{w}I$ where $\widetilde{W}$ is the Iwahori-Weyl group. By definition, the \emph{affine Deligne-Lusztig variety} $X_{\tilde{w}}(b)$ associated with $\tilde{w}\in\widetilde{W}$ and $b\in G(L)$ is defined as\[X_{\tilde{w}}(b)=\{gI\in G(L)/I\mid g^{-1}b\sigma(g)\in I\tilde{w}I\}.\]
%where $\dot{\tilde{w}}\in G(L)$ is a representative of $\tilde{w}$.

The affine Deligne-Lusztig variety $X_{\tilde{w}}(b)$ plays an important role in the study of the reduction of Shimura varieties with Iwahori level structure. More precisely, it is related to the intersection of the Newton stratum and the Kottwitz-Rapoport stratum. On the special fiber of a Shimura variety there are two important stratifications: one is the Newton stratification whose strata are indexed by certain $\sigma$-conjugacy classes $[b]\subset G(L)$; the other is the Kottwitz-Rapoport stratification whose strata are indexed by specific elements $\tilde{w}$ of the extended affine Weyl group $\widetilde{W}$ (see \cite{GHKR1}, \cite{H} and \cite{R} for details).

We are particularly interested in the following questions:
\begin{itemize}
\item When is $X_{\tilde{w}}(b)\neq\emptyset$?
\item If $X_{\tilde{w}}(b)$ is nonempty, what is $\dim X_{\tilde{w}}(b)$?
\end{itemize}
 It has been studied by many other authors see \cite{GHKR2}, \cite{GH}, \cite{GHN}, \cite{He2}, \cite{Re1}, \cite{Re2} and \cite{V}. But an explicit answer for these questions is still unknown for general $\tilde{w}$ and $b$.

In \cite{He2}, He obtained a remarkable breakthrough in the study of affine Deligne-Lusztig varieties in affine flag varieties, and He showed that the emptiness/nonemptiness pattern and dimension formula for affine Deligne-Lusztig varieties can be deduced by class polynomials of affine Hecke algebras. Thus, we can reduce questions in arithmetic geometry and number theory to questions in representation theory and Lie theory.

In this paper, we obtain the following results:
\begin{itemize}
\item We calculate class polynomials of the affine Hecke algebra of type (twisted) $\widetilde{A}_2$.
\item Using class polynomials we obtain emptiness/nonemptiness pattern of $X_{\tilde{w}}(b)$ and its dimension formula for $GL_3$, $U_3$ and $\mathbb{D}_3^{\times}$ (see \S \ref{ADLVbasic} for details).
\item We verify a conjecture of G\"{o}rtz-Haines-Kottwitz-Reuman for the general linear group, unitary group and division algebra of semisimple rank 2 (see Theorem \ref{ghkr}).
\item We obtain a close formula for suberbasic $b$ on the number of the rational points in $X_{\tilde{w}}(b)$.
\item Based on information of class polynomials and the reduction method, it is expected that the irreducible components of $X_{\tilde{w}}(b)$ of maximal dimension are controlled by the leading coefficient of the corresponding class polynomial. Then we describe the leading leading coefficients of corresponding class polynomials.
\end{itemize}

Here, we give a quick review of the content of this paper. In \S\ref{Prem}, we recall some definitions, e.g. Coxeter systems, (affine) Hecke algebras, loop groups, class polynomials and affine Deligne-Lusztig varieties. We also explain the algorithm of computation of class polynomials. We recall the ``Dimension$=$Degree" theorem as well. Section \S\ref{Class} is the most technical part (see my PhD thesis \cite{Y} for more detailed calculations). We first classify all the conjugacy ($\delta$-conjugacy) classes of $\widetilde{W}$, then we calculate the class polynomials. In \S\ref{Appl}, we apply class polynomials to the study of affine Deligne-Lusztig varieties.

\section{Preliminary data}\label{Prem}
\subsection{Coxeter systems and Hecke algebras}\label{Heckeclass}
To provide some context, let's recall Hecke algebras of Coxeter groups first. We follow \cite{Bo}, let $W$ be a group with identity 1 and $\mathbb{S}$ be a set of generators of $W$ such that $\mathbb{S}=\mathbb{S}^{-1}$ and $1\notin \mathbb{S}$. Every element of $W$ is the product of a finite sequence of elements of $\mathbb{S}$. We also assume that every element of $\mathbb{S}$ is of order 2.

\begin{defn}
$(W,\ \mathbb{S})$ is said to be a Coxeter system if it satisfies the following condition:\\
For $s,\ s'\in \mathbb{S}$, let $m_{ss'}$ be the order of $ss'$ and let $I_0$ be the set of pairs $(s,\ s')$ such that $m_{ss'}$ is finite. The generating set $\mathbb{S}$ and the relations $ss'^{m_{ss'}}=1$ for $(s,\ s')\in I_0$ form a presentation of the group $W$.
\end{defn}

In this condition, we call $W$ a Coxeter group. Let $(W,\ \mathbb{S})$ be a Coxeter system and $w\in W$. We recall the \emph{length of $w$} (with respect to $\mathbb{S}$), denoted by $\ell_{\mathbb{S}}(w)$ or simply by $\ell(w)$ is the smallest integer $r\geqslant0$ such that $w$ is the product of some (or equivalently, any) sequence of $r$ elements of $\mathbb{S}$.

We keep notations as in \cite{HY} $\S$1.1. Let $H$ be a group of automorphisms of the group $W$ that preserves $\mathbb{S}$. Set $W'=W\rtimes H$. Then an element in $W'$ is of the form $w \delta$ for some $w \in W$ and $\delta \in H$. We have that $(w \delta) (w' \delta')=w \delta(w') \delta \delta' \in W'$ with $\delta, \delta' \in H$. For $w \in W$ and $\delta \in H$, we set $\ell(w \delta)=\ell(w)$, where $\ell(w)$ is the length of $w$ in the Coxeter group $(W,\mathbb{S})$. Thus $H$ consists of length $0$ elements in $W'$.

For $J \subset \mathbb{S}$, we denote by $W_J$ the standard parabolic subgroup of $W$ generated by $s_j$ for $s_j\in J$ and by $W^J$ (resp. ${}^J W$) the set of minimal coset representatives in $W/W_J$ (resp. $W_J \backslash W$).

Let $\delta \in H$. For each $\delta$-orbit in $\mathbb{S}$, we pick a single element. Let $g$ be the product of these elements (in any order) and put $c=g \delta \in W'$. We call $c$ a \emph{Coxeter element} of $W'$ (see \cite{Spr}). Let $\mathbb{O}$ be a $\delta$-conjugacy class of $W'$, by definition $\mathbb{O}$ is called \emph{Coxeter} if it contains a Coxeter element of $W'$.

Let $(W,\ \mathbb{S})$ be a Coxeter system and $R_0$ be a commutative ring with 1 (in abuse of notation), and let $q\in\mathbb{C}^*$.

\begin{defn}\label{Hecke1}
The Hecke algebra $H$ (with identity $T_1$) associated to the Coxeter system $(W,\ \mathbb{S})$ over $R_0$ is the associative $R_0$-algebra which is given by the following presentations:
\begin{itemize}
\item Generators: $T_s$, $s\in \mathbb{S}$;
\item Relations:  $T_s^2=(q-1)T_s+qT_1$ and $(T_sT_t)^{m_{st}}=(T_tT_s)^{m_{ts}}$ for all $s,\ t\in \mathbb{S}$.
\end{itemize}
\end{defn}

\subsection{Twisted loop groups}
Let $\mathbf{k}$ be the algebraic closure of a finite field $\mathbb{F}_q$. Let $F=\mathbb{F}_q((\epsilon))$ and $L=\mathbf{k}((\epsilon))$ be the fields of Laurent series. Let $G$ be a connected reductive group over $F$ and splits over a tamely ramified extension of $L$. Let $S\subset G$ be a maximal $L$-split torus defined over $F$, $T=Z_G(S)$ be its centralizer and $N$ be the normalizer of $T$. Since $\mathbf{k}$ is algebraically closed, $G$ is quasi-split over $L$. Furthermore, $T$ is a maximal torus.

\begin{defn}
The algebraic loop group $LG$ associated with $G$ is the ind-group scheme over $\mathbf{k}$ representing the functor\[R\longmapsto LG(R)=G(R((\epsilon)))\]on the category of $\mathbf{k}$-algebras.
\end{defn}

Let $\sigma\in Gal(L/F)$ be the Frobenius automorphism. It induces an automorphism on $G(L)$, and we denote the induced automorphism by the same symbol.

Let $\mathbb{A}$ be the \emph{apartment} of $G(L)$ corresponding to $S$ and $\mathfrak{a}_C$ be a $\sigma$-invariant alcove in $\mathbb{A}$. Let $I\subset G(L)$ be the Iwahori subgroup corresponding to  $\mathfrak{a}_C$ over $L$ and $\widetilde{\mathbb{S}}$ be the set of simple reflections at the walls of  $\mathfrak{a}_C$.

By definition, the \emph{finite Weyl group $W$} associated with $S$ is\[W=N(L)/T(L)\]and the \emph{Iwahori-Weyl group $\widetilde{W}$} associated with $S$ is\[\widetilde{W}=N(L)/T(L)_1\]where $T(L)_1$ is the unique parahoric subgroup of $T(L)$. Since $\widetilde{W}\cong I\backslash G(L)/I$, we embed $\widetilde{W}$ set-theoretically into $G(L)$. Thus we identify representatives of $\widetilde{W}$ in $G(L)$ with elements in $\widetilde{W}$. We have the \emph{Iwahori-Bruhat decomposition}\[G(L)=\bigsqcup_{\tilde{w}\in\widetilde{W}}I\tilde{w}I.\]

Let $\Gamma$ be the absolute Galois group $Gal(\bar{L}/L)$ and $P$ be the $\Gamma$-coinvariants of $X_*(T)$. By \cite{HR}, and by choosing a special vertex in $\mathbb{A}$ we identify $T(L)/T(L)_1$ with $P$. We also obtain a split short exact sequence $1\longrightarrow P\longrightarrow\widetilde{W}\longrightarrow W\longrightarrow1$ and a semi-direct product $\widetilde{W}=P\rtimes W$ . The automorphism $\sigma$ on $G(L)$ induces an automorphism on $\widetilde{W}$ and we denote it by $\delta$. The map gives a bijection on $\widetilde{\mathbb{S}}$. We choose a special vertex in $\mathbb{A}$ such that the previous split short exact sequence is preserved by $\delta$. Thus $\delta$ induces an automorphism on $W$ and we denote it by the same symbol.

Let $\Phi$ be the set of roots of $(G,\ S)$ over $L$ and $\Phi_a$ be the set of affine roots. Let $\mathbb{S}$ be the set of simple roots in $\Phi$. We identify $\mathbb{S}$ with the set of simple reflections in $W$, then  $\mathbb{S}$ is a $\delta$-stable proper subset of $\widetilde{\mathbb{S}}$. We denote by $\Phi^+$ the set of positive roots of $\Phi$ and $\rho$ the half of the sum of all positive roots in $\Phi$.

Let $G_1$ be the subgroup of $G(L)$ generated by all parahoric subgroups. We take $N_1=N(L)\cap G_1$, then by \cite{BT}, the quadruple $(G_1,\ I,\ N_1,\ \widetilde{\mathbb{S}})$ is a double Tits system with affine Weyl group
\[W_a=N_1/(N(L)\cap I).\]

We identify $W_a$ to the Iwahori-Weyl group of the simply connected cover $G_{sc}$ of the the derived group $G_{der}$ of $G$. Let $T_{sc}$ be the maximal torus of $G_{sc}$ giving by $T$. Thus we have $W_a=X_*(T_{sc})_{\Gamma}\rtimes W$. It showed that there exists a reduced root system $\Delta$ such that $W_a=Q^{\vee}(\Delta)\rtimes W(\Delta)$, where $Q^{\vee}(\Delta)$ is the coroot lattice of $\Delta$. We write $Q$ for $Q^{\vee}(\Delta)$ and identify $Q$ with $X_*(T_{sc})_{\Gamma}$ and $W(\Delta)$ with $W$.

\begin{rmk}
The pairs $(W,\ \mathbb{S})$ and $(W_a,\ \widetilde{\mathbb{S}})$ are Coxeter systems. Thus we already have length functions on $W$ and $W_a$ (see \S \ref{Heckeclass}). But the Iwahori-Weyl group $\widetilde{W}$ is not a Coxeter group. For any element $\tilde{w}\in\widetilde{W}$, the length of $\tilde{w}$  (denoted as $\ell(\tilde{w})$) is the number of ``affine root hyperplanes" between $\tilde{w}(\mathfrak{a}_C)$ and $\mathfrak{a}_C$ in $\mathbb{A}$. Let $\Omega$ be the subgroup of $\widetilde{W}$ consisting length 0 elements. The Iwahori-Weyl group $\widetilde{W}$ is a quasi-Coxeter group in the sense that $\widetilde{W}=W_a\rtimes\Omega$.
\end{rmk}

\subsection{Class polynomials}
Analogizing the definition of a Hecke algebra associated to a Coxeter system, we recall a Hecke algebra associated with a quasi-Coxeter system.

\begin{defn}
Let $\widetilde{H}$ be the Hecke algebra associated with $\widetilde{W}$, i.e., $\widetilde{H}$ is the associative $A=\mathbb{Z}[v,\ v^{-1}]$-algebra with basis $T_{\tilde{w}}$ for $\tilde{w}\in\widetilde{W}$ and multiplication is given by \[T_{\tilde{x}}T_{\tilde{y}}=T_{\tilde{x}\tilde{y}},\ \ \ \ \ \  if\  \ \ell(\tilde{x})+\ell(\tilde{y})=\ell(\tilde{x}\tilde{y});\]
\[(T_s-v)(T_s+v^{-1})=0,\ \ \ \ \ for \ \ s\in\widetilde{\mathbb{S}}.\]
\end{defn}
Note that the map $T_{\tilde{w}}\mapsto T_{\delta(\tilde{w})}$ defines an $A$-algebra automorphism of $\widetilde{H}$, and we still denote it as $\delta$.

For any $\tilde{w},\ \tilde{w}'\in\widetilde{W}$ are said to be $\delta$-conjugate if $\tilde{w}'=\tilde{x}\tilde{w}\delta(\tilde{x})^{-1}$ for some $\tilde{x}\in\widetilde{W}$. For $\tilde{w},\ \tilde{w}'\in \widetilde{W}$ and $s_i\in\widetilde{\mathbb{S}}$, we write $\tilde{w}\xrightarrow[]{s_i}_{\delta}\tilde{w}'$ if $\tilde{w}'=s_i\tilde{w}s_{\delta(i)}$ and $\ell(\tilde{w})'\leqslant\ell(\tilde{w})$. We write $\tilde{w}\xrightarrow[]{}_{\delta}\tilde{w}'$ if there is a sequence $\tilde{w}=\tilde{w}_0,\ \tilde{w}_1,\ \cdots,\ \tilde{w}_n=\tilde{w}'$ of elements in $\widetilde{W}$ such that for any $k$, $\tilde{w}_{k-1}\xrightarrow[]{s_i}_{\delta}\tilde{w}_k$ for some $s_i\in\tilde{\mathbb{S}}$. We write $\tilde{w}\approx_{\delta}\tilde{w}'$ if $\tilde{w}\xrightarrow[]{}_{\delta}\tilde{w}'$ and $\tilde{w}'\xrightarrow[]{}_{\delta}\tilde{w}$. Write $\tilde{w}\tilde{\approx}_{\delta}\tilde{w}'$ if $\tilde{w}\approx_{\delta}\tau\tilde{w}'\delta(\tau)^{-1}$ for some $\tau\in\Omega$.

We say that $\tilde{w},\ \tilde{w}'\in \widetilde{W}$ are elementarily strongly $\delta$-conjugate if $\ell(\tilde{w})=\ell(\tilde{w}')$ and there exists $\tilde{x}\in\widetilde{W}$ such that $\tilde{w}'=\tilde{x}\tilde{w}\delta(\tilde{x})^{-1}$ and $\ell(\tilde{x}\tilde{w})=\ell(\tilde{x})+\ell(\tilde{w})$ or $\ell(\widetilde{w}\delta(\widetilde{x})^{-1})=\ell(\tilde{x})+\ell(\tilde{w})$. And we call that $\tilde{w},\ \tilde{w}'$ are strongly $\delta$-conjugate if there is a sequence $\tilde{w}=\tilde{w}_0,\ \tilde{w}_1,\ \cdots,\ \tilde{w}_n=\tilde{w}'$ of elements in $\widetilde{W}$ such that for any $i$, $\tilde{w}_{i-1}$ is elementarily strongly $\delta$-conjugate to $\tilde{w}_i$. We write $\tilde{w}\tilde{\thicksim}_{\delta}\tilde{w}'$ if $\tilde{w}$ and $\tilde{w}'$ are strongly $\delta$-conjugate.

In \cite{HN}, He and Nie proved that minimal length elements $w_{\mathbb{O}}$ of any $\delta$-conjugacy class $\mathbb{O}$ of $\widetilde{W}$ satisfy some special properties, generalizing the results of Geck and Pfeiffer \cite{GP} on finite Weyl groups. These properties play a key role in the study of affine Deligne-Lusztig varieties and affine Hecke algebras. In \cite{He2} and \cite{HN}, it is showed that for any $\delta$-conjugacy class $\mathbb{O}$, we can fix a minimal length representative $w_{\mathbb{O}}$ and the image of $T_{w_{\mathbb{O}}}$ in $\tilde{H}/[\tilde{H},\ \tilde{H}]_{\delta}$ where $[\tilde{H},\ \tilde{H}]_{\delta}$ is the $A$-submodule (we regard $\tilde{H}$ as a left $A$-module) of $\tilde{H}$ generated by all $\delta$-commutators (i.e. for any $h,\ h'\in\tilde{H}$, $[h,\ h']_{\delta}=hh'-h'\delta{(h)}$). Moreover, $T_{w_{\mathbb{O}}}$ is independent of the choice of $w_{\mathbb{O}}$ and forms a basis of  $\tilde{H}/[\tilde{H},\ \tilde{H}]_{\delta}$. It is proved in \cite[Theorem 6.7]{HN} that

\begin{thm}[He-Nie]
The elements $T_{\tilde{w}_{\mathbb{O}}}$ form an $A$-basis of $\widetilde{H}/[\widetilde{H},\ \widetilde{H}]_{\delta}$, here $\mathbb{O}$ runs over all the $\delta$-conjugacy classes of $\widetilde{W}$.
\end{thm}

From now on we denote $T_{\tilde{w}_{\mathbb{O}}}$ as $T_{\mathbb{O}}$ for simplicity. For any $\tilde{w}\in\widetilde{W}$ and a $\delta$-conjugacy class $\mathbb{O}$, there exists a unique $f_{\tilde{w},\mathbb{O}}\in A$ such that\[T_{\tilde{w}}\equiv\sum_{\mathbb{O}}f_{\tilde{w},\mathbb{O}}T_{\mathbb{O}}\ \ \mod\ [\widetilde{H},\ \widetilde{H}]_{\delta}.\]

$f_{\tilde{w},\mathbb{O}}$ is a polynomial in $\mathbb{Z}[v-v^{-1}]$ with nonnegative coefficient. This is called the \textbf{class polynomial} attached to $\tilde{w}$ and $\mathbb{O}$, and it can be constructed inductively as follows:

If $\tilde{w}$ is a minimal length element in a $\delta$-conjugacy class of $\widetilde{W}$, then we set \[f_{\tilde{w},\mathbb{O}}=\left\{
                                                                                                                                                 \begin{array}{ll}
                                                                                                                                                   1, & \text{if}\ \tilde{w}\in\mathbb{O},\\
                                                                                                                                                   0, & \text{if}\ \tilde{w}\notin\mathbb{O}.
                                                                                                                                                 \end{array}
                                                                                                                                               \right.
\]
If $\tilde{w}$ is not a minimal length element in its $\delta$-conjugacy class and for any $\tilde{w}'\in\tilde{W}$ with $\ell(\tilde{w}')<\ell(\tilde{w})$, $f_{\tilde{w}',\mathbb{O}}$ is constructed. By \cite{HN}, there exists $\tilde{w}_1\approx_{\delta}\tilde{w}$ and $s_i\in\widetilde{\mathbb{S}}$ such that $\ell(s_i\tilde{w}_1s_{\delta(i)})<\ell(\tilde{w}_1)=\ell(\tilde{w}).$ In this case, $\ell(s_i\tilde{w})<\ell(\tilde{w})$ and we define $f_{\tilde{w}, \mathbb{O}}$ as $$f_{\tilde{w},\mathbb{O}}=(v-v^{-1})f_{s_i\tilde{w}_1,\mathbb{O}}+f_{s_i\tilde{w}_1s_{\delta(i)},\mathbb{O}}.$$

\subsection{The ``Dimension$=$Degree'' theorem}\label{DD}
We keep the notations as before. Let $\textbf{Fl}=G(L)/I$ be the $fppf$ quotient, then $\textbf{Fl}$ is represented by an ind-scheme, ind-projective over $\mathbf{k}$. We also have the Iwahori-Bruhat decomposition:\[\textbf{Fl}=\bigsqcup_{\tilde{w}\in \widetilde{W}}I\tilde{w}I/I.\]

\begin{defn}
 For $\tilde{w}\in\widetilde{W}$ and $b\in G(L)$, the affine Deligne-Lusztig variety associated with $\tilde{w}$ and $b$ is the locally closed sub-ind scheme $X_{\tilde{w}}(b)(\mathbf{k})$ in the affine flag variety $\mathbf{Fl}$ defined as \[X_{\tilde{w}}(b)(\mathbf{k})=\{gI\in G(L)/I\mid g^{-1}b\sigma(g)\in I\tilde{w}I\}.\]
\end{defn}

\begin{rmk}
The affine Deligne-Lusztig variety $X_{\tilde{w}}(b)(\mathbf{k})$ is a finite-dimensional $\mathbf{k}$-scheme and locally of finite type over $\mathbf{k}$.
\end{rmk}

We call an element $\tilde{w}\in\widetilde{W}$ a $\delta$-\emph{straight element} if and only if for any $n\in\mathbb{N}$ we have $\ell(\tilde{w}\delta(\tilde{w})\cdots\delta^{n-1}(\tilde{w}))=n\ell(\tilde{w})$. We call a $\delta$-conjugacy class in $\widetilde{W}$ \emph{straight} if it contains some straight element. We denote by $\cdot_{\sigma}$ the $\sigma$-conjugation action on $G(L)$ and it is defined by that for any $g,\ g'\in G(L)$, $g\cdot_{\sigma}g'=gg'\sigma(g)^{-1}$. We have the following Kottwitz's classification of $\sigma$-conjugacy classes $B(G)$ on $G(L)$.

\begin{thm}[Kottwitz, He]
For any straight $\delta$-conjugacy class $\mathbb{O}$ of $\widetilde{W}$, we fix a minimal length representative $\tilde{w}_{\mathbb{O}}$. Then\[G(L)=\bigsqcup_{\mathbb{O}}G(L)\cdot_{\sigma}\tilde{w}_{\mathbb{O}},\]here $\mathbb{O}$ runs over all the straight $\delta$-conjugacy classes of $\widetilde{W}$.
\end{thm}

Let $(P/Q)_{\delta}$ be the $\delta$-coinvariants on $P/Q$, let \[\kappa:\widetilde{W}\longrightarrow\widetilde{W}/W_a\cong P/Q\longrightarrow(P/Q)_{\delta}\] be the natural projection. We call $\kappa$ the \emph{Kottwitz map}.

Let $P_{\mathbb{Q}}=P\otimes_{\mathbb{Z}}\mathbb{Q}$ and $P_{\mathbb{Q}}/W$ be the quotient of $P_{\mathbb{Q}}$ by the natural action of $W$. We can identify $P_{\mathbb{Q}}/W$ with $P_{\mathbb{Q},+}=\{\chi\in P_{\mathbb{Q}}\mid \alpha(\chi)\geqslant0,\ for\ all\ \alpha\in \Phi^+\}.$ Let $P^{\delta}_{\mathbb{Q},+}$ be the set of $\delta$-invariant points in $P_{\mathbb{Q},+}$.

Since the image of $W\rtimes\langle\delta\rangle$ in $Aut(W)$ is a finite group, for each $\tilde{w}=t^{\chi}w\in\widetilde{W}$, there exists $n\in\mathbb{N}$ such that $\delta^n=1$ and $w\delta(w)\delta^2(w)\cdots\delta^{n-1}(w)=1$. Then $\tilde{w}\delta(\tilde{w})\delta^2(\tilde{w})\cdots\delta^{n-1}(\tilde{w})=t^{\lambda}$ for some $\lambda\in P$. Let $\nu_{\tilde{w}}=\lambda/n\in P_{\mathbb{Q}}$ and $\bar{\nu}_{\tilde{w}}$ be the corresponding element in $P_{\mathbb{Q},+}$. Note that $\nu_{\tilde{w}}$ is independent of the choice of $n$. Let $\bar{\nu}_{\tilde{w}}$ be the unique element in $P_{\mathbb{Q},+}$ that lies in the $W$-orbit of $\nu_{\tilde{w}}$. Since $t^{\lambda}=\tilde{w}t^{\delta(\lambda)}\tilde{w}^{-1}=t^{w\delta(\lambda)}$, $\bar{\nu}_{\tilde{w}}\in P^{\delta}_{\mathbb{Q},+}$. We call the map $\widetilde{W}\longrightarrow P^{\delta}_{\mathbb{Q},+}$ with $\tilde{w}\mapsto\bar{\nu}_{\tilde{w}}$ the \emph{Newton map}. We define \[f:\widetilde{W}\longrightarrow P^{\delta}_{\mathbb{Q},+}\times(P/Q)_{\delta}\ \ \ \ \ by\ \ \ \  \tilde{w}\longmapsto(\bar{\nu}_{\tilde{w}},\ \kappa(\tilde{w})).\]Follows \cite{K}, $f$ %is the restriction to $\widetilde{W}$ of the map $G(L)\longrightarrow P^{\delta}_{\mathbb{Q},+}\times(P/Q)_{\delta}$, and it
is constant on each $\delta$-conjugacy class of $\widetilde{W}$. We denote the image of the map by $B(\widetilde{W},\ \delta)$.

\begin{rmk}
(1) There is a bijection between the set of $\sigma$-conjugacy classes $B(G)$ of $G(L)$ and the set of straight conjugacy classes of $\widetilde{W}$.

(2) Any $\sigma$-conjugacy class of $G(L)$ contains a representative in $\widetilde{W}$. Moreover, the map $f:\widetilde{W}\longrightarrow P^{\delta}_{\mathbb{Q},+}\times(P/Q)_{\delta}$ is in fact the restriction of a map defined on $G(L)$ as $b\mapsto (\bar{\nu}_{b},\ \kappa(b))$ and we call $\bar{\nu}_{b}$ the \emph{Newton vector} of $b$.
\end{rmk}

The ``Dimension$=$Degree'' theorem is a main result in \cite{He2}, we quote it as follows:

\begin{thm}[He]\label{DimDeg}
Let $b\in G(L)$ and $\tilde{w}\in\widetilde{W}$. Then\[\dim(X_{\tilde{w}}(b))=\max_\mathbb{O}\frac{1}{2}(\ell(\tilde{w})+\ell(\mathbb{O})+\deg(f_{\tilde{w},\mathbb{O}}))-\langle\bar{\nu}_b,2\rho\rangle,\]
here $\mathbb{O}$ runs over $\delta$-conjugacy classes of $\widetilde{W}$ with $f(\mathbb{O})=f(b)$ and $\ell(\mathbb{O})$ is the length of any minimal length element in $\mathbb{O}$.
\end{thm}

\begin{rmk}
We use the convention that the dimension of an empty variety and the degree of a zero polynomial are both $-\infty$.
\end{rmk}

Theorem \ref{DimDeg} relates the dimension of affine Deligne-Lusztig varieties to the degree of the class polynomials which provides both the theoretic and practical way to determine the dimension of affine Deligne-Lusztig varieties. It shows that the dimension and emptiness/nonemptiness pattern of affine Deligne-Lusztig varieties $X_{\tilde{w}}(b)$ only depend on the data $(\widetilde{W}, \delta, \tilde{w}, f(b))$ and thus independent of the choice of $G$. And it implies that the emptiness/nonemptiness and dimension formula of affine Deligne-Lusztig varieties only rely on the reduction method. This theorem inspired my study of the class polynomials of affine Hecke algebras.

\section{Class polynomials for affine Hecke algebras of type (twisted) $\widetilde{A}_2$}\label{Class}
We calculate class polynomials for affine Hecke algebras of type $\widetilde{A}_2$ for split groups of adjoint type ($PGL_3$). We let $\widetilde{W}$ be the Iwahori-Weyl group, then $\widetilde{W}=P\rtimes W=W_a\rtimes\Omega$ where $P$ is the coweight lattice and $W_a$ is the affine Wely group generated by the simple reflections $\tilde{\mathbb{S}}=\{s_0,\ s_1,\ s_2\}$. Here $\Omega=\langle\tau\rangle$ with $\tau^3=1$ and $\tau s_0 \tau^{-1}=s_1,\ \tau s_1 \tau^{-1}=s_2,\ \tau s_2 \tau^{-1}=s_0$. Let $\alpha_1$ and $\alpha_2$ be the corresponding simple roots.

Before we calculate class polynomials, let me quote the following lemma \cite[Lemma 5.1]{HN} which will be heavily used in the following computations.
\begin{lem}[He-Nie]\label{Hn}
Let $\tilde{w},\ \tilde{w}'\in\widetilde{W}$ with $\tilde{w}\tilde{\thicksim}\tilde{w}'$. Then\[T_{\tilde{w}}\equiv T_{\tilde{w}'}\ \mod [\widetilde{H},\widetilde{H}].\]
\end{lem}
As you will see that the whole section is the most technique part in this paper, we'd like to make some comments here. It is quite hard to calculate class polynomials and we know little about them in general. And till now, we have found neither an efficient nor a systematic way to calculate them in general situations. Even in the following low rank cases (of type (twisted) $\widetilde{A}_2$), the calculations are quite difficult and highly nontrivial. From the definition of a class polynomial $f_{\tilde{w},\mathbb{O}}$, two parameters are involved: an extended affine Weyl group element $\tilde{w}$ and a ($\delta$-) conjugacy class $\mathbb{O}$ in $\widetilde{W}$. As we know that for any $\tilde{w}$ we can be weakly reduced to $\tilde{w}'$ i.e. ($\tilde{w}\xrightarrow[]{}_{\delta}\tilde{w}'$) $\tilde{w}\xrightarrow[]{}\tilde{w}'$, where $\tilde{w}'$ is a minimal length element in the ($\delta$-) conjugacy class $\mathbb{O}$ of $\tilde{w}$ \cite{HN}. To calculate class polynomials, we will by way of doing the reduction procedure, and thus there are quite lot paths involved which cause enormous complexities and extreme difficulties. There is an illusion that the calculations in the coming subsections looks routine. Since for different $\tilde{w}$ we have to work hard on choosing correct path to make the calculations work, actually, those calculations are very subtle and difficult and you will see how in the following subsections.
\subsection{Class polynomials for elements in $W_a$}\label{split1}
We first classify all $\widetilde{W}$-conjugacy classes in $W_a$.
\begin{lem}\label{leng1}
Note that all elements in $W_a$ with length 1 are $\widetilde{W}$-conjugate. Let $\mathbb{O}_1$ be the $\widetilde{W}$-conjugacy class in $W_a$ with minimal length 1. Then
$$\mathbb{O}_1=\{t^{k\alpha_1}s_1, \ t^{k\alpha_2}s_2, \ t^{k(\alpha_1+\alpha_2)}s_1s_2s_1\mid k\in\mathbb{Z}\}.$$
\end{lem}
\begin{lem}\label{leng2}
All elements in $W_a$ with minimal length 2 are $\widetilde{W}$-conjugate. Let $\mathbb{O}_2$ be the $\widetilde{W}$-conjugacy class in $W_a$ with minimal length 2. Then
$$\mathbb{O}_2=\{t^{\lambda}s_1s_2,\  t^{\lambda}s_2s_1 \mid \lambda\in Q\}.$$
\end{lem}
Let $Q_{sh}=Q-\{k(\alpha_1+2\alpha_2),\ k(2\alpha_1+\alpha_2),\ k(\alpha_1-\alpha_2)\mid k\in\mathbb{Z}\}$.
\begin{lem}\label{lenglambda}
For any $\lambda\in P_+\cap Q_{sh}$, i.e. $\lambda=m\alpha_1+n\alpha_2$ where $m,\ n\in\mathbb{Z}$ and $1\leqslant m\leqslant n\leqslant2m-1$ or $1\leqslant n\leqslant m\leqslant2n-1$. We set $\mathbb{O}_{\lambda}=\widetilde{W}\cdot t^{\lambda}$, then $$\mathbb{O}_{\lambda}=\{t^{m\alpha_1+n\alpha_2},\ t^{(n-m)\alpha_1+n\alpha_2},\ t^{m\alpha_1+(m-n)\alpha_2},\ t^{-n\alpha_1-m\alpha_2},\ t^{-n\alpha_1+(m-n)\alpha_2},\ t^{(n-m)\alpha_1-m\alpha_2}\},$$ with $\ell(\mathbb{O}_{\lambda})=\ell(t^{\lambda})$. If $\lambda=m(\alpha_1+2\alpha_2)$, then $$\mathbb{O}_{\lambda}=\{t^{m(\alpha_1+2\alpha_2)},\ t^{-m(2\alpha_1+\alpha_2)},\ t^{m(\alpha_1-\alpha_2)}\},$$ or if $\lambda=m(2\alpha_1+\alpha_2)$, then $$\mathbb{O}_{\lambda}=\{t^{m(2\alpha_1+\alpha_2)},\ t^{-m(\alpha_1+2\alpha_2)},\ t^{-m(\alpha_1-\alpha_2)}\},$$with $\ell(\mathbb{O}_{\lambda})=\ell(t^{\lambda})$.
\end{lem}
\begin{lem}\label{lengi}
For $i\in\mathbb{N}_+$, set $\mathbb{C}_i=\{t^{k\alpha_1+i\alpha_2}s_1,\ t^{(-i)\alpha_1+k\alpha_2}s_2,\ t^{k\alpha_1+(k-i)\alpha_2}s_1s_2s_1\mid k\in\mathbb{Z}\}$ and $\mathbb{C}'_i=\{t^{k\alpha_1-i\alpha_2}s_1,\ t^{i\alpha_1+k\alpha_2}s_2,\ t^{(k-i)\alpha_1+k\alpha_2}s_1s_2s_1\mid k\in\mathbb{Z}\}$. Then $\mathbb{C}_i$ and $\mathbb{C}'_i$ are $\widetilde{W}$-conjugacy classes of $W_a$ with minimal length $3i$ if $i$ is odd or with minimal length $3i+1$ if $i$ is even.
\end{lem}
The proofs of the above four lemmas \ref{leng1}, \ref{leng2}, \ref{lenglambda} and \ref{lengi} are direct and methods are quite similar to each other. It is enough for us to give a proof of \ref{leng2}  as an example, and we omit the others.
\begin{proof}[Proof of Lemma \ref{leng2}]
Let $\mathbb{O}'_2=\{t^{\lambda}s_1s_2,\  t^{\lambda}s_2s_1 \mid \lambda\in Q\}$ it is quite easy to show that $\widetilde{W}\cdot\mathbb{O}'_2\subset\mathbb{O}'_2$. Now we prove that for any element $\tilde{w}\in\mathbb{O}'_2$, then $\tilde{w}$ is $\widetilde{W}$-conjugate to $s_0s_1$. We use induction on length $\ell(\tilde{w})$. Assume for any $\tilde{w}'\in\mathbb{O}'_2$ and $\ell(\tilde{w}')<\ell(\tilde{w})$, then $\tilde{w}'$ is $\widetilde{W}$-conjugate to $s_0s_1$. We know $\tilde{w}$ can be written uniquely as $xt^{\mu}y$ where $x\in W,\ \mu\in Q\cap P_+,\ y\in {}^{I(\mu)}W$, here $I(\mu)=\{s_i\in \mathbb{S}\mid\langle\mu,\alpha_i\rangle=0\}$. If $\tilde{w}=xt^{\mu}y$, $x=s_{i_1}\cdots s_{i_r}\neq 1$(reduced expression) with $s_{i_j}\in\mathbb{S}$ for $1\leqslant j\leqslant r$. Then set $\tilde{w}_1=s_{i_1}\tilde{w}s_{i_1}=s_{i_2}\cdots s_{i_r}t^{\mu}ys_{i_1}$, then $$\ell(\tilde{w}_1)=\ell(\mu)+\ell(x)-1-\ell(ys_{i_1})\leqslant\ell(\mu)+\ell(x)-\ell(y)=\ell(\tilde{w}).$$ If $\ell(\tilde{w}_1)<\ell(\tilde{w})$, by induction we are done. If $\ell(\tilde{w}_1)=\ell(\tilde{w})$, we set $\tilde{w}_2=s_{i_2}\tilde{w}_1s_{i_2}=s_{i_3}\cdots s_{i_r}t^{\mu}ys_{i_1}s_{i_2}$ and $$\ell(\tilde{w}_2)=\ell(\mu)+\ell(x)-2-\ell(ys_{i_1}s_{i_2})\leqslant\ell(\mu)+\ell(x)-1-\ell(ys_{i_1})=\ell(\tilde{w}_1).$$
By the same argument, we can reduce to the case $\tilde{w}=t^{\mu}s_1s_2$ or $t^{\mu}s_2s_1$ with $\mu\in Q\cap P_+$. Now if $\tilde{w}=t^{m\alpha_1+n\alpha_2}s_1s_2$ with $1\leqslant m\leqslant n<2m-1$ or $1\leqslant n<m\leqslant2n$. If $\ell(\tilde{w})=2$, obvious. If $\ell(\tilde{w})>2$, set $\tilde{w}_1=s_0\tilde{w}s_0=t^{(1-n)\alpha_1+(2-m)\alpha_2}s_2s_1$, we have $\ell(\tilde{w}_1)=\ell(\tilde{w})-2.$ By induction $\tilde{w}_1$ is $\widetilde{W}$-conjugate to $s_0s_1$, so is $\tilde{w}$. For $\tilde{w}=t^{k\alpha_1+2k\alpha_2}s_1s_2$ $(k\geqslant1)$, check directly $\tilde{w}_1=s_1\tilde{w}s_1$, then $\ell(\tilde{w}_1)=\ell(\tilde{w})-2$, which deduce that $\tilde{w}$ is $\widetilde{W}$-conjugate to $s_0s_1$. For $\tilde{w}=t^{k\alpha_1+(2k-1)\alpha_2}s_1s_2$ then $\tilde{w}_1=s_1s_0\tilde{w}s_0s_1$ and $\ell(\tilde{w}_1)=\ell(\tilde{w})-2$. By induction, $\tilde{w}_1$ is $\widetilde{W}$-conjugate to $s_0s_1$, so is $\tilde{w}$. If $\tilde{w}=t^{m\alpha_1+n\alpha_2}s_2s_1$ with $1\leqslant m\leqslant n\leqslant2m$ or $1\leqslant n<m<2n-1$. If $\ell(\tilde{w})=2$, obvious. If $\ell(\tilde{w})>2$, set $\tilde{w}_1=s_0\tilde{w}s_0=t^{(2-n)\alpha_1+(1-m)\alpha_2}s_1s_2$. Also $\ell(\tilde{w}_1)=\ell(\tilde{w})-2.$ By induction, $\tilde{w}_1$ is $\widetilde{W}$-conjugate to $s_0s_1$, so is $\tilde{w}$. For $\tilde{w}=t^{2k\alpha_1+k\alpha_1}s_2s_1$ $(k\geqslant1)$, check directly $\tilde{w}_1=s_2\tilde{w}s_2$, then $\ell(\tilde{w}_1)=\ell(\tilde{w})-2$, which deduce $\tilde{w}$ is $\widetilde{W}$-conjugate to $s_0s_1$. For $\tilde{w}=t^{(2k-1)\alpha_1+k\alpha_2}s_2s_1$ $(k\geqslant1)$, then $\tilde{w}_1=s_2s_0\tilde{w}s_0s_2$ and $\ell(\tilde{w}_1)=\ell(\tilde{w})-2$. By induction, $\tilde{w}_1$ is $\widetilde{W}$-conjugate to $s_0s_1$, so is $\tilde{w}$.
\end{proof}
\begin{thm}
The $\widetilde{W}$-conjugacy classes in $W_a$ are as following:
\[\{Id\},\ \mathbb{O}_1,\ \mathbb{O}_2,\ \mathbb{O}_{\lambda},\ \mathbb{C}_i,\ and\ \mathbb{C}'_i\] where $\lambda\in P_+\cap Q$, and $i\in\mathbb{N}_+$.
\end{thm}
\begin{proof}
Following Lemmas \ref{leng1}, \ref{leng2}, \ref{lenglambda}, \ref{lengi} and $$W_a=\{Id\}\sqcup\mathbb{O}_1\sqcup\mathbb{O}_2\sqcup(\sqcup_{\lambda\in P_+\cap Q}\mathbb{O}_{\lambda})\sqcup(\sqcup_{i\in\mathbb{N}_+}(\mathbb{C}_i\sqcup\mathbb{C}'_i)),$$ we obtain the theorem directly.
\end{proof}
Before the calculating of class polynomials, for convenience, we fix some notations. We set $\mathbb{I}=\{t^{m\alpha_1+n\alpha_2}w\mid m\in\mathbb{N}_+,\ n\in\mathbb{Z},\ 1-m\leqslant n\leqslant2m-1,\ w\in W\}\sqcup\{t^{k\alpha_1+2k\alpha_2},\ t^{k\alpha_1+2k\alpha_2}s_2,\ t^{k\alpha_1+2k\alpha_2}s_2s_1,\ t^{k(\alpha_1-\alpha_2)},\ t^{k(\alpha_1-\alpha_2)}s_1,\ t^{k(\alpha_1-\alpha_2)}s_2\mid k\in\mathbb{N}_+\}\sqcup\{Id,\ s_2\}$. Since $W_a=\mathbb{I}\cup(\tau\cdot\mathbb{I})\cup(\tau^2\cdot\mathbb{I})$ and \ref{Hn}, it is sufficient to consider elements in $\mathbb{I}$. For $\alpha\in Q$ we set $Q_{\alpha}=\{\lambda\in Q\cap P_+\mid\lambda<\alpha\}\cap Q_{sh}$. For $\lambda\in Q\cap P_+$, let $\mathbb{C}(t^{\lambda}s_1)$ to be the $\widetilde{W}$-conjugacy class contains $t^{\lambda}s_1$, and let $\mathbb{O}^{\leqslant}_{t^{\lambda}s_1}$ to be the set of all $\mathbb{C}_i$ with $\ell(\mathbb{C}_i)\leqslant\ell(\mathbb{C}(t^{\lambda}s_1))$. Similarly, let $\mathbb{C}'(t^{\lambda}s_2)$ to be the $\widetilde{W}$-conjugacy class contains $t^{\lambda}s_2$, and let $\mathbb{O}'^{\leqslant}_{t^{\lambda}s_2}$ to be the set of all $\mathbb{C}'_i$ with $\ell(\mathbb{C}'_i)\leqslant\ell(\mathbb{C}'(t^{\lambda}s_2))$. We calculate directly and obtain the following class polynomials. In the following, since the computations are quite similar, we list some typical calculations and others are easily obtained by similar methods. For other detailed calculations see \cite[Chapter 3]{Y}.
\begin{prop}\label{cplambda}
If $\tilde{w}\in\mathbb{O}_{\lambda}$ where $\lambda\in P_+\cap Q$. Then \[f_{\tilde{w},\mathbb{O}}=\left\{
                                                             \begin{array}{ll}
                                                               1, & \mathbb{O}=\mathbb{O}_{\lambda}\\
                                                               0, & otherwise.
                                                             \end{array}
                                                           \right.\]
\end{prop}
\begin{proof}
Since $\mathbb{O}_{\lambda}$ only has finitely many elements for each $\lambda$ and any two elements in $\mathbb{O}_{\lambda}$ are $\tilde{\thicksim}$, the proposition follows directly from the definition of class polynomials.
\end{proof}
In the following sections, for any $\tilde{w},\ \tilde{w}'\in\widetilde{W}$, we will always write $$T_{\tilde{w}}\equiv T_{\tilde{w}'} \mod[\widetilde{H},\widetilde{H}]\  (\text{or}\ [\widetilde{H},\widetilde{H}]_{\delta})$$ as  $$T_{\tilde{w}}\equiv T_{\tilde{w}'}$$ for short.
\begin{prop}\label{cpa21}
(1) If $\tilde{w}=t^{k(\alpha_1+2\alpha_2)}s_2s_1$, with $k\in\mathbb{N}_+$, then\[f_{\tilde{w},\mathbb{O}}=\left\{
                                                             \begin{array}{ll}
                                                               (v-v^{-1})^2, & \mathbb{O}=\mathbb{O}_{\lambda},\ \lambda\in Q_{(k-1)\alpha_1+(2k-2)\alpha_2} \\
                                                               (v-v^{-1}), & \mathbb{O}\in\mathbb{O}^{\leqslant}_{t^{k\alpha_1+(2k-1)\alpha_2}s_1}
                                                               \cup\mathbb{O}'^{\leqslant}_{t^{(k-1)\alpha_1+(2k-2)\alpha_2}s_2}\\
                                                               1, & \mathbb{O}=\mathbb{O}_2\\
                                                               0, & otherwise.
                                                             \end{array}
                                                           \right.\]
(2) If $\tilde{w}=t^{k(\alpha_1+2\alpha_2)}s_1s_2$, with $k\in\mathbb{N}_+$, then\[f_{\tilde{w},\mathbb{O}}=\left\{
                                                             \begin{array}{ll}
                                                               (v-v^{-1})^2, & \mathbb{O}=\mathbb{O}_{\lambda},\ \lambda\in Q_{k\alpha_1+2k\alpha_2} \\
                                                               (v-v^{-1}), & \mathbb{O}\in\mathbb{O}^{\leqslant}_{t^{k\alpha_1+(2k-1)\alpha_2}s_1}
                                                               \cup\mathbb{O}'^{\leqslant}_{t^{k(\alpha_1+2\alpha_2)}s_2}\\
                                                               1, & \mathbb{O}=\mathbb{O}_2\\
                                                               0, & otherwise.
                                                             \end{array}
                                                           \right.\]
\end{prop}
\begin{prop}\label{cpb12}
(1) If $\tilde{w}=t^{k(2\alpha_1+\alpha_2)}s_1s_2$, with $k\in\mathbb{N}_+$, then\[f_{\tilde{w},\mathbb{O}}=\left\{
                                                             \begin{array}{ll}
                                                               (v-v^{-1})^2, & \mathbb{O}=\mathbb{O}_{\lambda},\ \lambda\in Q_{(2k-2)\alpha_1+(k-1)\alpha_2} \\
                                                               (v-v^{-1}), & \mathbb{O}\in \mathbb{O}^{\leqslant}_{t^{(k-1)(2\alpha_1+\alpha_2)}s_1}\cup\mathbb{O}'^{\leqslant}_{t^{(2k-1)\alpha_1+k\alpha_2}s_2}\\
                                                               1, & \mathbb{O}=\mathbb{O}_2\\
                                                               0, & otherwise.
                                                             \end{array}
                                                           \right.\]
(2) If $\tilde{w}=t^{k(2\alpha_1+\alpha_2)}s_2s_1$, with $k\in\mathbb{N}_+$, then\[f_{\tilde{w},\mathbb{O}}=\left\{
                                                             \begin{array}{ll}
                                                               (v-v^{-1})^2, & \mathbb{O}=\mathbb{O}_{\lambda},\ \lambda\in Q_{k(2\alpha_1+\alpha_2)} \\
                                                               (v-v^{-1}), & \mathbb{O}\in \mathbb{O}^{\leqslant}_{t^{k(2\alpha_1+\alpha_2)}s_1}\cup\mathbb{O}'^{\leqslant}_{t^{(2k-1)\alpha_1+k\alpha_2}s_2}\\
                                                               1, & \mathbb{O}=\mathbb{O}_2\\
                                                               0, & otherwise.
                                                             \end{array}
                                                           \right.\]
\end{prop}
Here we give a proof of \ref{cpa21}. Symmetrically, we obtain \ref{cpb12}.
\begin{proof}[Proof of Proposition \ref{cpa21}] For (1), we have
\begin{align*}T_{t^{k(\alpha_1+2\alpha_2)}s_2s_1}\equiv&(v-v^{-1})(T_{t^{k\alpha_1+(2k-1)\alpha_2}s_1}+T_{t^{(k-1)\alpha_1+(2k-2)\alpha_2}s_1})\\
&+(v-v^{-1})T_{t^{(k-1)\alpha_1+(2k-2)\alpha_2}s_2}+T_{t^{(k-1)(\alpha_1+2\alpha_2)}s_2s_1}\\
\cdots&\cdots\\
\equiv&(v-v^{-1})[T_{t^{\alpha_1+\alpha_2}s_1}+\sum_{i=2}^{k}(T_{t^{(i-1)(\alpha_1+2\alpha_2)}s_1}+T_{t^{i\alpha_1+(2i-1)\alpha_2}s_1})]\\
&+(v-v^{-1})\sum_{i=1}^{k-1}T_{t^{i(\alpha_1+2\alpha_2)}s_2}+T_{\mathbb{O}_2}\\
\cdots&\cdots\\
\equiv&T_{\mathbb{O}_2}+(v-v^{-1})^2\sum_{\lambda\in Q_{(k-1)\alpha_1+(2k-2)\alpha_2}}T_{\mathbb{O}_{\lambda}}\\
&+(v-v^{-1})\sum_{\mathbb{O}\in\mathbb{O}^{\leqslant}_{t^{k\alpha_1+(2k-1)\alpha_2}s_1}\cup\mathbb{O}'^{\leqslant}_{t^{(k-1)(\alpha_1+2\alpha_2)}s_2}}T_{\mathbb{O}}.
\end{align*}
Thus (1) is proved.
By $T_{t^{k(\alpha_1+2\alpha_2)}s_1s_2}\equiv(v-v^{-1})T_{t^{k(\alpha_1+2\alpha_2)}s_2}+T_{t^{k(\alpha_1+2\alpha_2)}s_2s_1}$ and (1), (2) is proved.
\end{proof}
\begin{prop}\label{cpa1221}
For $k\in\mathbb{N}_+$, we have$$t^{k\alpha_1+(2k-1)\alpha_2}s_1s_2\tilde{\thicksim}t^{k\alpha_1+(2k-1)\alpha_2}s_2s_1,$$thus $T_{t^{k\alpha_1+(2k-1)\alpha_2}s_1s_2}\equiv T_{t^{k\alpha_1+(2k-1)\alpha_2}s_2s_1}$. If $\tilde{w}=t^{k\alpha_1+(2k-1)\alpha_2}s_1s_2$ or $t^{k\alpha_1+(2k-1)\alpha_2}s_2s_1$ with $k=1$, then $\tilde{w}$ is already of minimal length in $\mathbb{O}_2$. So $f_{\tilde{w},\mathbb{O}_2}=1$ and 0 for other $\mathbb{O}$. When $k\geqslant2$, then \[f_{\tilde{w},\mathbb{O}}=\left\{
                                                             \begin{array}{ll}
                                                               (v-v^{-1})^2, & \mathbb{O}=\mathbb{O}_{\lambda},\ \lambda\in Q_{(k-1)(\alpha_1+2\alpha_2)} \\
                                                               (v-v^{-1}), & \mathbb{O}\in \mathbb{O}^{\leqslant}_{t^{(k-1)(\alpha_1+2\alpha_2)}s_1}\cup\mathbb{O}'^{\leqslant}_{t^{(k-1)(\alpha_1+2\alpha_2)}s_2}\\
                                                               1, & \mathbb{O}=\mathbb{O}_2\\
                                                               0, & otherwise.
                                                             \end{array}
                                                           \right.\]
\end{prop}
\begin{prop}\label{cpb1221}
For $k\in\mathbb{N}_+$, we have $$t^{(2k-1)\alpha_1+k\alpha_2}s_1s_2\tilde{\thicksim}t^{(2k-1)\alpha_1+k\alpha_2}s_2s_1,$$thus $T_{t^{(2k-1)\alpha_1+k\alpha_2}s_1s_2}\equiv T_{t^{(2k-1)\alpha_1+k\alpha_2}s_2s_1}$. If $\tilde{w}=t^{(2k-1)\alpha_1+k\alpha_2}s_1s_2$ or $t^{(2k-1)\alpha_1+k\alpha_2}s_2s_1$ with $k=1$, then $\tilde{w}$ is already of minimal length in $\mathbb{O}_2$. So $f_{\tilde{w},\mathbb{O}_2}=1$ and 0 for other $\mathbb{O}$. When $k\geqslant2$, then \[f_{\tilde{w},\mathbb{O}}=\left\{
                                                             \begin{array}{ll}
                                                               (v-v^{-1})^2, & \mathbb{O}=\mathbb{O}_{\lambda},\ \lambda\in Q_{(2k-2)\alpha_1+(k-1)\alpha_2} \\
                                                               (v-v^{-1}), & \mathbb{O}\in \mathbb{O}^{\leqslant}_{t^{(k-1)(2\alpha_1+\alpha_2)}s_1}\cup\mathbb{O}'^{\leqslant}_{t^{(k-1)(2\alpha_1+\alpha_2)}s_2}\\
                                                               1, & \mathbb{O}=\mathbb{O}_2\\
                                                               0, & otherwise.
                                                             \end{array}
                                                           \right.\]
\end{prop}
The proof of \ref{cpb1221} is quite similar to that of \ref{cpa1221}, and thus we just give a proof of \ref{cpa1221}.
\begin{proof}[Proof of Proposition \ref{cpa1221}]
Since $t^{k\alpha_1+(2k-1)\alpha_2}s_2s_1=\tau^{-1}s_0t^{k\alpha_1+(2k-1)\alpha_2}s_1s_2s_0\tau$, $$t^{k\alpha_1+(2k-1)\alpha_2}s_1s_2\tilde{\thicksim}t^{k\alpha_1+(2k-1)\alpha_2}s_2s_1,$$ and hence $T_{t^{k\alpha_1+(2k-1)\alpha_2}s_1s_2}\equiv T_{t^{k\alpha_1+(2k-1)\alpha_2}s_2s_1}$. When $k=1$, the statement is obvious. It is sufficient to consider $\tilde{w}=t^{k\alpha_1+(2k-1)\alpha_2}s_2s_1$ where $k\geqslant2$. Along with Proposition \ref{cpa21} (2),
\begin{align*}T_{t^{k\alpha_1+(2k-1)\alpha_2}s_2s_1}\equiv&(v-v^{-1})T_{t^{(k-1)(\alpha_1+2\alpha_2)}s_1}+T_{t^{(k-1)(\alpha_1+2\alpha_2)}s_1s_2}\\
\cdots&\cdots\\
\equiv&(v-v^{-1})\sum_{i=1}^{k-1}(T_{t^{i\alpha_1+(2i-1)\alpha_2}s_1}+T_{t^{i(\alpha_1+2\alpha_2)}s_1})\\
&+(v-v^{-1})\sum_{i=1}^{k-1}T_{t^{i(\alpha_1+2\alpha_2)}s_2}+T_{\mathbb{O}_2}\\
\cdots&\cdots\\
\equiv&T_{\mathbb{O}_2}+(v-v^{-1})^2\sum_{\lambda\in Q_{(k-1)\alpha_1+(2k-2)\alpha_2}}T_{\mathbb{O}_{\lambda}}\\
&+(v-v^{-1})\sum_{\mathbb{O}\in\mathbb{O}^{\leqslant}_{t^{(k-1)(\alpha_1+2\alpha_2)}s_1}\cup\mathbb{O}'^{\leqslant}_{t^{(k-1)(\alpha_1+2\alpha_2)}s_2}}T_{\mathbb{O}}.
\end{align*}
The proposition is proved.
\end{proof}
For $\lambda\in Q\cap P_+$, we set $$D_{\lambda}=\{\lambda'\in Q\cap P_+\mid \lambda'<\lambda-\alpha_1,\ \lambda'\neq \lambda-i\alpha_1\ (i\in\mathbb{N}_+)\}\cap Q_{sh},$$ $$D'_{\lambda}=\{\lambda'\in Q\cap P_+\mid \lambda'<\lambda-\alpha_2,\ \lambda'\neq \lambda-i\alpha_2\ (i\in\mathbb{N}_+)\}\cap Q_{sh}.$$ And we set $$E_{\lambda}=\{\lambda'\in Q\cap P_+\mid \lambda'=\lambda-i\alpha_1\ (i\in\mathbb{N}_+)\}\cap Q_{sh},$$ $$E'_{\lambda}=\{\lambda'\in Q\cap P_+\mid \lambda'=\lambda-i\alpha_2\ (i\in\mathbb{N}_+)\}\cap Q_{sh}.$$
\begin{prop}\label{cpalambda1221}
(1) If $\tilde{w}=t^{\lambda}s_1s_2$ with $\lambda=m\alpha_1+n\alpha_2(m,\ n\in\mathbb{Z},\ 2\leqslant m\leqslant n<2m-1)$, then \[f_{\tilde{w},\mathbb{O}}=\left\{
                                                             \begin{array}{ll}
                                                               (v-v^{-1})^2, & \mathbb{O}=\mathbb{O}_{\lambda'},\ \lambda'\in D_{\lambda} \\
                                                               (v-v^{-1}), & \mathbb{O}\in \mathbb{O}^{\leqslant}_{t^{\lambda-\alpha_1-\alpha_2}s_1}\cup \mathbb{O}'^{\leqslant}_{t^{\lambda-\alpha_1}s_2}\\
                                                               1, & \mathbb{O}=\mathbb{O}_2\\
                                                               0, & otherwise;
                                                             \end{array}
                                                           \right.\]
(2) If $\tilde{w}=t^{\lambda}s_2s_1$ with $\lambda=m\alpha_1+n\alpha_2(m,\ n\in\mathbb{Z},\ 2\leqslant m<n<2m-1)$, then \[f_{\tilde{w},\mathbb{O}}=\left\{
                                                             \begin{array}{ll}
                                                               (v-v^{-1})^2, & \mathbb{O}=\mathbb{O}_{\lambda'}\ with\ \lambda'\in D'_{\lambda} \\
                                                               (v-v^{-1}), & \mathbb{O}\in \mathbb{O}^{\leqslant}_{t^{\lambda-\alpha_2}s_1}\cup \mathbb{O}'^{\leqslant}_{t^{\lambda-\alpha_1-\alpha_2}s_2}\\
                                                               1, & \mathbb{O}=\mathbb{O}_2\\
                                                               0, & otherwise.
                                                             \end{array}
                                                           \right.\]
\end{prop}
\begin{prop}\label{cpblambda1221}
(1) If $\tilde{w}=t^{\lambda}s_2s_1$ with $\lambda=m\alpha_1+n\alpha_2(m,\ n\in\mathbb{Z},\ 2\leqslant n\leqslant m<2n-1)$, then \[f_{\tilde{w},\mathbb{O}}=\left\{
                                                             \begin{array}{ll}
                                                               (v-v^{-1})^2, & \mathbb{O}=\mathbb{O}_{\lambda'},\ \lambda'\in D'_{\lambda} \\
                                                               (v-v^{-1}), & \mathbb{O}\in \mathbb{O}^{\leqslant}_{t^{\lambda-\alpha_2}s_1}\cup \mathbb{O}'^{\leqslant}_{t^{\lambda-\alpha_1-\alpha_2}s_2}\\
                                                               1, & \mathbb{O}=\mathbb{O}_2\\
                                                               0, & otherwise;
                                                             \end{array}
                                                           \right.\]
(2) If $\tilde{w}=t^{\lambda}s_1s_2$ with $\lambda=m\alpha_1+n\alpha_2(m,\ n\in\mathbb{Z},\ 2\leqslant n<m<2n-1)$, then \[f_{\tilde{w},\mathbb{O}}=\left\{
                                                             \begin{array}{ll}
                                                               (v-v^{-1})^2, & \mathbb{O}=\mathbb{O}_{\lambda'},\ \lambda'\in D_{\lambda} \\
                                                               (v-v^{-1}), & \mathbb{O}\in \mathbb{O}^{\leqslant}_{t^{\lambda-\alpha_1-\alpha_2}s_1}\cup \mathbb{O}'^{\leqslant}_{t^{\lambda-\alpha_1}s_2}\\
                                                               1, & \mathbb{O}=\mathbb{O}_2\\
                                                               0, & otherwise.
                                                             \end{array}
                                                           \right.\]
\end{prop}
To avoid of repetition, we just prove \ref{cpalambda1221}.
\begin{proof}[Proof of Proposition \ref{cpalambda1221}]
For (1), let $\lambda=m\alpha_1+n\alpha_2$
\begin{align*}T_{t^{\lambda}s_1s_2}\equiv&(v-v^{-1})T_{s_0t^{\lambda}s_1s_2}+T_{s_0t^{\lambda}s_1s_2s_0}\\
=&(v-v^{-1})T_{t^{(1-n)\alpha_1+(1-m)\alpha_2}s_1}+T_{t^{(1-n)\alpha_1+(2-m)\alpha_2}s_2s_1}\\
\equiv&(v-v^{-1})T_{\tau^{-1}\cdot t^{(1-n)\alpha_1+(1-m)\alpha_2}s_1}+T_{\tau^{-1}\cdot t^{(1-n)\alpha_1+(2-m)\alpha_2}s_2s_1}\\
=&(v-v^{-1})T_{t^{(n-m+1)\alpha_1+n\alpha_2}s_1s_2s_1}+T_{t^{(n-m+1)\alpha_1+n\alpha_2}s_2s_1}\\
\equiv&(v-v^{-1})T_{s_1t^{(n-m+1)\alpha_1+n\alpha_2}s_1s_2s_1s_1}+(v-v^{-1})T_{s_0t^{(n-m+1)\alpha_1+n\alpha_2}s_2s_1}+T_{s_0t^{(n-m+1)\alpha_1+n\alpha_2}s_2s_1s_0}\\
=&(v-v^{-1})T_{t^{(m-1)\alpha_1+n\alpha_2}s_2}+(v-v^{-1})T_{t^{(1-n)\alpha_1+(m-n)\alpha_2}s_2}+T_{t^{(2-n)\alpha_1+(m-n)\alpha_2}s_1s_2}\\
\equiv&(v-v^{-1})T_{t^{(m-1)\alpha_1+n\alpha_2}s_2}+(v-v^{-1})T_{\tau^{-1}\cdot t^{(1-n)\alpha_1+(m-n)\alpha_2}s_2}+T_{\tau^{-1}\cdot t^{(2-n)\alpha_1+(m-n)\alpha_2}s_1s_2}\\
=&(v-v^{-1})(T_{t^{(m-1)\alpha_1+n\alpha_2}s_2}+T_{t^{(m-1)\alpha_1+(n-1)\alpha_2}s_1})+T_{t^{(m-1)\alpha_1+(n-1)\alpha_2}s_1s_2}\\
\cdots&\cdots\ \ \ \ \ \ \ \ \ \ \ \ \ \ \ \ \ \ \ (Inductively)\\
\equiv&(v-v^{-1})[(T_{t^{(m-1)\alpha_1+n\alpha_2}s_2}+T_{t^{(m-1)\alpha_1+(n-1)\alpha_2}s_1})+(T_{t^{(m-2)\alpha_1+(n-1)\alpha_2}s_2}+T_{t^{(m-2)\alpha_1+(n-2)\alpha_2}s_1})+\cdots\\
&\cdots+(T_{t^{(n-m+1)\alpha_1+2(n-m+1)\alpha_2}s_2}+T_{t^{(n-m+1)\alpha_1+(2(n-m+1)-1)\alpha_2}s_1})]+T_{t^{(n-m+1)\alpha_1+(2(n-m+1)-1)\alpha_2}s_1s_2}\\
\cdots&\cdots\ \ \ \ \ \ \ \ \ \ \ \ \ \ \ \ \ \ \ (\ref{cpa1221})\\
\equiv&T_{\mathbb{O}_2}+(v-v^{-1})^2\sum_{\lambda'\in D_{\lambda}}T_{\mathbb{O}_{\lambda'}}+(v-v^{-1})\sum_{\mathbb{O}\in\mathbb{O}^{\leqslant}_{t^{\lambda-\alpha_1-\alpha_2}s_1}\cup\mathbb{O}'^{\leqslant}_{t^{\lambda-\alpha_1}s_2}}T_{\mathbb{O}}.
\end{align*}
For (2) we combine (1) to obtain
\begin{align*}
T_{t^{\lambda}s_2s_1}\equiv&(v-v^{-1})T_{t^{\lambda-\alpha_2}s_1}+T_{t^{\lambda-\alpha_2}s_1s_2}\\
\cdots&\cdots\\
\equiv&T_{\mathbb{O}_2}+(v-v^{-1})^2\sum_{\lambda'\in D'_{\lambda}}T_{\mathbb{O}_{\lambda'}}+(v-v^{-1})\sum_{\mathbb{O}\in\mathbb{O}^{\leqslant}_{t^{\lambda-\alpha_2}s_1}\cup\mathbb{O}'^{\leqslant}_{t^{\lambda-\alpha_1-\alpha_2}s_2}}T_{\mathbb{O}}.
\end{align*}
\end{proof}
\begin{cor}\label{cpaclambda}
(1) If $\tilde{w}=t^{(2k+1)\alpha_1+k\alpha_2}s_1s_2$, $k\in\mathbb{N}$, then \[f_{\tilde{w},\mathbb{O}}=\left\{
                                                             \begin{array}{ll}
                                                               (v-v^{-1})^2, & \mathbb{O}=\mathbb{O}_{\lambda},\ \lambda\in Q_{2k\alpha_1+k\alpha_2} \\
                                                               (v-v^{-1}), & \mathbb{O}\in \mathbb{O}^{\leqslant}_{t^{2k\alpha_1+k\alpha_2}s_1}\cup \mathbb{O}'^{\leqslant}_{t^{2k\alpha_1+k\alpha_2}s_2}\\
                                                               1, & \mathbb{O}=\mathbb{O}_2\\
                                                               0, & otherwise.
                                                             \end{array}
                                                           \right.\]
(2) If $\tilde{w}=t^{(2k+1)\alpha_1+k\alpha_2}s_2s_1$, $k\in\mathbb{N}$, then \[f_{\tilde{w},\mathbb{O}}=\left\{
                                                             \begin{array}{ll}
                                                               (v-v^{-1})^2, & \mathbb{O}=\mathbb{O}_{\lambda},\ \lambda\in Q_{2k\alpha_1+k\alpha_2}\cup E_{(2k+1)\alpha_1+(k+1)\alpha_2}\\
                                                               (v-v^{-1}), & \mathbb{O}\in \mathbb{O}^{\leqslant}_{t^{(2k+1)\alpha_1+(k+1)\alpha_2}s_1}\cup \mathbb{O}'^{\leqslant}_{t^{2k\alpha_1+k\alpha_2}s_2}\\
                                                               1, & \mathbb{O}=\mathbb{O}_2\\
                                                               0, & otherwise.
                                                             \end{array}
                                                           \right.\]
(3) If $\tilde{w}=t^{m\alpha_1+n\alpha_2}s_1s_2$, $m,\ n\in\mathbb{N}$ and $m-2n>1$, then \[f_{\tilde{w},\mathbb{O}}=\left\{
                                                             \begin{array}{ll}
                                                               (v-v^{-1})^2, & \mathbb{O}=\mathbb{O}_{\lambda},\ \lambda\in D'_{m\alpha_1+(m-n)\alpha_2}\\
                                                               (v-v^{-1}), & \mathbb{O}\in \mathbb{O}^{\leqslant}_{t^{m\alpha_1+(m-n-1)\alpha_2}s_1}\cup \mathbb{O}'^{\leqslant}_{t^{(m-1)\alpha_1+(m-n-1)\alpha_2}s_2}\\
                                                               1, & \mathbb{O}=\mathbb{O}_2\\
                                                               0, & otherwise.
                                                             \end{array}
                                                           \right.\]
(4) If $\tilde{w}=t^{m\alpha_1+n\alpha_2}s_2s_1$, $m,\ n\in\mathbb{N}$ and $m-2n>1$, then \[f_{\tilde{w},\mathbb{O}}=\left\{
                                                             \begin{array}{ll}
                                                               (v-v^{-1})^2, & \mathbb{O}=\mathbb{O}_{\lambda},\ \lambda\in D_{m\alpha_1+(m-n)\alpha_2}\cup E_{m\alpha_1+(m-n)\alpha_2}\\
                                                               (v-v^{-1}), & \mathbb{O}\in \mathbb{O}^{\leqslant}_{t^{m\alpha_1+(m-n)\alpha_2}s_1}\cup \mathbb{O}'^{\leqslant}_{t^{(m-1)\alpha_1+(m-n)\alpha_2}s_2}\\
                                                               1, & \mathbb{O}=\mathbb{O}_2\\
                                                               0, & otherwise.
                                                             \end{array}
                                                           \right.\]
(5) If $\tilde{w}=t^{m\alpha_1-n\alpha_2}s_1s_2$, $m,\ n\in\mathbb{N}_+$ and $m-n>1$, then \[f_{\tilde{w},\mathbb{O}}=\left\{
                                                             \begin{array}{ll}
                                                               (v-v^{-1})^2, & \mathbb{O}=\mathbb{O}_{\lambda},\ \lambda\in D'_{m\alpha_1+(m+n)\alpha_2}\\
                                                               (v-v^{-1}), & \mathbb{O}\in \mathbb{O}^{\leqslant}_{t^{m\alpha_1+(m+n-1)\alpha_2}s_1}\cup\mathbb{O}'^{\leqslant}_{t^{(m-1)\alpha_1+(m+n-1)\alpha_2}s_2}\\
                                                               1, & \mathbb{O}=\mathbb{O}_2\\
                                                               0, & otherwise.
                                                             \end{array}
                                                           \right.\]
(6) If $\tilde{w}=t^{m\alpha_1-n\alpha_2}s_2s_1$, $m,\ n\in\mathbb{N}$ and $m-n>1$, then \[f_{\tilde{w},\mathbb{O}}=\left\{
                                                             \begin{array}{ll}
                                                               (v-v^{-1})^2, & \mathbb{O}=\mathbb{O}_{\lambda},\ \lambda\in D_{m\alpha_1+(m+n)\alpha_2}\cup E_{m\alpha_1+(m+n)\alpha_2}\\
                                                               (v-v^{-1}), & \mathbb{O}\in \mathbb{O}^{\leqslant}_{t^{m\alpha_1+(m+n)\alpha_2}s_1}\cup\mathbb{O}'^{\leqslant}_{t^{(m-1)\alpha_1+(m+n)\alpha_2}s_2}\\
                                                               1, & \mathbb{O}=\mathbb{O}_2\\
                                                               0, & otherwise.
                                                             \end{array}
                                                           \right.\]
\end{cor}
\begin{proof}
For any $\tilde{w}$ which lies in the Corollary, if we take $\tilde{w}'=s_2\tilde{w}s_2$, then either $T_{\tilde{w}}\equiv (v-v^{-1})T_{s_2\tilde{w}}+T_{\tilde{w}'}$ or $T_{\tilde{w}}\equiv T_{\tilde{w}'}$. At this time, $\tilde{w}'$ appears in the previous propositions. Thus combining the previous propositions, we obtain the Corollary.
\end{proof}
\begin{prop}\label{cpo1}
(1) If $\tilde{w}\in\mathbb{O}_1$ and $\ell(\tilde{w})=4k-1(k\in\mathbb{N}_+)$. Then\[f_{\tilde{w},\mathbb{O}}=\left\{
                                                             \begin{array}{ll}
                                                               (k-i)(v-v^{-1})^3+(v-v^{-1}), & \mathbb{O}=\mathbb{O}_{i(\alpha_1+\alpha_2)}(1\leqslant i\leqslant k-1)\\
                                                               (k-i)(v-v^{-1})^3, & \mathbb{O}=\mathbb{O}_{\lambda},\ \lambda\in E_{i(\alpha_1+\alpha_2)}\cup E'_{i(\alpha_1+\alpha_2)}(3\leqslant i\leqslant k-1)\\
                                                               (k-i)(v-v^{-1})^2, & \mathbb{O}=\mathbb{C}_i\ or\ \mathbb{C}'_i(1\leqslant i\leqslant k-1)\\
                                                               k(v-v^{-1}), & \mathbb{O}=\mathbb{O}_2\\
                                                               1, & \mathbb{O}=\mathbb{O}_1\\
                                                               0, & otherwise.
                                                             \end{array}
                                                           \right.\]
(2) If $\tilde{w}\in\mathbb{O}_1$ and $\ell(\tilde{w})=4k-3(k\in\mathbb{N}_+)$. Then\[f_{\tilde{w},\mathbb{O}}=\left\{
                                                             \begin{array}{ll}
                                                               (v-v^{-1}), & \mathbb{O}=\mathbb{O}_{(k-1)(\alpha_1+\alpha_2)}\\
                                                               (k-1-i)(v-v^{-1})^3+(v-v^{-1}), & \mathbb{O}=\mathbb{O}_{i(\alpha_1+\alpha_2)}(1\leqslant i\leqslant k-2)\\
                                                               (k-1-i)(v-v^{-1})^3, & \mathbb{O}=\mathbb{O}_{\lambda},\ \lambda\in E_{i(\alpha_1+\alpha_2)}\cup E'_{i(\alpha_1+\alpha_2)}(3\leqslant i\leqslant k-2)\\
                                                               (k-1-i)(v-v^{-1})^2, & \mathbb{O}=\mathbb{C}_i\ or\ \mathbb{C}'_i(1\leqslant i\leqslant k-2)\\
                                                               (k-1)(v-v^{-1}), & \mathbb{O}=\mathbb{O}_2\\
                                                               1, & \mathbb{O}=\mathbb{O}_1\\
                                                               0, & otherwise.
                                                             \end{array}
                                                           \right.\]
\end{prop}
\begin{proof}
We prove (1) and the proof of (2) is similar. Since any element $\tilde{w}\in\mathbb{O}_1$ with $\ell(\tilde{w})=4k-1$ is $\tilde{\thicksim}$ to an element in $\{\tilde{w}=t^{k\alpha_1}s_1\mid k\in\mathbb{N}_+\}$, it is enough for us to consider $\tilde{w}=t^{k\alpha_1}s_1\ (k\in\mathbb{N}_+)$.
\begin{align*}
T_{\tilde{w}}\equiv&(v-v^{-1})T_{s_2\tilde{w}}+T_{s_2\tilde{w}s_2}\\
=&(v-v^{-1})T_{t^{k(\alpha_1+\alpha_2)}s_2s_1}+T_{t^{k(\alpha_1+\alpha_2)}s_1s_2s_1}\\
\equiv&(v-v^{-1})T_{t^{k(\alpha_1+\alpha_2)}s_2s_1}+(v-v^{-1})T_{s_0t^{k(\alpha_1+\alpha_2)}s_1s_2s_1}+T_{s_0t^{k(\alpha_1+\alpha_2)}s_1s_2s_1s_0}\\
=&(v-v^{-1})T_{t^{k(\alpha_1+\alpha_2)}s_2s_1}+(v-v^{-1})T_{t^{(1-k)(\alpha_1+\alpha_2)}}+T_{t^{(2-k)(\alpha_1+\alpha_2)}s_1s_2s_1}\\
\equiv&(v-v^{-1})T_{t^{k(\alpha_1+\alpha_2)}s_2s_1}+(v-v^{-1})T_{\tau\cdot t^{(1-k)(\alpha_1+\alpha_2)}}+T_{\tau\cdot t^{(2-k)(\alpha_1+\alpha_2)}s_1s_2s_1}\\
=&(v-v^{-1})T_{t^{k(\alpha_1+\alpha_2)}s_2s_1}+(v-v^{-1})T_{t^{(k-1)\alpha_1}}+T_{t^{(k-1)\alpha_1}s_1}\\
\equiv&(v-v^{-1})T_{t^{k(\alpha_1+\alpha_2)}s_2s_1}+(v-v^{-1})T_{s_2t^{(k-1)\alpha_1}s_2}+T_{t^{(k-1)\alpha_1}s_1}\\
=&(v-v^{-1})T_{t^{k(\alpha_1+\alpha_2)}s_2s_1}+(v-v^{-1})T_{\mathbb{O}_{(k-1)(\alpha_1+\alpha_2)}}+T_{t^{(k-1)\alpha_1}s_1}\\
\cdots&\cdots\\
\equiv&(v-v^{-1})\sum_{i=1}^kT_{t^{i(\alpha_1+\alpha_2)}s_2s_1}+(v-v^{-1})\sum_{i=1}^{k-1}T_{\mathbb{O}_{i(\alpha_1+\alpha_2)}}+T_{\mathbb{O}_1}.
\end{align*}
Along with Proposition \ref{cpblambda1221} (1), we get
\begin{align*}
T_{t^{k\alpha_1}s_1}\equiv&\sum_{i=1}^{k-1}[(k-i)(v-v^{-1})^3+(v-v^{-1})]T_{\mathbb{O}_{i(\alpha_1+\alpha_2)}}\\
+&\sum_{i=3}^{k-1}\sum_{\lambda\in E_{i(\alpha_1+\alpha_2)}\cup E'_{i(\alpha_1+\alpha_2)}}(k-i)(v-v^{-1})^3T_{\mathbb{O}_{\lambda}}\\
+&\sum_{i=1}^{k-1}(k-i)(v-v^{-1})^2(T_{\mathbb{C}_i}+T_{\mathbb{C}'_i})+k(v-v^{-1})T_{\mathbb{O}_2}+T_{\mathbb{O}_1}.
\end{align*}(1) is proved.
\end{proof}
\begin{prop}\label{cpci}
(1) If $\tilde{w}\in \mathbb{C}_i(i\in\mathbb{N}_+)$ and $\ell(\mathbb{C}_i)\leqslant\ell(\tilde{w})\leqslant6i+1$. Then \[f_{\tilde{w},\mathbb{O}}=\left\{
                                                             \begin{array}{ll}
                                                               (v-v^{-1}), & \mathbb{O}=\mathbb{O}_{j\alpha_1+i\alpha_2}\ (\lfloor\frac{i}{2}\rfloor+1\leqslant j\leqslant\frac{\ell(\tilde{w})-1}{2}-i)\\
                                                               1, & \mathbb{O}=\mathbb{C}_i\\
                                                               0, & otherwise.
                                                             \end{array}
                                                           \right.\]
(2) If $\tilde{w}\in \mathbb{C}_i(i\in\mathbb{N}_+)$ and $\ell(\tilde{w})=6i-1+4k(k\geqslant1)$. Then \[f_{\tilde{w},\mathbb{O}}=\left\{
                                                             \begin{array}{ll}
                                                               (v-v^{-1}), & \mathbb{O}=\mathbb{O}_{2i\alpha_1+i\alpha_2}\\
                                                               k(v-v^{-1}), & \mathbb{O}=\mathbb{O}_2\\
                                                               k(v-v^{-1})^2, & \mathbb{O}\in \mathbb{O}^{\leqslant}_{t^{2i\alpha_1+i\alpha_2}s_1}\cup\mathbb{O}'^{\leqslant}_{t^{2i\alpha_1+i\alpha_2}s_2}\backslash \mathbb{C}_i\\
                                                               k(v-v^{-1})^3, & \mathbb{O}=\mathbb{O}_{\lambda},\ \lambda\in Q_{2i\alpha_1+i\alpha_2}\backslash E_{2i\alpha_1+i\alpha_2}\\
                                                               k(v-v^{-1})^3+(v-v^{-1}), & \mathbb{O}=\mathbb{O}_{\lambda},\ \lambda\in E_{2i\alpha_1+i\alpha_2}\\
                                                               (k-j)(v-v^{-1})^2, & \mathbb{O}\in \mathbb{C}(t^{(2i+j+1)\alpha_1+(i+j)\alpha_2}s_1)\cup \mathbb{C}'(t^{(2i+j)\alpha_1+(i+j)\alpha_2}s_2)\\
                                                               & (1\leqslant j\leqslant k-1)\\
                                                               (k-j)(v-v^{-1})^3, & \mathbb{O}=\mathbb{O}_{\lambda}\ \  \lambda\in E_{(2i+j)\alpha_1+(i+j)\alpha_2}\cup E'_{(2i+j)\alpha_1+(i+j)\alpha_2}\\
                                                               &(1\leqslant j\leqslant k-1)\\
                                                               (k-j)(v-v^{-1})^3+(v-v^{-1}), & \mathbb{O}=\mathbb{O}_{(2i+j)\alpha_1+(i+j)\alpha_2}\ \ (1\leqslant j\leqslant k-1)\\
                                                               k(v-v^{-1})^2+1, & \mathbb{O}=\mathbb{C}_i\\
                                                               0, & otherwise.
                                                             \end{array}
                                                           \right.\]
(3) If $\tilde{w}\in \mathbb{C}_i(i\in\mathbb{N}_+)$ and $\ell(\tilde{w})=6i+1+4k(k\geqslant1)$. Then \[f_{\tilde{w},\mathbb{O}}=\left\{
                                                             \begin{array}{ll}
                                                               (v-v^{-1}), & \mathbb{O}=\mathbb{O}_{2i\alpha_1+i\alpha_2}\\
                                                               k(v-v^{-1}), & \mathbb{O}=\mathbb{O}_2\\
                                                               k(v-v^{-1})^2, & \mathbb{O}\in \mathbb{O}^{\leqslant}_{t^{2i\alpha_1+i\alpha_2}s_1}\cup\mathbb{O}'^{\leqslant}_{t^{2i\alpha_1+i\alpha_2}s_2}\backslash \mathbb{C}_i\\
                                                               k(v-v^{-1})^3, & \mathbb{O}=\mathbb{O}_{\lambda},\ \lambda\in Q_{2i\alpha_1+i\alpha_2}\backslash E_{2i\alpha_1+i\alpha_2}\\
                                                               k(v-v^{-1})^3+(v-v^{-1}), & \mathbb{O}=\mathbb{O}_{\lambda},\ \lambda\in E_{2i\alpha_1+i\alpha_2}\\
                                                               (k-j)(v-v^{-1})^2, & \mathbb{O}\in \mathbb{C}(t^{(2i+j+1)\alpha_1+(i+j)\alpha_2}s_1)\cup \mathbb{C}'(t^{(2i+j)\alpha_1+(i+j)\alpha_2}s_2)\\
                                                               & (1\leqslant j\leqslant k-1)\\
                                                               (k-j)(v-v^{-1})^3, & \mathbb{O}=\mathbb{O}_{\lambda}\ \  \lambda\in E_{(2i+j)\alpha_1+(i+j)\alpha_2}\cup E'_{(2i+j)\alpha_1+(i+j)\alpha_2}\\
                                                               &(1\leqslant j\leqslant k-1)\\
                                                               (k-j)(v-v^{-1})^3+(v-v^{-1}), & \mathbb{O}=\mathbb{O}_{(2i+j)\alpha_1+(i+j)\alpha_2}\ \ (1\leqslant j\leqslant k-1)\\
                                                               k(v-v^{-1})^2+1, & \mathbb{O}=\mathbb{C}_i\\
                                                               (v-v^{-1}), & \mathbb{O}=\mathbb{O}_{(2i+k)\alpha_1+(i+k)\alpha_2}\\
                                                               0, & otherwise.
                                                             \end{array}
                                                           \right.\]
\end{prop}
\begin{prop}\label{cpc'i}
(1) If $\tilde{w}\in \mathbb{C}'_i(i\in\mathbb{N}_+)$ and $\ell(\mathbb{C}'_i)\leqslant\ell(\tilde{w})\leqslant6i+1$. Then \[f_{\tilde{w},\mathbb{O}}=\left\{
                                                             \begin{array}{ll}
                                                               (v-v^{-1}), & \mathbb{O}=\mathbb{O}_{i\alpha_1+(\lfloor\frac{i}{2}\rfloor+j)\alpha_2}\ (1\leqslant j\leqslant\frac{\ell(\tilde{w})-\ell(\mathbb{C}'_i)}{2})\\
                                                               1, & \mathbb{O}=\mathbb{C}'_i\\
                                                               0, & otherwise.
                                                             \end{array}
                                                           \right.\]
(2) If $\tilde{w}\in \mathbb{C}'_i(i\in\mathbb{N}_+)$ and $\ell(\tilde{w})=6i-1+4k(k\geqslant1)$. Then \[f_{\tilde{w},\mathbb{O}}=\left\{
                                                             \begin{array}{ll}
                                                               (v-v^{-1}), & \mathbb{O}=\mathbb{O}_{i\alpha_1+2i\alpha_2}\\
                                                               k(v-v^{-1}), & \mathbb{O}=\mathbb{O}_2\\
                                                               k(v-v^{-1})^2, & \mathbb{O}\in \mathbb{O}^{\leqslant}_{t^{i(\alpha_1+2\alpha_2)}s_1}\cup\mathbb{O}'^{\leqslant}_{t^{i(\alpha_1+2\alpha_2)}s_2}\backslash \mathbb{C}'_i\\
                                                               k(v-v^{-1})^3, & \mathbb{O}=\mathbb{O}_{\lambda},\ \lambda\in Q_{i\alpha_1+2i\alpha_2}\backslash E'_{i\alpha_1+2i\alpha_2}\\
                                                               k(v-v^{-1})^3+(v-v^{-1}), & \mathbb{O}=\mathbb{O}_{\lambda},\ \lambda\in E'_{i\alpha_1+2i\alpha_2}\\
                                                               (k-j)(v-v^{-1})^2, & \mathbb{O}\in \mathbb{C}(t^{(i+j)\alpha_1+(2i+j)\alpha_2}s_1)\cup \mathbb{C}'(t^{(i+j)\alpha_1+(2i+j+1)\alpha_2}s_2)\\
                                                               & (1\leqslant j\leqslant k-1)\\
                                                               (k-j)(v-v^{-1})^3, & \mathbb{O}=\mathbb{O}_{\lambda},\ \lambda\in E_{(i+j)\alpha_1+(2i+j)\alpha_2}\cup E'_{(i+j)\alpha_1+(2i+j)\alpha_2}\\
                                                               & (1\leqslant j\leqslant k-1)\\
                                                               (k-j)(v-v^{-1})^3+(v-v^{-1}), & \mathbb{O}=\mathbb{O}_{(i+j)\alpha_1+(2i+j)\alpha_2}\ (1\leqslant j\leqslant k-1)\\
                                                               k(v-v^{-1})^2+1, & \mathbb{O}=\mathbb{C}'_i\\
                                                               0, & otherwise.
                                                             \end{array}
                                                           \right.\]
(3) If $\tilde{w}\in \mathbb{C}'_i(i\in\mathbb{N}_+)$ and $\ell(\tilde{w})=6i+1+4k(k\geqslant1)$. Then \[f_{\tilde{w},\mathbb{O}}=\left\{
                                                             \begin{array}{ll}
                                                               (v-v^{-1}), & \mathbb{O}=\mathbb{O}_{i\alpha_1+2i\alpha_2}\\
                                                               k(v-v^{-1}), & \mathbb{O}=\mathbb{O}_2\\
                                                               k(v-v^{-1})^2, & \mathbb{O}\in \mathbb{O}^{\leqslant}_{t^{i(\alpha_1+2\alpha_2)}s_1}\cup\mathbb{O}'^{\leqslant}_{t^{i(\alpha_1+2\alpha_2)}s_2}\backslash \mathbb{C}'_i\\
                                                               k(v-v^{-1})^3, & \mathbb{O}=\mathbb{O}_{\lambda},\ \lambda\in Q_{i\alpha_1+2i\alpha_2}\backslash E'_{i\alpha_1+2i\alpha_2}\\
                                                               k(v-v^{-1})^3+(v-v^{-1}), & \mathbb{O}=\mathbb{O}_{\lambda},\ \lambda\in E'_{i\alpha_1+2i\alpha_2}\\
                                                               (k-j)(v-v^{-1})^2, & \mathbb{O}\in \mathbb{C}(t^{(i+j)\alpha_1+(2i+j)\alpha_2}s_1)\cup \mathbb{C}'(t^{(i+j)\alpha_1+(2i+j+1)\alpha_2}s_2)\\
                                                               & (1\leqslant j\leqslant k-1)\\
                                                               (k-j)(v-v^{-1})^3, & \mathbb{O}=\mathbb{O}_{\lambda},\ \lambda\in E_{(i+j)\alpha_1+(2i+j)\alpha_2}\cup E'_{(i+j)\alpha_1+(2i+j)\alpha_2}\\
                                                               & (1\leqslant j\leqslant k-1)\\
                                                               (k-j)(v-v^{-1})^3+(v-v^{-1}), & \mathbb{O}=\mathbb{O}_{(i+j)\alpha_1+(2i+j)\alpha_2}\ (1\leqslant j\leqslant k-1)\\
                                                               k(v-v^{-1})^2+1, & \mathbb{O}=\mathbb{C}'_i\\
                                                               (v-v^{-1}), & \mathbb{O}=\mathbb{O}_{(i+k)\alpha_1+(2i+k)\alpha_2}\\
                                                               0, & otherwise.
                                                             \end{array}
                                                           \right.\]
\end{prop}
Also, the proofs of \ref{cpci} and \ref{cpc'i} are similar. We choose \ref{cpci} to prove and leave the other one.
\begin{proof}[Proof of Proposition \ref{cpci}]
Since (3) is an easy consequence of (2), we prove (1) and (2). Any element $\tilde{w}\in\mathbb{C}_i$ with $i\in\mathbb{N}_+$ is $\tilde{\thicksim}$ to an element in $\{t^{k\alpha_1+(k-i)\alpha_2}s_1s_2s_1,\ t^{k\alpha_1+i\alpha_2}s_1\mid k\geqslant\lfloor\frac{i}{2}\rfloor+1\}\subset\mathbb{C}_i$. For $\lfloor\frac{i}{2}\rfloor+1\leqslant k\leqslant2i$, $t^{k\alpha_1+i\alpha_2}s_1=s_2t^{k\alpha_1+(k-i)\alpha_2}s_1s_2s_1s_2$ and $\ell(t^{k\alpha_1+i\alpha_2}s_1)=\ell(t^{k\alpha_1+(k-i)\alpha_2}s_1s_2s_1)$. It is sufficient for us to consider $\tilde{w}\in\{t^{k\alpha_1+i\alpha_2}s_1\mid k\geqslant\lfloor\frac{i}{2}\rfloor+1\}\sqcup\{t^{k\alpha_1+(k-i)\alpha_2}s_1s_2s_1\mid k\geqslant 2i+1\}$.\\
(1) For $\tilde{w}\in \mathbb{C}_i$ and $\ell(\mathbb{C}_i)\leqslant\ell(\tilde{w})\leqslant6i-1$. We just consider $\tilde{w}=t^{k\alpha_1+i\alpha_2}s_1$ with $\lfloor\frac{i}{2}\rfloor+1\leqslant k\leqslant2i$.
\begin{align*}
T_{\tilde{w}}\equiv& (v-v^{-1})T_{\mathbb{O}_{(k-1)\alpha_1+i\alpha_2}}+T_{t^{(k-1)\alpha_1+i\alpha_2}s_1}\\
&\cdots\\
\equiv&(v-v^{-1})\sum_{j=\lfloor\frac{i}{2}\rfloor+1}^{k-1}T_{\mathbb{O}_{j\alpha_1+i\alpha_2}}+T_{\mathbb{C}_i}.
\end{align*}
If $\ell(\tilde{w})=6i+1$, we can take $\tilde{w}=t^{(2i+1)\alpha_1+(i+1)\alpha_2}s_1s_2s_1$. Then \[T_{t^{(2i+1)\alpha_1+(i+1)\alpha_2}s_1s_2s_1}\equiv(v-v^{-1})\sum_{\lambda\in E_{2i\alpha_1+i\alpha_2}\cup\{2i\alpha_1+i\alpha_2\}}T_{\mathbb{O}_{\lambda}}+T_{\mathbb{C}_i}.\]
(2) We consider the case $\tilde{w}=t^{(2i+k)\alpha_1+i\alpha_2}s_1$, with $k\geqslant1$
\begin{align*}
T_{\tilde{w}}\equiv& (v-v^{-1})T_{t^{(2i+k)\alpha_1+(k+i)\alpha_2}s_2s_1}+T_{t^{(2i+k)\alpha_1+(k+i)\alpha_2}s_1s_2s_1}\\
\equiv& (v-v^{-1})T_{t^{(2i+k)\alpha_1+(k+i)\alpha_2}s_2s_1}+(v-v^{-1})T_{\mathbb{O}_{(2i+k-1)\alpha_1+(k+i-1)\alpha_2}}+T_{t^{(2i+k-1)\alpha_1+i\alpha_2}s_1}\\
\cdots&\cdots\\
\equiv&(v-v^{-1})\sum_{j=1}^{k}T_{t^{(2i+j)\alpha_1+(j+i)\alpha_2}s_2s_1}+(v-v^{-1})\sum_{j=0}^{k-1}T_{\mathbb{O}_{(2i+j)\alpha_1+(j+i)\alpha_2}}\\
+&(v-v^{-1})\sum_{j=\lfloor\frac{i}{2}\rfloor+1}^{2i-1}T_{\mathbb{O}_{j\alpha_1+i\alpha_2}}+T_{\mathbb{C}_i}.
\end{align*}
Together with Proposition \ref{cpblambda1221} (1), then (2) is proved.
\end{proof}
\subsection{Class polynomials for elements in $W_a\tau$}
We set $$\mathbb{I}\tau=\{t^{m\alpha_1+n\alpha_2}w\tau\mid m\in\mathbb{N}_+,\ n\in\mathbb{Z},\ 1-m\leqslant n\leqslant2m-1,\ w\in W\}\sqcup\{\tau,\ s_2\tau\}$$$$\sqcup\{t^{k\alpha_1+2k\alpha_2}\tau,\ t^{k\alpha_1+2k\alpha_2}s_2\tau,\ t^{k\alpha_1+2k\alpha_2}s_2s_1\tau,\ t^{k(\alpha_1-\alpha_2)}\tau,\ t^{k(\alpha_1-\alpha_2)}s_1\tau,\ t^{k(\alpha_1-\alpha_2)}s_2\tau\mid k\in\mathbb{N}_+\}.$$ Since $W_a\tau=\mathbb{I}\tau\cup\tau\cdot(\mathbb{I}\tau)\cup\tau^2\cdot(\mathbb{I}\tau)$ and any element $\tilde{w}\in W_a\tau$ is $\tilde{\thicksim}$ to an element in $\mathbb{I}\tau$, it is sufficient to consider elements in $\mathbb{I}\tau$. Now we are going to classify all $\widetilde{W}$-conjugacy classes in $W_a\tau$.
\begin{lem}\label{lengtau0}
Let $\mathbb{O}_{id,\tau}$ be the set of $\widetilde{W}$-conjugacy class in $W_a\tau$ with minimal length 0. Then\[\mathbb{O}_{id,\tau}=\{t^{m\alpha_1+n\alpha_2}\tau, \ t^{m\alpha_1+n\alpha_2}s_1s_2\tau\mid m,\ n\in\mathbb{Z}\}.\]
\end{lem}
\begin{lem}\label{lenglambdatau}
For any $\lambda\in P_+\bigcap Q$ and $\lambda\neq 2m\alpha_1+m\alpha_2\ (m\in\mathbb{N})$ i.e. $\lambda=m\alpha_1+n\alpha_2$ where $1\leqslant m\leqslant n\leqslant2m$ or $1\leqslant n<m\leqslant2n-1$, we set $\mathbb{O}_{\lambda,\tau}=\widetilde{W}\cdot(t^{\lambda}s_2s_1\tau).$ If $n=2m$ or $m=2n-1$ where $m,\ n\in\mathbb{N}_+$, then$$\mathbb{O}_{\lambda,\tau}=\{t^{m\alpha_1+n\alpha_2}s_2s_1\tau,\ t^{(1-n)\alpha_1+(m+1-n)\alpha_2}s_2s_1\tau,\ t^{(n-m)\alpha_1+(1-m)\alpha_2}s_2s_1\tau\}.$$If $2\leqslant m\leqslant n\leqslant2m-1$ or $1\leqslant n<m<2n-1$, then $$\mathbb{O}_{\lambda,\tau}=\{t^{m\alpha_1+n\alpha_2}s_2s_1\tau,\ t^{(1-n)\alpha_1+(m+1-n)\alpha_2}s_2s_1\tau,\ t^{(n-m)\alpha_1+(1-m)\alpha_2}s_2s_1\tau\}$$ $$\sqcup\{t^{(n-m)\alpha_1+n\alpha_2}s_2s_1\tau,\ t^{(1-n)\alpha_1+(1-m)\alpha_2}s_2s_1\tau,\ t^{m\alpha_1+(m+1-n)\alpha_2}s_2s_1\tau\}$$and $\ell(\mathbb{O}_{\lambda,\tau})=\ell(t^{\lambda}s_2s_1\tau)$.
\end{lem}
\begin{lem}\label{lengtaui}
For $i\in\mathbb{Z}$, let $\mathbb{O}_{i,\tau}=\widetilde{W}\cdot(t^{(\lfloor\frac{i}{2}\rfloor+1)\alpha_1+i\alpha_2}s_1s_2s_1\tau)$. Then $\mathbb{O}_{i,\tau}$ is the set of $\widetilde{W}$-conjugacy class in $W_a\tau$ with \[\ell(\mathbb{O}_{i,\tau})=\ell(t^{(\lfloor\frac{i}{2}\rfloor+1)\alpha_1+i\alpha_2}s_1s_2s_1\tau).\] Moreover, $\mathbb{O}_{i,\tau}=\{t^{k\alpha_1+i\alpha_2}s_1s_2s_1\tau,\ t^{(k+i-1)\alpha_1+k\alpha_2}s_2\tau,\ t^{(1-i)\alpha_1+k\alpha_2}s_1\tau\mid k\in\mathbb{Z}\}$.
\end{lem}
Since the approach of proving the above lemmas \ref{lengtau0}, \ref{lenglambdatau} and \ref{lengtaui} are quite similar, here we give a proof of \ref{lengtaui} as an example.
\begin{proof}[Proof of Lemma \ref{lengtaui}]
(1) We set $$A=\{t^{k\alpha_1+i\alpha_2}s_1s_2s_1\tau,\ t^{(k+i-1)\alpha_1+k\alpha_2}s_2\tau,\ t^{(1-i)\alpha_1+k\alpha_2}s_1\tau\mid k\in\mathbb{Z}\},$$ then $\widetilde{W}\cdot A\subset A$. Since $$A=\{t^{k\alpha_1+i\alpha_2}s_1s_2s_1\tau\mid k\in\mathbb{Z}\}\sqcup\{t^{(1-i)\alpha_1+k\alpha_2}s_1\tau\mid k\in\mathbb{Z}\}\sqcup\{t^{(k+i-1)\alpha_1+k\alpha_2}s_2\tau\mid k\in\mathbb{Z}\}$$$$=\{t^{k\alpha_1+i\alpha_2}s_1s_2s_1\tau\mid k\in\mathbb{Z}\}\sqcup\tau\cdot\{t^{k\alpha_1+i\alpha_2}s_1s_2s_1\tau\mid k\in\mathbb{Z}\}\sqcup\tau^{-1}\cdot \{t^{k\alpha_1+i\alpha_2}s_1s_2s_1\tau\mid k\in\mathbb{Z}\}$$$$=\{t^{(k+i-1)\alpha_1+k\alpha_2}s_2\tau\mid k\in\mathbb{Z}\}\sqcup\tau\cdot\{t^{(k+i-1)\alpha_1+k\alpha_2}s_2\tau\mid k\in\mathbb{Z}\}\sqcup\tau^{-1}\cdot \{t^{(k+i-1)\alpha_1+k\alpha_2}s_2\tau\mid k\in\mathbb{Z}\}$$$$=\{t^{(1-i)\alpha_1+k\alpha_2}s_1\tau\mid k\in\mathbb{Z}\}\sqcup\tau\cdot\{t^{(1-i)\alpha_1+k\alpha_2}s_1\tau\mid k\in\mathbb{Z}\}\sqcup\tau^{-1}\cdot \{t^{(1-i)\alpha_1+k\alpha_2}s_1\tau\mid k\in\mathbb{Z}\},$$$s_0\cdot(t^{k\alpha_1+i\alpha_2}s_1s_2s_1\tau)=t^{(1-i)\alpha_1+(1-k)\alpha_2}s_1\tau,\ s_1\cdot(t^{k\alpha_1+i\alpha_2}s_1s_2s_1\tau)=t^{(i-k)\alpha_1+i\alpha_2}s_1s_2s_1\tau,\ s_2\cdot(t^{k\alpha_1+i\alpha_2}s_1s_2s_1\tau)=t^{(k-1)\alpha_1+(k-i)\alpha_2}s_2\tau$, $s_0\cdot(t^{(k+i-1)\alpha_1+k\alpha_2}s_2\tau)=t^{(1-k)\alpha_1+(2-i-k)\alpha_2}s_2\tau,\ s_1\cdot(t^{(k+i-1)\alpha_1+k\alpha_2}s_2\tau)=t^{(1-i)\alpha_1+k\alpha_2}s_1\tau,\ s_2\cdot(t^{(k+i-1)\alpha_1+k\alpha_2}s_2\tau)=t^{(k+i)\alpha_1+i\alpha_2}s_1s_2s_1\tau$, $s_0\cdot(t^{(1-i)\alpha_1+k\alpha_2}s_1\tau)=t^{(1-k)\alpha_1+i\alpha_2}s_1s_2s_1\tau,\ s_1\cdot(t^{(1-i)\alpha_1+k\alpha_2}s_1\tau)=t^{(k+i-1)\alpha_1+k\alpha_2}s_2\tau,\ s_2\cdot(t^{(1-i)\alpha_1+k\alpha_2}s_1\tau)=t^{(1-i)\alpha_1-(k+i)\alpha_2}s_1\tau$, thus (1) is proved.\\
(2) For any $\tilde{w}\in\widetilde{W}$, by direct calculation we have $$\ell(t^{(\lfloor\frac{i}{2}\rfloor+1)\alpha_1+i\alpha_2}s_1s_2s_1\tau)\leqslant\ell(\tilde{w}t^{(\lfloor\frac{i}{2}\rfloor+1)\alpha_1+i\alpha_2}s_1s_2s_1\tau\tilde{w}^{-1}).$$
For any $\tilde{w}\in A$, $\tilde{w}$ is $\widetilde{W}$-conjugate to $t^{(\lfloor\frac{i}{2}\rfloor+1)\alpha_1+i\alpha_2}s_1s_2s_1\tau$. We use induction on length. If for all $\tilde{w}'\in A$ with $\ell(\tilde{w}')<\ell(\tilde{w})$, then $\tilde{w}'$ is $\widetilde{W}$-conjugate to $t^{(\lfloor\frac{i}{2}\rfloor+1)\alpha_1+i\alpha_2}s_1s_2s_1\tau$. From the proof of (1), it is sufficient for us to consider that $\tilde{w}\in \{t^{k\alpha_1+i\alpha_2}s_1s_2s_1\tau\mid k\in\mathbb{Z}\}$. Now we take $\tilde{w}=t^{k\alpha_1+i\alpha_2}s_1s_2s_1\tau$. (a) For $i>0$ and $i$ is odd. If $k\geqslant2i$, then we take $\tilde{w}_1=s_0\cdot(\tilde{w})$. Then $\ell(\tilde{w}_1)=\ell(\tilde{w})-2$, thus by induction $\tilde{w}$ is $\widetilde{W}$-conjugate to $t^{(\lfloor\frac{i}{2}\rfloor+1)\alpha_1+i\alpha_2}s_1s_2s_1\tau$. If $\lfloor\frac{i}{2}\rfloor+2\leqslant k\leqslant2i-1$, then we take $\tilde{w}_1=s_2s_0\cdot(\tilde{w})$. Then $\ell(\tilde{w}_1)=\ell(\tilde{w})-2$, thus by induction $\tilde{w}$ is $\widetilde{W}$-conjugate to $t^{(\lfloor\frac{i}{2}\rfloor+1)\alpha_1+i\alpha_2}s_1s_2s_1\tau$. If $k=\lfloor\frac{i}{2}\rfloor+1$, then $\tilde{w}$ is already of minimal length. If $k\leqslant\lfloor\frac{i}{2}\rfloor$, then we take $\tilde{w}_1=s_1\cdot(\tilde{w})$. Then $\ell(\tilde{w}_1)=\ell(\tilde{w})-2$, thus by induction $\tilde{w}$ is $\widetilde{W}$-conjugate to $t^{(\lfloor\frac{i}{2}\rfloor+1)\alpha_1+i\alpha_2}s_1s_2s_1\tau$. (b) For $i>0$ and $i$ is even. If $k\geqslant2i$ or $\lfloor\frac{i}{2}\rfloor+2\leqslant k\leqslant2i-1$ or $k\leqslant\lfloor\frac{i}{2}\rfloor-1$ then $\tilde{w}_1$ is as before, and thus $\tilde{w}$ is $\widetilde{W}$-conjugate to $t^{(\lfloor\frac{i}{2}\rfloor+1)\alpha_1+i\alpha_2}s_1s_2s_1\tau$. If $i\leqslant0$, the argument is similar we omit it here.
By (1) and (2), the lemma is proved.
\end{proof}
\begin{thm}
The $\widetilde{W}$-conjugacy classes in $W_a\tau$ are:
$\mathbb{O}_{id,\tau},\ \mathbb{O}_{\lambda,\tau}$, and $\mathbb{O}_{i,\tau}$ with $i\in\mathbb{Z}$, $\lambda\in P_+\bigcap Q$ and $\lambda\neq 2m\alpha_1+m\alpha_2\ (m\in\mathbb{N})$.
\end{thm}
\begin{proof}
Follows directly from Lemmas \ref{lengtau0}, \ref{lenglambdatau} and \ref{lengtaui}.
\end{proof}
\begin{prop}\label{cplambdatau}
For any $\tilde{w}\in\mathbb{O}_{\lambda, \tau}$ where $\lambda\in P_+\bigcap Q$ and $\lambda\neq 2m\alpha_1+m\alpha_2\ (m\in\mathbb{N})$. Then \[f_{\tilde{w},\mathbb{O}}=\left\{
                                                             \begin{array}{ll}
                                                               1, & \mathbb{O}=\mathbb{O}_{\lambda, \tau}.\\
                                                               0, & otherwise.
                                                             \end{array}
                                                           \right.\]
\end{prop}
\begin{proof}
Following the definition of class polynomials, the proposition is obvious.
\end{proof}
For $\lambda=m\alpha_1+n\alpha_2\in P_+\cap Q_{sh}$, we set $$E_{m\alpha_1+n\alpha_2,\tau}=\{\lambda'=k\alpha_1+n\alpha_2\mid\lfloor\frac{n}{2}\rfloor+1\leqslant k\leqslant m-1\},$$ $$E'_{m\alpha_1+n\alpha_2,\tau}=\{\lambda'=m\alpha_1+k\alpha_2\mid\lfloor\frac{m+1}{2}\rfloor+1\leqslant k\leqslant n\}.$$
\begin{prop}
(1) Let  $\lambda=m\alpha_1+n\alpha_2\in P_+\cap Q_{sh}$, and $\tilde{w}=t^{\lambda}s_1s_2s_1\tau$. Then \[f_{\tilde{w},\mathbb{O}}=\left\{
                                                             \begin{array}{ll}
                                                               (v-v^{-1}), & \mathbb{O}=\mathbb{O}_{\lambda', \tau},\ with\ \lambda'\in E_{m\alpha_1+n\alpha_2,\tau}\\
                                                               1, & \mathbb{O}=\mathbb{O}_{n, \tau}\\
                                                               0, & otherwise
                                                             \end{array}
                                                           \right.\]
(2) Let  $\lambda=m\alpha_1+n\alpha_2\in P_+\cap Q_{sh}$ with $n\neq2m$, and $\tilde{w}=t^{\lambda}s_1\tau$. Then \[f_{\tilde{w},\mathbb{O}}=\left\{
                                                             \begin{array}{ll}
                                                               (v-v^{-1}), & \mathbb{O}=\mathbb{O}_{\lambda', \tau},\ with\ \lambda'\in E'_{m\alpha_1+n\alpha_2,\tau}\\
                                                               1, & \mathbb{O}=\mathbb{O}_{1-m, \tau}\\
                                                               0, & otherwise
                                                             \end{array}
                                                           \right.\]
\end{prop}
\begin{proof}
(1): If $\lambda=k\alpha_1+(2k-1)\alpha_2$ or $\lambda=(k+1)\alpha_1+2k\alpha_2$ where $k\in\mathbb{N}_+$, then $\tilde{w}=t^{\lambda}s_1s_2s_1\tau$ is already of minimal length in its $\widetilde{W}$-conjugacy class. Thus it is sufficient to consider those $\lambda=m\alpha_1+n\alpha_2\in P_+\cap Q_{sh}$ with $n\leqslant2m-3$, in this case we have
\begin{align*}
T_{t^{m\alpha_1+n\alpha_2}s_1s_2s_1\tau}\equiv&T_{s_0t^{m\alpha_1+n\alpha_2}s_1s_2s_1\tau s_0}=T_{t^{(1-n)\alpha_1+(1-m)\alpha_2}s_1\tau}\\
\equiv&(v-v^{-1})T_{s_2t^{(1-n)\alpha_1+(1-m)\alpha_2}s_1\tau}+T_{s_2t^{(1-n)\alpha_1+(1-m)\alpha_2}s_1\tau s_2}\\
\equiv&(v-v^{-1})T_{\tau^{-1}\cdot t^{(1-n)\alpha_1+(m-n)\alpha_2}s_2s_1\tau}+T_{\tau^{-1}\cdot t^{(1-n)\alpha_1+(m-n-1)\alpha_2}s_1\tau}\\
\equiv&(v-v^{-1})T_{t^{(m-1)\alpha_1+n\alpha_2}s_2s_1\tau}+T_{t^{(m-1)\alpha_1+n\alpha_2}s_1s_2s_1\tau}\\
\cdots&\cdots\\
\equiv&(v-v^{-1})\sum_{i=\lfloor\frac{n}{2}\rfloor+1}^{m-1}T_{\mathbb{O}_{i\alpha_1+n\alpha_2,\tau}}+T_{\mathbb{O}_{n, \tau}}.
\end{align*}
(2): If $\lambda=(2k-1)\alpha_1+k\alpha_2$, then $\tilde{w}=t^{\lambda}s_1\tau$ is a minimal length element in its $\widetilde{W}$-conjugacy class. Thus it is sufficient to consider $\lambda=m\alpha_1+n\alpha_2\in P_+\cap Q_{sh}$ with $m\leqslant2n-2$, thus by a similar argument we have
\begin{align*}
T_{t^{m\alpha_1+n\alpha_2}s_1\tau}\equiv&T_{s_0t^{m\alpha_1+n\alpha_2}s_1\tau s_0}=T_{t^{(1-n)\alpha_1+(1-m)\alpha_2}s_1s_2s_1\tau}\\
\cdots&\cdots\\
\equiv&(v-v^{-1})\sum_{i=\lfloor\frac{m+1}{2}\rfloor+1}^{n}T_{\mathbb{O}_{m\alpha_1+i\alpha_2,\tau}}+T_{\mathbb{O}_{1-m, \tau}}.
\end{align*}Thus (2) is proved.
\end{proof}
Note: $\tilde{w}=t^{2k\alpha_1+k\alpha_2}s_1\tau$ for $k\in\mathbb{N}_+$ is a minimal length element of $\mathbb{O}_{1-2k, \tau}$.
\begin{prop}\label{cpatau}
If $\tilde{w}=t^{k\alpha_1+2k\alpha_2}\tau$ with $k\in\mathbb{N}$. Then \[f_{\tilde{w},\mathbb{O}}=\left\{
                                                             \begin{array}{ll}
                                                               (v-v^{-1})^2, & \mathbb{O}=\mathbb{O}_{\lambda, \tau},\ where \ \lambda\in\sqcup_{j=2}^{k}E'_{j\alpha_1+(2j-1)\alpha_2,\tau}\\
                                                               (v-v^{-1}), & \mathbb{O}=\mathbb{O}_{i, \tau},\ where\ -k+1\leqslant i\leqslant2k\\
                                                               1, & \mathbb{O}=\mathbb{O}_{id, \tau}\\
                                                               0, & otherwise
                                                             \end{array}
                                                           \right.\]
\end{prop}
\begin{proof}
If $k=0$, there is nothing to prove. If $k=1$, then $$T_{t^{\alpha_1+2\alpha_2}\tau}\equiv(v-v^{-1})(T_{\mathbb{O}_{2,\tau}}+T_{\mathbb{O}_{1,\tau}}+T_{\mathbb{O}_{0,\tau}})+T_{\mathbb{O}_{id,\tau}}.$$If $k\geqslant2$, then
\begin{align*}T_{t^{k\alpha_1+2k\alpha_2}\tau}\equiv&(v-v^{-1})T_{t^{k\alpha_1+(1-k)\alpha_2}s_2\tau}+T_{t^{k\alpha_1+(1-k)\alpha_2}\tau}\\
\equiv&(v-v^{-1})T_{s_2t^{k\alpha_1+(1-k)\alpha_2}s_2\tau s_2}+(v-v^{-1})T_{s_0t^{k\alpha_1+(1-k)\alpha_2}\tau}+T_{s_0t^{k\alpha_1+(1-k)\alpha_2}\tau s_0}\\
\equiv&(v-v^{-1})T_{\mathbb{O}_{2k, \tau}}+(v-v^{-1})T_{t^{k\alpha_1+(1-k)\alpha_2}s_1s_2s_1\tau}+T_{t^{k\alpha_1+(1-k)\alpha_2}s_1s_2\tau}\\
\equiv&(v-v^{-1})T_{\mathbb{O}_{2k, \tau}}+(v-v^{-1})T_{\tau\cdot(t^{k\alpha_1+(1-k)\alpha_2}s_1s_2s_1\tau)}+T_{\tau\cdot(t^{k\alpha_1+(1-k)\alpha_2}s_1s_2\tau)}\\
\equiv&(v-v^{-1})T_{\mathbb{O}_{2k, \tau}}+(v-v^{-1})T_{t^{k\alpha_1+(2k-1)\alpha_2}s_1\tau}\\
&+(v-v^{-1})T_{s_0t^{k\alpha_1+(2k-1)\alpha_2}s_1s_2\tau}+T_{s_0t^{k\alpha_1+(2k-1)\alpha_2}s_1s_2\tau s_0}\\
\equiv&(v-v^{-1})T_{\mathbb{O}_{2k, \tau}}+(v-v^{-1})T_{t^{k\alpha_1+(2k-1)\alpha_2}s_1\tau}\\
&+(v-v^{-1})T_{\tau^{-1}\cdot(t^{(2-2k)\alpha_1+(1-k)\alpha_2}s_1\tau)}+T_{\tau^{-1}\cdot(t^{(2-2k)\alpha_1+(1-k)\alpha_2}\tau)}\\
=&(v-v^{-1})(T_{\mathbb{O}_{2k, \tau}}+T_{\mathbb{O}_{2k-1, \tau}})+(v-v^{-1})T_{t^{k\alpha_1+(2k-1)\alpha_2}s_1\tau}+T_{t^{(k-1)\alpha_1+2(k-1)\alpha_2}\tau}\\
\cdots&\cdots\\
\equiv&(v-v^{-1})(T_{\mathbb{O}_{2k, \tau}}+T_{\mathbb{O}_{2k-1, \tau}}+T_{\mathbb{O}_{2k-2, \tau}}+T_{\mathbb{O}_{2k-3, \tau}}+\cdots+T_{\mathbb{O}_{4, \tau}}+T_{\mathbb{O}_{3, \tau}})\\
&+(v-v^{-1})(T_{t^{k\alpha_1+(2k-1)\alpha_2}s_1\tau}+T_{t^{(k-1)\alpha_1+(2k-3)\alpha_2}s_1\tau}+\cdots+T_{t^{2\alpha_1+3\alpha_2}s_1\tau})+T_{t^{\alpha_1+2\alpha_2}\tau}\\
\equiv&(v-v^{-1})^2\sum_{\lambda\in\sqcup_{j=2}^{k}E'_{j\alpha_1+(2j-1)\alpha_2,\tau}}T_{\mathbb{O}_{\lambda,\tau}}+(v-v^{-1})\sum_{i=-k+1}^{2k}T_{\mathbb{O}_{i,\tau}}+T_{\mathbb{O}_{id,\tau}}.
\end{align*}The proposition is proved.
\end{proof}
\begin{prop}\label{cpatau12}
If $\tilde{w}=t^{k\alpha_1+(2k-1)\alpha_2}s_1s_2\tau$ with $k\in\mathbb{N}_+$. Then \[f_{\tilde{w},\mathbb{O}}=\left\{
                                                             \begin{array}{ll}
                                                               (v-v^{-1})^2, & \mathbb{O}=\mathbb{O}_{\lambda, \tau},\ where \ \lambda\in\sqcup_{j=2}^{k-1}E'_{j\alpha_1+(2j-1)\alpha_2,\tau}\\
                                                               (v-v^{-1}), & \mathbb{O}=\mathbb{O}_{i, \tau},\ where\ -k+2\leqslant i\leqslant2k-1\\
                                                               1, & \mathbb{O}=\mathbb{O}_{id, \tau}\\
                                                               0, & otherwise
                                                             \end{array}
                                                           \right.\]
\end{prop}
\begin{prop}\label{cpaatau}
If $\tilde{w}=t^{k\alpha_1+(2k-1)\alpha_2}\tau$ with $k\in\mathbb{N}_+$. Then \[f_{\tilde{w},\mathbb{O}}=\left\{
                                                             \begin{array}{ll}
                                                               (v-v^{-1})^2, & \mathbb{O}=\mathbb{O}_{\lambda, \tau},\ where \ \lambda\in\sqcup_{j=2}^{k}E'_{j\alpha_1+(2j-1)\alpha_2,\tau}\\
                                                               (v-v^{-1}), & \mathbb{O}=\mathbb{O}_{i, \tau},\ where\ -k+1\leqslant i\leqslant2k-1\\
                                                               1, & \mathbb{O}=\mathbb{O}_{id, \tau}\\
                                                               0, & otherwise
                                                             \end{array}
                                                           \right.\]
\end{prop}
The proofs of \ref{cpatau12} and \ref{cpaatau} are just like that of \ref{cpatau}. To avoid of tedious repetitions, we omit them.
\begin{rmk}
For $k\in\mathbb{N}_+$, since $t^{(k+1)\alpha_1+2k\alpha_2}s_1s_2\tau=s_2(\tau^{-1}\cdot t^{k\alpha_1+2k\alpha_2}\tau)s_2$, we have $t^{k\alpha_1+2k\alpha_2}\tau\tilde{\thicksim} t^{(k+1)\alpha_1+2k\alpha_2}s_1s_2\tau$ and thus $T_{t^{k\alpha_1+2k\alpha_2}\tau}\equiv T_{t^{(k+1)\alpha_1+2k\alpha_2}s_1s_2\tau}$.
\end{rmk}
\begin{prop}\label{cpbtau12}
If $\tilde{w}=t^{2k\alpha_1+k\alpha_2}s_1s_2\tau$ where $k\in\mathbb{N}_+$. Then \[f_{\tilde{w},\mathbb{O}}=\left\{
                                                             \begin{array}{ll}
                                                               (v-v^{-1})^2, & \mathbb{O}=\mathbb{O}_{\lambda, \tau},\ where \ \lambda\in\sqcup_{j=2}^{k}E_{(2j-1)\alpha_1+j\alpha_2,\tau}\\
                                                               (v-v^{-1}), & \mathbb{O}=\mathbb{O}_{i, \tau},\ where\ 2-2k\leqslant i\leqslant k\\
                                                               1, & \mathbb{O}=\mathbb{O}_{id, \tau}\\
                                                               0, & otherwise
                                                             \end{array}
                                                           \right.\]
\end{prop}
\begin{prop}\label{cpbtau}
If $\tilde{w}=t^{2k\alpha_1+k\alpha_2}\tau$ where $k\in\mathbb{N}_+$. Then \[f_{\tilde{w},\mathbb{O}}=\left\{
                                                             \begin{array}{ll}
                                                               (v-v^{-1})^2, & \mathbb{O}=\mathbb{O}_{\lambda, \tau},\ where \ \lambda\in\sqcup_{j=2}^{k}E_{(2j-1)\alpha_1+j\alpha_2,\tau}\\
                                                               (v-v^{-1}), & \mathbb{O}=\mathbb{O}_{i, \tau},\ where\ 1-2k\leqslant i\leqslant k\\
                                                               1, & \mathbb{O}=\mathbb{O}_{id, \tau}\\
                                                               0, & otherwise
                                                             \end{array}
                                                           \right.\]
\end{prop}
\begin{prop}\label{cpbbtau12}
If $\tilde{w}=t^{(2k-1)\alpha_1+k\alpha_2}s_1s_2\tau$ where $k\in\mathbb{N}_{\geqslant2}$. Then \[f_{\tilde{w},\mathbb{O}}=\left\{
                                                             \begin{array}{ll}
                                                               (v-v^{-1})^2, & \mathbb{O}=\mathbb{O}_{\lambda, \tau},\ where \ \lambda\in\sqcup_{j=2}^{k}E_{(2j-1)\alpha_1+j\alpha_2,\tau}\\
                                                               (v-v^{-1}), & \mathbb{O}=\mathbb{O}_{i, \tau},\ where\ 3-2k\leqslant i\leqslant k\\
                                                               1, & \mathbb{O}=\mathbb{O}_{id, \tau}\\
                                                               0, & otherwise
                                                             \end{array}
                                                           \right.\]
\end{prop}
We prove \ref{cpbtau12}, and since the proofs of \ref{cpbtau} and \ref{cpbbtau12} are similar, we omit them.
\begin{proof}[Proof of Proposition \ref{cpbtau12}]If $k=1$, then $T_{t^{2\alpha_1+\alpha_2}s_1s_2\tau}\equiv(v-v^{-1})(T_{\mathbb{O}_{0, \tau}}+T_{\mathbb{O}_{1, \tau}})+T_{\mathbb{O}_{id, \tau}}$. If for $k\geqslant2$, then
\begin{align*}
T_{t^{2k\alpha_1+k\alpha_2}s_1s_2\tau}\equiv&T_{s_0t^{2k\alpha_1+k\alpha_2}s_1s_2\tau s_0}=T_{t^{(1-k)\alpha_1+(1-2k)\alpha_2}\tau}\\
\equiv&(v-v^{-1})T_{t^{(1-k)\alpha_1+k\alpha_2}s_2\tau}+T_{t^{(2-k)\alpha_1+k\alpha_2}s_1s_2\tau}\\
\equiv&(v-v^{-1})T_{\tau^{-1}\cdot t^{(1-k)\alpha_1+k\alpha_2}s_2\tau}+T_{\tau^{-1}\cdot t^{(2-k)\alpha_1+k\alpha_2}s_1s_2\tau}\\
\equiv&(v-v^{-1})T_{t^{(2k-1)\alpha_1+(k-1)\alpha_2}s_1\tau}+(v-v^{-1})T_{s_2t^{(2k-1)\alpha_1+(k-1)\alpha_2}s_1s_2\tau}+T_{s_2t^{(2k-1)\alpha_1+(k-1)\alpha_2}s_1s_2\tau s_2}\\
\cdots&\cdots\\
\equiv&(v-v^{-1})(T_{t^{(2k-1)\alpha_1+(k-1)\alpha_2}s_1\tau}+T_{t^{(2k-1)\alpha_1+k\alpha_2}s_1s_2s_1\tau}+T_{t^{(2k-2)\alpha_1+(k-1)\alpha_2}s_1\tau})+T_{t^{(2k-2)\alpha_1+(k-1)\alpha_2}s_1s_2\tau}
\end{align*}
\begin{align*}
\cdots&\cdots\\
\equiv&(v-v^{-1})\sum_{i=-2(k-1)}^{0}T_{\mathbb{O}_{i,\tau}}+(v-v^{-1})\sum_{i=1}^{k}T_{t^{(2i-1)\alpha_1+i\alpha_2}s_1s_2s_1\tau}+T_{\mathbb{O}_{id,\tau}}\\
\equiv&(v-v^{-1})^2\sum_{\lambda\in\sqcup_{j=2}^{k}E_{(2j-1)\alpha_1+j\alpha_2,\tau}}T_{\mathbb{O}_{\lambda,\tau}}+(v-v^{-1})\sum_{i=-2(k-1)}^{k}T_{\mathbb{O}_{i,\tau}}+T_{\mathbb{O}_{id,\tau}}.
\end{align*}
\end{proof}
\begin{rmk}
For $k\in\mathbb{N}_{\geqslant2}$, since $t^{(2k-1)\alpha_1+k\alpha_2}\tau=s_1(\tau\cdot(t^{2k\alpha_1+k\alpha_2}s_1s_2\tau))s_1$ and $\ell(t^{2k\alpha_1+k\alpha_2}s_1s_2\tau)=\ell(t^{(2k-1)\alpha_1+k\alpha_2}\tau)$, thus $T_{t^{2k\alpha_1+k\alpha_2}s_1s_2\tau}\equiv T_{t^{(2k-1)\alpha_1+k\alpha_2}\tau}$.
\end{rmk}
\begin{prop}\label{cpalambdatau}
Let $\tilde{w}=t^{m\alpha_1+n\alpha_2}\tau$ where $m,\ n\in\mathbb{N}_+$, $m\alpha_1+n\alpha_2\in P_+\cap Q_{sh}$, $m\neq2n-1$ and $n\neq2m-1$. Then \[f_{\tilde{w},\mathbb{O}}=\left\{
                                                             \begin{array}{ll}
                                                               (v-v^{-1})^2, & \mathbb{O}=\mathbb{O}_{\lambda, \tau},\ where \ \lambda\in\sqcup_{j=2}^{\lfloor\frac{n+1}{2}\rfloor}E'_{j\alpha_1+(2j-1)\alpha_2,\tau}\sqcup\sqcup_{j=\lfloor\frac{n+1}{2}\rfloor+1}^{m}E'_{j\alpha_1+n\alpha_2,\tau}\\
                                                               (v-v^{-1}), & \mathbb{O}=\mathbb{O}_{i, \tau},\ where\ 1-m\leqslant i\leqslant n\\
                                                               1, & \mathbb{O}=\mathbb{O}_{id, \tau}\\
                                                               0, & otherwise
                                                             \end{array}
                                                           \right.\]
\end{prop}
\begin{prop}\label{cpalambdatau12}
Let $\tilde{w}=t^{m\alpha_1+n\alpha_2}s_1s_2\tau$ where $m,\ n\in\mathbb{N}_+$, $\lambda\in P_+\cap Q_{sh}$, $m\neq2n-1$ and $n\neq2m-1$, $n\neq2m-2$. Then \[f_{\tilde{w},\mathbb{O}}=\left\{
                                                             \begin{array}{ll}
                                                               (v-v^{-1})^2, & \mathbb{O}=\mathbb{O}_{\lambda, \tau},\ where \ \lambda\in\sqcup_{j=2}^{\lfloor\frac{m+1}{2}\rfloor}E_{(2j-1)\alpha_1+j\alpha_2,\tau}\sqcup\sqcup_{j=\lfloor\frac{m+1}{2}\rfloor+1}^{n}E_{m\alpha_1+j\alpha_2,\tau}\\
                                                               (v-v^{-1}), & \mathbb{O}=\mathbb{O}_{i, \tau},\ where\ 2-m\leqslant i\leqslant n\\
                                                               1, & \mathbb{O}=\mathbb{O}_{id, \tau}\\
                                                               0, & otherwise
                                                             \end{array}
                                                           \right.\]
\end{prop}
We prove \ref{cpalambdatau}, and the proof of \ref{cpalambdatau12} is similar which will be omitted.
\begin{proof}[Proof of Proposition \ref{cpalambdatau}]If $n$ is odd, then
\begin{align*}
T_{t^{m\alpha_1+n\alpha_2}\tau}\equiv&(v-v^{-1})T_{t^{(1-n)\alpha_1+(1-m)\alpha_2}s_1s_2s_1\tau}+T_{t^{(1-n)\alpha_1+(1-m)\alpha_2}s_1s_2\tau}\\
\equiv&(v-v^{-1})T_{\tau^{-1}\cdot(s_2t^{(1-n)\alpha_1+(1-m)\alpha_2}s_1s_2s_1\tau s_2)}+T_{\tau^{-1}\cdot(s_2t^{(1-n)\alpha_1+(1-m)\alpha_2}s_1s_2\tau s_2)}\\
=&(v-v^{-1})T_{t^{m\alpha_1+n\alpha_2}s_1\tau}+T_{t^{(m-1)\alpha_1+n\alpha_2}\tau}\\
\cdots&\cdots\\
\equiv&(v-v^{-1})\sum_{j=\frac{n+1}{2}+1}^{m}T_{t^{j\alpha_1+n\alpha_2}s_1\tau}+T_{t^{\frac{n+1}{2}\alpha_1+n\alpha_2}\tau}\\
\equiv&(v-v^{-1})\sum_{j=1-m}^{-\frac{n+1}{2}}T_{\mathbb{O}_{j,\tau}}+(v-v^{-1})^2\sum_{\lambda\in\sqcup_{j=\frac{n+1}{2}+1}^{m}E'_{j\alpha_1+n\alpha_2,\tau}}T_{\mathbb{O}_{\lambda,\tau}}\\
&+(v-v^{-1})\sum_{j=1-\frac{n+1}{2}}^{n}T_{\mathbb{O}_{j,\tau}}+(v-v^{-1})^2\sum_{\lambda\in\sqcup_{j=2}^{\frac{n+1}{2}}E'_{j\alpha_1+(2j-1)\alpha_2,\tau}}T_{\mathbb{O}_{\lambda,\tau}}+T_{\mathbb{O}_{id,\tau}}\\
=&(v-v^{-1})^2\sum_{\lambda\in\sqcup_{j=2}^{\frac{n+1}{2}}E'_{j\alpha_1+(2j-1)\alpha_2,\tau}\sqcup\sqcup_{j=\frac{n+1}{2}+1}^{m}E'_{j\alpha_1+n\alpha_2,\tau}}T_{\mathbb{O}_{\lambda,\tau}}+(v-v^{-1})\sum_{i=1-m}^{n}T_{\mathbb{O}_{i,\tau}}+T_{\mathbb{O}_{id,\tau}}.
\end{align*}If $n$ is even, then by a similar argument
\begin{align*}
T_{t^{m\alpha_1+n\alpha_2}\tau}\equiv&(v-v^{-1})\sum_{i=\frac{n}{2}+1}^{m}T_{t^{i\alpha_1+n\alpha_2}s_1\tau}+T_{t^{\frac{n}{2}\alpha_1+n\alpha_2}\tau}\\
\equiv&(v-v^{-1})^2\sum_{\lambda\in\sqcup_{j=2}^{\frac{n}{2}}E'_{j\alpha_1+(2j-1)\alpha_2,\tau}\sqcup\sqcup_{j=\frac{n}{2}+1}^{m}E'_{j\alpha_1+n\alpha_2,\tau}}T_{\mathbb{O}_{\lambda,\tau}}+(v-v^{-1})\sum_{i=1-m}^{n}T_{\mathbb{O}_{i,\tau}}+T_{\mathbb{O}_{id,\tau}}.
\end{align*}Combine them together, the proposition is proved.
\end{proof}
\begin{rmk}
If $\tilde{w}=t^{m\alpha_1+n\alpha_2}\tau$ where $n\geqslant1$, $m\geqslant2n+1$ or $n\leqslant0$, $m+n\geqslant1$, and we set $\tilde{w}_1=s_2\tilde{w}s_2=t^{(m+1)\alpha_1+(m-n)\alpha_2}s_1s_2\tau$. Then $\ell(\tilde{w}_1)=\ell(\tilde{w})$, $\tilde{w}\tilde{\thicksim}\tilde{w}_1$ and hence $T_{\tilde{w}}\equiv T_{\tilde{w}_1}$. We know that $\tilde{w}_1$ is contained in the situation we considered above.
\end{rmk}
\begin{cor}\label{cor1}
(1) If $\tilde{w}=t^{(2n+1)\alpha_1+n\alpha_2}\tau$ where $n\in\mathbb{N}_+$. Then \[f_{\tilde{w},\mathbb{O}}=\left\{
                                                             \begin{array}{ll}
                                                               (v-v^{-1})^2, & \mathbb{O}=\mathbb{O}_{\lambda, \tau},\ where \ \lambda\in\sqcup_{j=2}^{n+1}E_{(2j-1)\alpha_1+j\alpha_2,\tau}\\
                                                               (v-v^{-1}), & \mathbb{O}=\mathbb{O}_{i, \tau},\ where\ -2n\leqslant i\leqslant n+1\\
                                                               1, & \mathbb{O}=\mathbb{O}_{id, \tau}\\
                                                               0, & otherwise.
                                                             \end{array}
                                                           \right.\]
(2) If $\tilde{w}=t^{(2n+2)\alpha_1+n\alpha_2}\tau$ where $n\in\mathbb{N}_+$. Then \[f_{\tilde{w},\mathbb{O}}=\left\{
                                                             \begin{array}{ll}
                                                               (v-v^{-1})^2, & \mathbb{O}=\mathbb{O}_{\lambda, \tau},\ where \ \lambda\in\sqcup_{j=2}^{n+2}E_{(2j-1)\alpha_1+j\alpha_2,\tau}\\
                                                               (v-v^{-1}), & \mathbb{O}=\mathbb{O}_{i, \tau},\ where\ -1-2n\leqslant i\leqslant n+2\\
                                                               1, & \mathbb{O}=\mathbb{O}_{id, \tau}\\
                                                               0, & otherwise.
                                                             \end{array}
                                                           \right.\]
(3) If $\tilde{w}=t^{m\alpha_1+n\alpha_2}\tau$ where $n\geqslant1$, $m\geqslant2n+3$ or $n\leqslant0$, $m+n\geqslant1$. Then\[f_{\tilde{w},\mathbb{O}}=\left\{
                                                             \begin{array}{ll}
                                                               (v-v^{-1})^2, & \mathbb{O}=\mathbb{O}_{\lambda, \tau},\ where \ \lambda\in\sqcup_{j=2}^{\lfloor\frac{m+2}{2}\rfloor}E_{(2j-1)\alpha_1+j\alpha_2,\tau}\sqcup\sqcup_{j=\lfloor\frac{m+2}{2}\rfloor+1}^{m-n}E_{(m+1)\alpha_1+j\alpha_2,\tau}\\
                                                               (v-v^{-1}), & \mathbb{O}=\mathbb{O}_{i, \tau},\ where\ 1-m\leqslant i\leqslant m-n\\
                                                               1, & \mathbb{O}=\mathbb{O}_{id, \tau}\\
                                                               0, & otherwise.
                                                             \end{array}
                                                           \right.\]
\end{cor}
\begin{rmk}
If $\tilde{w}=t^{m\alpha_1+n\alpha_2}s_1s_2\tau$ where $n\geqslant1$, $m\geqslant2n+1$ or $n\leqslant0$, $m+n\geqslant2$, then $T_{t^{m\alpha_1+n\alpha_2}s_1s_2\tau}\equiv(v-v^{-1})T_{t^{m\alpha_1+(m-n)\alpha_2}s_1s_2s_1\tau}+T_{t^{(m-1)\alpha_1+(m-n-1)\alpha_2}\tau}$.
\end{rmk}
\begin{cor}\label{cor2}
(1) If $\tilde{w}=t^{(2k+1)\alpha_1+k\alpha_2}s_1s_2\tau$ where $k\in\mathbb{N}_+$, then\[f_{\tilde{w},\mathbb{O}}=\left\{
                                                             \begin{array}{ll}
                                                               (v-v^{-1})^2, & \mathbb{O}=\mathbb{O}_{\lambda, \tau},\ where \ \lambda\in\sqcup_{j=2}^{k+1}E_{(2j-1)\alpha_1+j\alpha_2,\tau}\\
                                                               (v-v^{-1}), & \mathbb{O}=\mathbb{O}_{i, \tau},\ where\ 1-2k\leqslant i\leqslant k+1\\
                                                               1, & \mathbb{O}=\mathbb{O}_{id, \tau}\\
                                                               0, & otherwise.
                                                             \end{array}
                                                           \right.\]
(2) If $\tilde{w}=t^{(2k+2)\alpha_1+k\alpha_2}s_1s_2\tau$ where $k\in\mathbb{N}_+$, then\[f_{\tilde{w},\mathbb{O}}=\left\{
                                                             \begin{array}{ll}
                                                               (v-v^{-1})^2, & \mathbb{O}=\mathbb{O}_{\lambda, \tau},\ where \ \lambda\in\sqcup_{j=2}^{k+1}E_{(2j-1)\alpha_1+j\alpha_2,\tau}\sqcup E_{(2k+2)\alpha_1+(k+2)\alpha_2,\tau}\\
                                                               (v-v^{-1}), & \mathbb{O}=\mathbb{O}_{i, \tau},\ where\ -2k\leqslant i\leqslant k+2\\
                                                               1, & \mathbb{O}=\mathbb{O}_{id, \tau}\\
                                                               0, & otherwise.
                                                             \end{array}
                                                           \right.\]
(3) If $\tilde{w}=t^{(k+2)\alpha_1-k\alpha_2}s_1s_2\tau$ where $k\in\mathbb{N}$, then\[f_{\tilde{w},\mathbb{O}}=\left\{
                                                             \begin{array}{ll}
                                                               (v-v^{-1})^2, & \mathbb{O}=\mathbb{O}_{\lambda, \tau},\ where \ \lambda\in\sqcup_{j=2}^{k+1}E'_{j\alpha_1+(2j-1)\alpha_2,\tau}\\
                                                               (v-v^{-1}), & \mathbb{O}=\mathbb{O}_{i, \tau},\ where\ -k\leqslant i\leqslant 2k+2\\
                                                               1, & \mathbb{O}=\mathbb{O}_{id, \tau}\\
                                                               0, & otherwise.
                                                             \end{array}
                                                           \right.\]
(4) If $\tilde{w}=t^{m\alpha_1+n\alpha_2}s_1s_2\tau$ where $n\in\mathbb{N}_+$, $m\geqslant2n+3$ or $n\in\mathbb{Z}_{\leqslant0}$, $m+n\geqslant3$ then\[f_{\tilde{w},\mathbb{O}}=\left\{
                                                             \begin{array}{ll}
                                                               (v-v^{-1})^2, & \mathbb{O}=\mathbb{O}_{\lambda, \tau},\ where\\
                                                               &\lambda\in\sqcup_{j=2}^{\lfloor\frac{m-n}{2}\rfloor}E'_{j\alpha_1+(2j-1)\alpha_2,\tau}\sqcup\sqcup_{j=\lfloor\frac{m-n}{2}\rfloor+1}^{m-1}E'_{j\alpha_1+(m-n-1)\alpha_2,\tau}\sqcup E_{m\alpha_1+(m-n)\alpha_2,\tau}\\
                                                               (v-v^{-1}), & \mathbb{O}=\mathbb{O}_{i, \tau},\ where\ 2-m\leqslant i\leqslant m-n\\
                                                               1, & \mathbb{O}=\mathbb{O}_{id, \tau}\\
                                                               0, & otherwise.
                                                             \end{array}
                                                           \right.\]
\end{cor}
\begin{proof}[Proofs of Corollary \ref{cor1} and \ref{cor2}]We choose \ref{cor1} (1) and \ref{cor2} (4) to prove. For \ref{cor1} (1), it follows directly from\begin{align*}T_{t^{(2n+1)\alpha_1+n\alpha_2}\tau}\equiv&T_{t^{(2n+2)\alpha_1+(n+1)\alpha_2}s_1s_2\tau}\\
\equiv&(v-v^{-1})^2\sum_{\lambda\in\sqcup_{j=2}^{n+1}E_{(2j-1)\alpha_1+j\alpha_2,\tau}}T_{\mathbb{O}_{\lambda, \tau}}+(v-v^{-1})\sum_{i=-2n}^{n+1}T_{\mathbb{O}_{i, \tau}}+T_{\mathbb{O}_{id, \tau}}.
\end{align*}
And for \ref{cor2} (4), we obtain the class polynomials form\begin{align*}T_{t^{m\alpha_1+n\alpha_2}s_1s_2\tau}\equiv&(v-v^{-1})T_{t^{m\alpha_1+(m-n)\alpha_2}s_1s_2s_1\tau}+T_{t^{(m-1)\alpha_1+(m-n-1)\alpha_2}\tau}\\
\equiv&(v-v^{-1})^2\sum_{\lambda\in\sqcup_{j=2}^{\lfloor\frac{m-n}{2}\rfloor}E'_{j\alpha_1+(2j-1)\alpha_2,\tau}\sqcup\sqcup_{j=\lfloor\frac{m-n}{2}\rfloor+1}^{m-1}E'_{j\alpha_1+(m-n-1)\alpha_2,\tau}\sqcup E_{m\alpha_1+(m-n)\alpha_2,\tau}}T_{\mathbb{O}_{\lambda, \tau}}\\
&+(v-v^{-1})\sum_{i=2-m}^{m-n}T_{\mathbb{O}_{i, \tau}}+T_{\mathbb{O}_{id, \tau}}.
\end{align*}
\end{proof}
\begin{prop}\label{cpoddtau}
For $i\in\mathbb{N}_+$ and $\tilde{w}\in\mathbb{O}_{i,\tau}$.
(1) If $\ell(\mathbb{O}_{i,\tau})\leqslant\ell(\tilde{w})\leqslant6i-5$. Then\[f_{\tilde{w},\mathbb{O}}=\left\{
                                                             \begin{array}{ll}
                                                               (v-v^{-1}), & \mathbb{O}=\mathbb{O}_{\lambda, \tau},\ where\ \lambda\in E_{(\lfloor\frac{i}{2}\rfloor+1+\frac{\ell(\tilde{w})-\ell(\mathbb{O}_{i,\tau})}{2})\alpha_1+i\alpha_2,\tau}\\
                                                               1, & \mathbb{O}=\mathbb{O}_{i, \tau}\\
                                                               0, & otherwise
                                                             \end{array}
                                                           \right.\]
(2) If $\ell(\tilde{w})=6i-3+4k$ where $k\in\mathbb{N}$. Then\[f_{\tilde{w},\mathbb{O}}=\left\{
                                                             \begin{array}{ll}
                                                               k(v-v^{-1})^3, & \mathbb{O}=\mathbb{O}_{\lambda, \tau},\ where\ \lambda\in\sqcup_{j=2}^{i-1}E_{(2j-1)\alpha_1+j\alpha_2,\tau}\\
                                                               k(v-v^{-1})^3+(v-v^{-1}), & \mathbb{O}=\mathbb{O}_{\lambda, \tau},\ where\ \lambda\in E_{(2i-1)\alpha_1+i\alpha_2,\tau}\\
                                                               (k-j)(v-v^{-1})^3, & \mathbb{O}=\mathbb{O}_{\lambda, \tau},\ where\ 1\leqslant j\leqslant k-1,\ \  \lambda\in E_{(2i-1+j)\alpha_1+(i+j)\alpha_2,\tau}\\
                                                               (k-1-j)(v-v^{-1})^3, & \mathbb{O}=\mathbb{O}_{\lambda, \tau},\ where\ 2\leqslant j\leqslant k-2,\ \ \lambda\in E'_{(2i+j)\alpha_1+(i+j)\alpha_2,\tau}\\
                                                               (k-j)(v-v^{-1})^3+(v-v^{-1}), & \mathbb{O}=\mathbb{O}_{(2i-1+j)\alpha_1+(i+j)\alpha_2, \tau},\ where\ 1\leqslant j\leqslant k\\
                                                               k(v-v^{-1})^2, & \mathbb{O}=\mathbb{O}_{l, \tau},\ where\ 2-2i\leqslant l\leqslant i-1\\
                                                               k(v-v^{-1})^2+1, & \mathbb{O}=\mathbb{O}_{i, \tau}\\
                                                               (k-j)(v-v^{-1})^2, & \mathbb{O}=\mathbb{O}_{2-2i-j, \tau}\ or\ \mathbb{O}_{i+j, \tau},\ where\ 1\leqslant j\leqslant k-1\\
                                                               k(v-v^{-1}), & \mathbb{O}=\mathbb{O}_{id,\tau}\\
                                                               (v-v^{-1}), & \mathbb{O}=\mathbb{O}_{(2i-1)\alpha_1+i\alpha_2,\tau}\\
                                                               0, & otherwise
                                                             \end{array}
                                                           \right.\]
(3) If $\ell(\tilde{w})=6i-1+4k$ where $k\in\mathbb{N}$. Then\[f_{\tilde{w},\mathbb{O}}=\left\{
                                                             \begin{array}{ll}
                                                               (k+1)(v-v^{-1})^3, & \mathbb{O}=\mathbb{O}_{\lambda, \tau},\ where\ \lambda\in\sqcup_{j=2}^{i-1}E_{(2j-1)\alpha_1+j\alpha_2,\tau}\\
                                                               (k+1)(v-v^{-1})^3+(v-v^{-1}), & \mathbb{O}=\mathbb{O}_{\lambda, \tau},\ where\ \lambda\in E_{(2i-1)\alpha_1+i\alpha_2,\tau}\\
                                                               (k+1-j)(v-v^{-1})^3, & \mathbb{O}=\mathbb{O}_{\lambda, \tau},\ where\ 1\leqslant j\leqslant k,\ \  \lambda\in E_{(2i-1+j)\alpha_1+(i+j)\alpha_2,\tau}\\
                                                               (k-j)(v-v^{-1})^3, & \mathbb{O}=\mathbb{O}_{\lambda, \tau},\ where\ 2\leqslant j\leqslant k-1,\ \ \lambda\in E'_{(2i+j)\alpha_1+(i+j)\alpha_2,\tau}\\
                                                               (k+1-j)(v-v^{-1})^3+(v-v^{-1}), & \mathbb{O}=\mathbb{O}_{(2i-1+j)\alpha_1+(i+j)\alpha_2, \tau},\ where\ 1\leqslant j\leqslant k\\
                                                               (k+1)(v-v^{-1})^2, & \mathbb{O}=\mathbb{O}_{l, \tau},\ where\ 2-2i\leqslant l\leqslant i-1\\
                                                               (k+1)(v-v^{-1})^2+1, & \mathbb{O}=\mathbb{O}_{i, \tau}\\
                                                               (k+1-j)(v-v^{-1})^2, & \mathbb{O}=\mathbb{O}_{2-2i-j, \tau}\ or\ \mathbb{O}_{i+j, \tau},\ where\ 1\leqslant j\leqslant k\\
                                                               (k+1)(v-v^{-1}), & \mathbb{O}=\mathbb{O}_{id,\tau}\\
                                                               (v-v^{-1}), & \mathbb{O}=\mathbb{O}_{(2i-1)\alpha_1+i\alpha_2,\tau}\\
                                                               0, & otherwise
                                                             \end{array}
                                                           \right.\]
\end{prop}
\begin{prop}\label{cpodd1tau}
For $i\in\mathbb{N}$ and $\tilde{w}\in\mathbb{O}_{-i,\tau}$.
(1) If $\ell(\mathbb{O}_{-i,\tau})\leqslant\ell(\tilde{w})\leqslant6i+3$. Then\[f_{\tilde{w},\mathbb{O}}=\left\{
                                                             \begin{array}{ll}
                                                               (v-v^{-1}), & \mathbb{O}=\mathbb{O}_{\lambda, \tau},\ where\ \lambda\in E'_{(i+1)\alpha_1+(\lfloor\frac{i+1}{2}\rfloor+\frac{\ell(\tilde{w})-\ell(\mathbb{O}_{-i,\tau})}{2})\alpha_2,\tau}\\
                                                               1, & \mathbb{O}=\mathbb{O}_{-i, \tau}\\
                                                               0, & otherwise.
                                                             \end{array}
                                                           \right.\]
(2) If $\ell(\tilde{w})=6i+5+4k$ where $k\in\mathbb{N}$. Then\[f_{\tilde{w},\mathbb{O}}=\left\{
                                                             \begin{array}{ll}
                                                               k(v-v^{-1})^3, & \mathbb{O}=\mathbb{O}_{\lambda, \tau},\ where\ \lambda\in\sqcup_{j=2}^{i}E'_{j\alpha_1+(2j-1)\alpha_2,\tau}\\
                                                               k(v-v^{-1})^3+(v-v^{-1}), & \mathbb{O}=\mathbb{O}_{\lambda, \tau},\ where\ \lambda\in E'_{(i+1)\alpha_1+(2i+1)\alpha_2,\tau}\\
                                                               (k-j)(v-v^{-1})^3, & \mathbb{O}=\mathbb{O}_{\lambda, \tau},\ where\ 1\leqslant j\leqslant k-1,\ \lambda\in E'_{(i+1+j)\alpha_1+(2i+1+j)\alpha_2,\tau}\\
                                                               (k-j)(v-v^{-1})^3, & \mathbb{O}=\mathbb{O}_{\lambda, \tau},\ where\ 2\leqslant j\leqslant k-1,\ \lambda\in E_{(i+1+j)\alpha_1+(2i+2+j)\alpha_2,\tau}\\
                                                               (k-j)(v-v^{-1})^3+(v-v^{-1}), & \mathbb{O}=\mathbb{O}_{(i+1+j)\alpha_1+(2i+2+j)\alpha_2, \tau},\ where\ 1\leqslant j\leqslant k\\
                                                               k(v-v^{-1})^2, & \mathbb{O}=\mathbb{O}_{l, \tau},\ where\ 1-i\leqslant l\leqslant 2i+2\\
                                                               k(v-v^{-1})^2+1, & \mathbb{O}=\mathbb{O}_{-i, \tau}\\
                                                               (k-j)(v-v^{-1})^2, & \mathbb{O}=\mathbb{O}_{-i-j, \tau}\ or\ \mathbb{O}_{2i+2+j, \tau},\ where\ 1\leqslant j\leqslant k-1\\
                                                               k(v-v^{-1}), & \mathbb{O}=\mathbb{O}_{id,\tau}\\
                                                               (v-v^{-1}), & \mathbb{O}=\mathbb{O}_{(i+1)\alpha_1+(2i+2)\alpha_2,\tau}\\
                                                               0, & otherwise.
                                                             \end{array}
                                                           \right.\]
(3) If $\ell(\tilde{w})=6i+7+4k$ where $k\in\mathbb{N}$. Then\[f_{\tilde{w},\mathbb{O}}=\left\{
                                                             \begin{array}{ll}
                                                               (k+1)(v-v^{-1})^3, & \mathbb{O}=\mathbb{O}_{\lambda, \tau},\ where\ \lambda\in\sqcup_{j=2}^{i}E'_{j\alpha_1+(2j-1)\alpha_2,\tau}\\
                                                               (k+1)(v-v^{-1})^3+(v-v^{-1}), & \mathbb{O}=\mathbb{O}_{\lambda, \tau},\ where\ \lambda\in E'_{(i+1)\alpha_1+(2i+1)\alpha_2,\tau}\\
                                                               (k+1-j)(v-v^{-1})^3, & \mathbb{O}=\mathbb{O}_{\lambda, \tau},\ where\ 1\leqslant j\leqslant k,\ \lambda\in E'_{(i+1+j)\alpha_1+(2i+1+j)\alpha_2,\tau}\\
                                                               (k+1-j)(v-v^{-1})^3, & \mathbb{O}=\mathbb{O}_{\lambda, \tau},\ where\ 2\leqslant j\leqslant k,\ \lambda\in E_{(i+1+j)\alpha_1+(2i+2+j)\alpha_2,\tau}\\
                                                               (k+1-j)(v-v^{-1})^3+(v-v^{-1}), & \mathbb{O}=\mathbb{O}_{(i+1+j)\alpha_1+(2i+2+j)\alpha_2, \tau},\ where\ 1\leqslant j\leqslant k\\
                                                               (k+1)(v-v^{-1})^2, & \mathbb{O}=\mathbb{O}_{l, \tau},\ where\ 1-i\leqslant l\leqslant 2i+2\\
                                                               (k+1)(v-v^{-1})^2+1, & \mathbb{O}=\mathbb{O}_{-i, \tau}\\
                                                               (k+1-j)(v-v^{-1})^2, & \mathbb{O}=\mathbb{O}_{-i-j, \tau}\ or\ \mathbb{O}_{2i+2+j, \tau},\ where\ 1\leqslant j\leqslant k\\
                                                               (k+1)(v-v^{-1}), & \mathbb{O}=\mathbb{O}_{id,\tau}\\
                                                               (v-v^{-1}), & \mathbb{O}=\mathbb{O}_{(i+1)\alpha_1+(2i+2)\alpha_2,\tau}\\
                                                               0, & otherwise.
                                                             \end{array}
                                                           \right.\]
\end{prop}
We calculate class polynomials for Proposition \ref{cpoddtau}. Symmetrically, we obtain them for Proposition \ref{cpodd1tau}.
\begin{proof}[Proof of Proposition \ref{cpoddtau}]Since (3) is easily deduced from (2), it is sufficient for us to prove (1) and (2). We set $A=\{t^{k\alpha_1+i\alpha_2}s_1s_2s_1\tau\mid k\geqslant\lfloor\frac{i}{2}\rfloor+1\}\sqcup\{t^{(2i-1+k)\alpha_1+(i+k)\alpha_2}s_2\tau\mid k\in\mathbb{N}\}$. Then we can check directly that any element in $\mathbb{O}_{i,\tau}$ is $\tilde{\thicksim}$ to a unique element in $A$. Hence it is sufficient to consider those elements in $A$. For (1) It is sufficient to consider that $\tilde{w}=t^{(\lfloor\frac{i}{2}\rfloor+1+\frac{\ell(\tilde{w})-\ell(\mathbb{O}_{i,\tau})}{2})\alpha_1+i\alpha_2}s_1s_2s_1\tau$, then
\[T_{\tilde{w}}\equiv(v-v^{-1})\sum_{E_{(\lfloor\frac{i}{2}\rfloor+1+\frac{\ell(\tilde{w})-\ell(\mathbb{O}_{i,\tau})}{2})\alpha_1+i\alpha_2,\tau}}+T_{\mathbb{O}_{id,\tau}}.
\]For (2), it is sufficient to consider that $\tilde{w}=t^{(2i-1+k)\alpha_1+(i+k)\alpha_2}s_2\tau$ where $k\in\mathbb{N}$. If $k\leqslant3$, we calculate the class polynomials one by one, now if $k\geqslant4$, then
\begin{align*}T_{t^{(2i-1+k)\alpha_1+(i+k)\alpha_2}s_2\tau}\equiv&(v-v^{-1})T_{t^{(2i-1+k)\alpha_1+(i+k)\alpha_2}s_2s_1\tau}+(v-v^{-1})T_{t^{(2i-1+k)\alpha_1+(i+k-1)\alpha_2}s_1s_2\tau}+T_{t^{(2i-2+k)\alpha_1+(i+k-1)\alpha_2}s_2\tau}\\
\cdots&\cdots\\
\equiv&(v-v^{-1})\sum_{j=3}^{k}T_{\mathbb{O}_{(2i-1+j)\alpha_1+(i+j)\alpha_2,\tau}}+(v-v^{-1})\sum_{j=2}^{k-1}T_{t^{(2i+j)\alpha_1+(i+j)\alpha_2}s_1s_2\tau}+T_{t^{(2i+1)\alpha_1+(i+2)\alpha_2}s_2\tau}\\
\end{align*}We go further to obtain the class polynomials from the righthand side which is actually
\[k(v-v^{-1})^3\sum_{\lambda\in\sqcup_{j=2}^{i-1}E_{(2j-1)\alpha_1+j\alpha_2,\tau}}T_{\mathbb{O}_{\lambda, \tau}}+(k(v-v^{-1})^3+(v-v^{-1}))\sum_{\lambda\in E_{(2i-1)\alpha_1+i\alpha_2,\tau}}T_{\mathbb{O}_{\lambda, \tau}}\]
\[+\sum_{j=1}^{k-1}(k-j)(v-v^{-1})^3\sum_{\lambda\in E_{(2j-1+j)\alpha_1+(i+j)\alpha_2,\tau}}T_{\mathbb{O}_{\lambda, \tau}}+\sum_{j=2}^{k-2}(k-1-j)(v-v^{-1})^3\sum_{\lambda\in E'_{(2i+j)\alpha_1+(i+j)\alpha_2,\tau}}T_{\mathbb{O}_{\lambda, \tau}}\]
\[+\sum_{j=1}^{k}((k-j)(v-v^{-1})^3+(v-v^{-1}))T_{\mathbb{O}_{(2i-1+j)\alpha_1+(i+j)\alpha_2, \tau}}+\sum_{j=2-2i}^{i-1}k(v-v^{-1})^2T_{\mathbb{O}_{j, \tau}}+(k(v-v^{-1})^2+1)T_{\mathbb{O}_{i, \tau}}\]\[+(v-v^{-1})T_{\mathbb{O}_{(2i-1)\alpha_1+i\alpha_2,\tau}}+\sum_{j=1}^{k-1}(k-j)(v-v^{-1})^2(T_{\mathbb{O}_{2-2i-j, \tau}}+T_{\mathbb{O}_{i+j, \tau}})+k(v-v^{-1})T_{\mathbb{O}_{id,\tau}}.\]

\end{proof}
\subsection{The quasi-split case}
At first we classify the $\delta$-conjugacy classes in $\widetilde{W}$. In this section we set $\cdot$ to be the usual $\widetilde{W}$-conjugation (i.e. $x\cdot y=xyx^{-1}$) and $\cdot_{\delta}$ be the $\delta$-conjugation (i.e., $x\cdot y=xy\delta(x)^{-1}$). Let\begin{align*}\mathbb{O}_{0,\delta}&=\{t^{k(\alpha_1+2\alpha_2)}s_2s_1,\ t^{k(2\alpha_1+\alpha_2)}s_1s_2,\ t^{k(\alpha_1-\alpha_2)}\mid\ k\in\mathbb{Z}\}\\
\sqcup&\{(\tau\cdot t^{k(\alpha_1+2\alpha_2)}s_2s_1)\tau^2,\ (\tau\cdot t^{k(2\alpha_1+\alpha_2)}s_1s_2)\tau^2,\ (\tau\cdot t^{k(\alpha_1-\alpha_2)})\tau^2\ \mid\ k\in\mathbb{Z}\}\\
\sqcup&\{(\tau^2\cdot t^{k(\alpha_1+2\alpha_2)}s_2s_1)\tau,\ (\tau^2\cdot t^{k(2\alpha_1+\alpha_2)}s_1s_2)\tau,\ (\tau^2\cdot t^{k(\alpha_1-\alpha_2)})\tau\ \mid\ k\in\mathbb{Z}\};
\end{align*}
\begin{align*}\mathbb{O}_{1,\delta}=&\{t^{\lambda}s_1,\ t^{\lambda}s_2\ \mid\ \lambda\in Q\}\sqcup\{(\tau\cdot t^{\lambda}s_1)\tau^2,\ (\tau\cdot t^{\lambda}s_2)\tau^2\ \mid\ \lambda\in Q\}\\
\sqcup&\{(\tau^2\cdot t^{\lambda}s_1)\tau,\ (\tau^2\cdot t^{\lambda}s_2)\tau\ \mid\ \lambda\in Q\};
\end{align*}
\begin{align*}\mathbb{O}'_{1,\delta}=&\{t^{k\alpha_1+(2i+1)\alpha_2}s_1s_2s_1,\ t^{(2k+1)\alpha_1+2i\alpha_2}s_1s_2s_1\ \mid\ k,\ i\in\mathbb{Z}\}\\
\sqcup&\{(\tau\cdot t^{k\alpha_1+(2i+1)\alpha_2}s_1s_2s_1)\tau^2,\ (\tau\cdot t^{(2k+1)\alpha_1+2i\alpha_2}s_1s_2s_1)\tau^2\ \mid\ k,\ i\in\mathbb{Z}\}\\
\sqcup&\{(\tau^2\cdot t^{k\alpha_1+(2i+1)\alpha_2}s_1s_2s_1)\tau,\ (\tau^2\cdot t^{(2k+1)\alpha_1+2i\alpha_2}s_1s_2s_1)\tau\ \mid\ k,\ i\in\mathbb{Z}\};
\end{align*}
\begin{align*}\mathbb{O}_{3,\delta}=&\{t^{2k\alpha_1+2i\alpha_2}s_1s_2s_1\  \mid\ k,\ i\in\mathbb{Z}\}\sqcup\{(\tau\cdot t^{2k\alpha_1+2i\alpha_2}s_1s_2s_1)\tau^2\ \mid\ k,\ i\in\mathbb{Z}\}\\
\sqcup&\{(\tau^2\cdot t^{2k\alpha_1+2i\alpha_2}s_1s_2s_1)\tau\ \mid\ k,\ i\in\mathbb{Z}\};
\end{align*}
For $m\in\mathbb{N}_+$, we set \[\epsilon(m)=\left\{
                                                             \begin{array}{ll}
                                                             1, &\ m\ \ \ \ \text{odd}\\
                                                             0, &\ m\ \ \ \ \text{even}.
                                                             \end{array}
                                                           \right.\]

\begin{align*}\mathbb{O}_{2m,\delta}=&\{t^{(k-\lfloor\frac{m}{2}\rfloor)\alpha_1+(2k+\epsilon(m))\alpha_2}s_2s_1,\ t^{(k+\lfloor\frac{m+1}{2}\rfloor)\alpha_1+(2k+\epsilon(m))\alpha_2}s_2s_1\mid\ k\in\mathbb{Z}\}\\
\sqcup&\{t^{(2k+\epsilon(m))\alpha_1+(k+\lfloor\frac{m+1}{2}\rfloor)\alpha_2}s_1s_2,\ t^{(2k+\epsilon(m))\alpha_1+(k-\lfloor\frac{m}{2}\rfloor)\alpha_2}s_1s_2\mid\ k\in\mathbb{Z}\}\\
\sqcup&\{t^{k\alpha_1+(m-k)\alpha_2},\ t^{(k-m)\alpha_1-k\alpha_2}\mid\ k\in\mathbb{Z}\}\\
\sqcup&\{(\tau\cdot t^{(k-\lfloor\frac{m}{2}\rfloor)\alpha_1+(2k+\epsilon(m))\alpha_2}s_2s_1)\tau^2,\ (\tau\cdot t^{(k+\lfloor\frac{m+1}{2}\rfloor)\alpha_1+(2k+\epsilon(m))\alpha_2}s_2s_1)\tau^2\mid\ k\in\mathbb{Z}\}\\
\sqcup&\{(\tau\cdot t^{(2k+\epsilon(m))\alpha_1+(k+\lfloor\frac{m+1}{2}\rfloor)\alpha_2}s_1s_2)\tau^2,\ (\tau\cdot t^{(2k+\epsilon(m))\alpha_1+(k-\lfloor\frac{m}{2}\rfloor)\alpha_2}s_1s_2)\tau^2\mid\ k\in\mathbb{Z}\}\\
\sqcup&\{(\tau\cdot t^{k\alpha_1+(m-k)\alpha_2})\tau^2,\ (\tau\cdot t^{(k-m)\alpha_1-k\alpha_2})\tau^2\mid\ k\in\mathbb{Z}\}\\
\sqcup&\{(\tau^2\cdot t^{(k-\lfloor\frac{m}{2}\rfloor)\alpha_1+(2k+\epsilon(m))\alpha_2}s_2s_1)\tau,\ (\tau^2\cdot t^{(k+\lfloor\frac{m+1}{2}\rfloor)\alpha_1+(2k+\epsilon(m))\alpha_2}s_2s_1)\tau\mid\ k\in\mathbb{Z}\}\\
\sqcup&\{(\tau^2\cdot t^{(2k+\epsilon(m))\alpha_1+(k+\lfloor\frac{m+1}{2}\rfloor)\alpha_2}s_1s_2)\tau,\ (\tau^2\cdot t^{(2k+\epsilon(m))\alpha_1+(k-\lfloor\frac{m}{2}\rfloor)\alpha_2}s_1s_2)\tau\mid\ k\in\mathbb{Z}\}\\
\sqcup&\{(\tau^2\cdot t^{k\alpha_1+(m-k)\alpha_2})\tau,\ (\tau^2\cdot t^{(k-m)\alpha_1-k\alpha_2})\tau\mid\ k\in\mathbb{Z}\}.
\end{align*}
\begin{rmk}\label{delta}
Let $\tilde{w}\in\widetilde{W}$ then $\tilde{w}=w_a\tau'$ where $w_a\in W_a$ and $\tau'\in\Omega$. Thus $\tilde{w}\tilde{\thicksim}w_a$. In the following, to consider $\tilde{w}\in\widetilde{W}$, it is sufficient to consider $w_a\in W_a$.
\end{rmk}
\begin{thm}\label{conjdelta}
$\mathbb{O}_{0,\delta},\ \mathbb{O}_{1,\delta},\ \mathbb{O}'_{1,\delta},\ \mathbb{O}_{3,\delta},\ \mathbb{O}_{2m,\delta}$ where $m\in\mathbb{N}_+$ form all the $\delta$-conjugacy classes of $\widetilde{W}$. Moreover, $$\ell(\mathbb{O}_{0,\delta})=0,\ \ \ell(\mathbb{O}_{3,\delta})=3,$$\[\ell(\mathbb{O}_{1,\delta})=\ell(\mathbb{O}'_{1,\delta})=1,\]\[\ell(\mathbb{O}_{2m,\delta})=2m.\]
\end{thm}
\begin{proof}
By the definitions of $\mathbb{O}_{0,\delta},\ \mathbb{O}_{1,\delta},\ \mathbb{O}'_{1,\delta},\ \mathbb{O}_{3,\delta},\ \mathbb{O}_{2m,\delta}$ for $m\in\mathbb{N}_+$, we have $\widetilde{W}=\mathbb{O}_{0,\delta}\sqcup\mathbb{O}_{1,\delta}\sqcup\mathbb{O}'_{1,\delta}\sqcup\mathbb{O}_{3,\delta}\sqcup\sqcup_{m\in\mathbb{N}_+}\mathbb{O}_{2m,\delta}$. And the theorem is proved together with Lemma \ref{oi}.
\end{proof}
\begin{lem}\label{oi}
(1) $\mathbb{O}_{0,\delta}$ is the $\delta$-conjugacy class of $\widetilde{W}$ with minimal length 0;\\
(2) $\mathbb{O}_{1,\delta}$ and $\mathbb{O}'_{1,\delta}$ are the $\delta$-conjugacy classes of $\widetilde{W}$ with minimal length 1;\\
(3) $\mathbb{O}_{3,\delta}$ is the $\delta$-conjugacy class of $\widetilde{W}$ with minimal length 3;\\
(4) If $m\in\mathbb{N}_+$, then $\mathbb{O}_{2m,\delta}$ is the $\delta$-conjugacy class of $\widetilde{W}$ with minimal length $2m$.
\end{lem}
\begin{proof}
We prove (1) and others are similar. We check easily that for any $\tilde{w}\in\widetilde{W}$, then $\tilde{w}\cdot_{\delta}\mathbb{O}_{0,\delta}\subset\mathbb{O}_{0,\delta}$. Then we need to show that for any element $\tilde{w}\in\mathbb{O}_{0,\delta}$, $\tilde{w}$ is $\delta$-conjugate to $id$. By \ref{delta}, it is sufficient to consider that $\tilde{w}\in\{t^{k(\alpha_1+2\alpha_2)}s_2s_1,\ t^{k(2\alpha_1+\alpha_2)}s_1s_2,\ t^{k(\alpha_1-\alpha_2)}\mid\ k\in\mathbb{Z}\}$ and we use induction on the length. It $\tilde{w}=id$, there is nothing to prove. If $\tilde{w}'\in\{t^{k(\alpha_1+2\alpha_2)}s_2s_1,\ t^{k(2\alpha_1+\alpha_2)}s_1s_2,\ t^{k(\alpha_1-\alpha_2)}\mid\ k\in\mathbb{Z}\}$, and $\ell(\tilde{w}')<\ell(\tilde{w})$, then $\tilde{w}'$ is $\delta$-conjugate to $id$. If $\tilde{w}=t^{k(\alpha_1+2\alpha_2)}s_2s_1$ or $t^{k(2\alpha_1+\alpha_2)}s_1s_2$ where $k\in\mathbb{N}_+$, and we set $\tilde{w}_1=s_0\tilde{w}s_0$, then $\ell(\tilde{w}_1)=\ell(\tilde{w})-2$, thus $\tilde{w}_1$ is $\delta$-conjugate to $id$, so is $\tilde{w}$. If $\tilde{w}=t^{k(\alpha_1-\alpha_2)}$ or $t^{-k(\alpha_1+2\alpha_2)}s_2s_1$ where $k\in\mathbb{N}_+$, and we set $\tilde{w}_1=s_2\tilde{w}s_1$ then $\ell(\tilde{w}_1)=\ell(\tilde{w})-2$, thus $\tilde{w}_1$ is $\delta$-conjugate to $id$, so is $\tilde{w}$. If $\tilde{w}=t^{k(\alpha_2-\alpha_1)}$ or $t^{-k(2\alpha_1+\alpha_2)}s_1s_2$ where $k\in\mathbb{N}_+$, and we set $\tilde{w}_1=s_1\tilde{w}s_2$ then $\ell(\tilde{w}_1)=\ell(\tilde{w})-2$, thus $\tilde{w}_1$ is $\delta$-conjugate to $id$, so is $\tilde{w}$. Hence (1) is proved.
\end{proof}
\begin{thm}\label{TT}
Let $\tilde{w},\ \tilde{w}'\in\mathbb{O}_{i,\delta}$ where $i=0,\ 1$ or $2m$ where $m\in N_+$. If $\ell(\tilde{w})=\ell(\tilde{w}')$, then $T_{\tilde{w}}\equiv T_{\tilde{w}'}$.
\end{thm}
\begin{proof}
We use induction on the length $\ell$. We first prove that $\tilde{w},\ \tilde{w}'\in\mathbb{O}_{1,\delta}$. If $\ell(\tilde{w})=\ell(\tilde{w}')=1$, obviously. Now if $\ell(\tilde{w})=\ell(\tilde{w}')=2k+1$ where $k\in\mathbb{N}_+$ and we assume for any $\tilde{w}_1,\ \tilde{w}'_1\in\mathbb{O}_{1,\delta}$ and $\ell(\tilde{w}_1)=\ell(\tilde{w}'_1)<2k+1$ then $T_{\tilde{w}_1}\equiv T_{\tilde{w}'_1}$. It is sufficient to show that if $\tilde{w}_1=s_i\tilde{w}\delta(s_i)$ and $\ell(\tilde{w}_1)=\ell(\tilde{w})-2$, then $s_i\tilde{w}$ is a minimal length element in $\mathbb{O}_{2k,\delta}$. If $k=3j+3$, then $\tilde{w}=t^{(j+2)\alpha_1-j\alpha_2+i(\alpha_1+2\alpha_2)}s_1$ where $0\leqslant i\leqslant j$. If $k=3j+4$, then $\tilde{w}=t^{(j+2)\alpha_1-(j+1)\alpha_2+i(\alpha_1+2\alpha_2)}s_1$ where $0\leqslant i\leqslant j+1$. If $k=3j+5$, then $\tilde{w}=t^{(j+4)\alpha_1-(j+1)\alpha_2+i(\alpha_1+2\alpha_2)}s_1$ where $0\leqslant i\leqslant j+1$. Thus $s_0\tilde{w}\thicksim\tilde{w}s_0$ and $\tilde{w}s_0$ is a minimal length element of $\mathbb{O}_{2k,\delta}$. Similar argument for other cases. For $\tilde{w}$ and $\tilde{w}'$ are contained in $\mathbb{O}_{0,\delta}$. For $k\in\mathbb{N}_+$, and if $\ell(\tilde{w})=6k-4$, then $\tilde{w}=t^{(1-k)(\alpha_1+2\alpha_2)}s_2s_1$ or $\tilde{w}=t^{(1-k)(2\alpha_1+\alpha_2)}s_1s_2$. If $\ell(\tilde{w})=6k-2$, then $\tilde{w}=t^{k(\alpha_1+2\alpha_2)}s_2s_1$ or $\tilde{w}=t^{k(2\alpha_1+\alpha_2)}s_1s_2$. If $\ell(\tilde{w})=6k$, then $\tilde{w}=t^{k(\alpha_1-\alpha_2)}$ or $\tilde{w}=t^{k(\alpha_2-\alpha_1)}$. By symmetry, we know $T_{t^{(1-k)(\alpha_1+2\alpha_2)}s_2s_1}\equiv T_{t^{(1-k)(2\alpha_1+\alpha_2)}s_1s_2}$, $T_{t^{k(\alpha_1+2\alpha_2)}s_2s_1}\equiv T_{t^{k(2\alpha_1+\alpha_2)}s_1s_2}$ or $T_{t^{k(\alpha_1-\alpha_2)}}\equiv T_{t^{k(\alpha_2-\alpha_1)}}$. Thus for $\mathbb{O}_{0,\delta}$ is proved. If $\tilde{w}\in\mathbb{O}_{2m,\delta}$ and $\ell(\tilde{w})=2m+2k$ for some $k\in\mathbb{N}_+$. If $\tilde{w}=t^{\lambda}s_2s_1$, then we have $\tilde{w}\xrightarrow[]{s_0}_{\delta}s_0\tilde{w}s_0$ or $\tilde{w}\xrightarrow[]{s_2}_{\delta}s_2\tilde{w}s_1$. In either case, $T_{\tilde{w}}\equiv(v-v^{-1})T_{s_i\tilde{w}}+T_{s_i\tilde{w}\delta(s_i)}$ where $i=0$ or $2$, $s_i\tilde{w}\in\mathbb{O}_{1,\delta}$. Similar argument for $\tilde{w}'$ where $\ell(\tilde{w})=\ell(\tilde{w}')$. We have some $i$ such that $T_{\tilde{w}'}\equiv(v-v^{-1})T_{s_i'\tilde{w}'}+T_{s_i'\tilde{w}'\delta(s_i')}$ and $s_i'\tilde{w}'\in\mathbb{O}_{1,\delta}$. Then by the proof for $\mathbb{O}_{1,\delta}$ and induction we have $T_{\tilde{w}}\equiv T_{\tilde{w}'}$.
\end{proof}
\begin{prop}\label{cpdelta1}
If $\tilde{w}\in\mathbb{O}_{1,\delta}$ and $\ell(\tilde{w})=2k+1$ where $k\in\mathbb{N}$. Then\[f_{\tilde{w},\mathbb{O}}=\left\{
                                                             \begin{array}{ll}
                                                               (v-v^{-1}), & \mathbb{O}=\mathbb{O}_{2j,\delta}\ (1\leqslant j\leqslant k)\\
                                                               1, & \mathbb{O}=\mathbb{O}_{1,\delta}\\
                                                               0, & otherwise.
                                                             \end{array}
                                                           \right.\]
\end{prop}
\begin{proof}
This follows directly from the proof of Lemma \ref{oi} and Theorem \ref{TT}.
\end{proof}
\begin{prop}\label{cpdelta0}
If $\tilde{w}\in\mathbb{O}_{0,\delta}$ and $\ell(\tilde{w})=2k$ where $k\in\mathbb{N}$. Then\[f_{\tilde{w},\mathbb{O}}=\left\{
                                                             \begin{array}{ll}
                                                               (k-j)(v-v^{-1})^2, & \mathbb{O}=\mathbb{O}_{2j,\delta}\ (1\leqslant j\leqslant k-1)\\
                                                               k(v-v^{-1}), & \mathbb{O}=\mathbb{O}_{1,\delta}\\
                                                               1, & \mathbb{O}=\mathbb{O}_{0,\delta}\\
                                                               0, & otherwise.
                                                             \end{array}
                                                           \right.\]
\end{prop}
\begin{proof}
For $\tilde{w}\in\mathbb{O}_{0,\delta}$ and $\ell(\tilde{w})=2k$ where $k\in\mathbb{N}$. If $k=0$, it is obvious. For $k>0$, there exist $s_i$ with $i=0,\ 1$ or $2$ such that\[T_{\tilde{w}}\equiv(v-v^{-1})T_{s_i\tilde{w}}+T_{s_i\tilde{w}\delta(s_i)}.\]Here $s_i\tilde{w}\in\mathbb{O}_{1,\delta}$ with $\ell(s_i\tilde{w})=2k-1$ and $\ell(s_i\tilde{w}\delta(s_i))=2k-2$. Use Proposition \ref{cpdelta1} and calculate inductively we have
\begin{align*}T_{\tilde{w}}\equiv&\sum_{j=1}^{k-1}(k-j)(v-v^{-1})^2T_{\mathbb{O}_{2j,\delta}}+k(v-v^{-1})T_{\mathbb{O}_{1,\delta}}+T_{\mathbb{O}_{0,\delta}}.
\end{align*}
\end{proof}
\begin{prop}\label{cpdelta2m}
Let $m\in\mathbb{N}_+$, if $\tilde{w}\in\mathbb{O}_{2m,\delta}$ and $\ell(\tilde{w})=2m+2k$ where $k\in\mathbb{N}$. Then\[f_{\tilde{w},\mathbb{O}}=\left\{
                                                             \begin{array}{ll}
                                                               (k-j)(v-v^{-1})^2, & \mathbb{O}=\mathbb{O}_{2m+2j,\delta}\ (1\leqslant j\leqslant k-1)\\
                                                               k(v-v^{-1})^2+1, & \mathbb{O}=\mathbb{O}_{2m,\delta}\\
                                                               k(v-v^{-1})^2, & \mathbb{O}=\mathbb{O}_{2m-2j,\delta}\ (1\leqslant j\leqslant m-1)\\
                                                               k(v-v^{-1}), & \mathbb{O}=\mathbb{O}_{1,\delta}\\
                                                               0, & otherwise.
                                                             \end{array}
                                                           \right.\]
\end{prop}
\begin{proof}
The proof is quite similar to that of Proposition \ref{cpdelta0} and we will omit it.
\end{proof}
If $\tilde{w},\ \tilde{w}'\in\mathbb{O}'_{1,\delta}$ lie in the critical strip and $\ell(\tilde{w})=\ell(\tilde{w}')$, then $T_{\tilde{w}}\equiv T_{\tilde{w}'}$. Moreover, if $\tilde{w}\in\mathbb{O}'_{1,\delta}$ with $\ell(\tilde{w})=2k+1$ where $k\in\mathbb{N}$, then\[f_{\tilde{w},\mathbb{O}}=\left\{
                                                             \begin{array}{ll}
                                                               (v-v^{-1}), & \mathbb{O}=\mathbb{O}_{2j,\delta},\ 1\leqslant j\leqslant k\\
                                                               1, & \mathbb{O}=\mathbb{O}'_{1,\delta}\\
                                                               0, & otherwise.
                                                             \end{array}
                                                           \right.\]
If $\tilde{w}\in\mathbb{O}'_{1,\delta}$ lies in the shrunkun Weyl chambers, we do not have a uniform formula for those class polynomials. For example, if $\tilde{w}=t^{2\alpha_1+\alpha_2}s_1s_2s_1$, then\[f_{\tilde{w},\mathbb{O}}=\left\{
                                                             \begin{array}{ll}
                                                               (v-v^{-1})^3+2(v-v^{-1}), & \mathbb{O}=\mathbb{O}_{2,\delta}\\
                                                               (v-v^{-1})^2, & \mathbb{O}=\mathbb{O}_{1,\delta}\\
                                                               1, & \mathbb{O}=\mathbb{O}'_{1,\delta}\\
                                                               0, & otherwise.
                                                             \end{array}
                                                           \right.\]
If $\ell(\tilde{w})=7$ and $\tilde{w}=t^{\alpha_1-\alpha_2}s_1s_2s_1$, then\[f_{\tilde{w},\mathbb{O}}=\left\{
                                                             \begin{array}{ll}
                                                               (v-v^{-1}), & \mathbb{O}=\mathbb{O}_{6,\delta}\\
                                                               (v-v^{-1})^3+2(v-v^{-1}), & \mathbb{O}=\mathbb{O}_{2,\delta}\\
                                                               (v-v^{-1})^2, & \mathbb{O}=\mathbb{O}_{1,\delta}\\
                                                               1, & \mathbb{O}=\mathbb{O}'_{1,\delta}\\
                                                               0, & otherwise.
                                                             \end{array}
                                                           \right.\]
If $\ell(\tilde{w})=7$ and $\tilde{w}=t^{-\alpha_1-\alpha_2}s_1s_2s_1$, then\[f_{\tilde{w},\mathbb{O}}=\left\{
                                                             \begin{array}{ll}
                                                                (v-v^{-1})^3+(v-v^{-1}), & \mathbb{O}=\mathbb{O}_{4,\delta}\\
                                                               (v-v^{-1})^3+2(v-v^{-1}), & \mathbb{O}=\mathbb{O}_{2,\delta}\\
                                                               2(v-v^{-1})^2, & \mathbb{O}=\mathbb{O}_{1,\delta}\\
                                                               1, & \mathbb{O}=\mathbb{O}'_{1,\delta}\\
                                                               0, & otherwise.
                                                             \end{array}
                                                           \right.\]
 Inductively, if $f_{\tilde{w},\mathbb{O}}\neq0$, we deduce that $\deg f_{\tilde{w},\mathbb{O}_{1,\delta}}=2$ and $\deg f_{\tilde{w},\mathbb{O}_{2m,\delta}}=3$ or $1$ for certain $m$. Moreover, $f_{\tilde{w},\mathbb{O}'_{1,\delta}}=1$.

Similar argument for $\tilde{w}\in\mathbb{O}_{3,\delta}$.

\section{Applications}\label{Appl}
\subsection{Affine Deligne-Lusztig varieties of basic elements}\label{ADLVbasic}
First we assume that $b\in PGL_3(L)$ and $\tilde{w}\in\widetilde{W}$.
\begin{thm}\label{empty}
\begin{enumerate}
\item If $b=1$, then the affine Deligne-Lusztig variety $X_{\tilde{w}}(b)\neq\emptyset$ if and only if $\tilde{w}$ satisfies one of the following conditions:
\begin{enumerate}
\item $\tilde{w}=id$
\item $\tilde{w}\in\mathbb{O}_1$ or $\tilde{w}\in\mathbb{O}_2$
\item $\tilde{w}\in\mathbb{C}_i$ or $\mathbb{C}'_i$, where $i\in\mathbb{N}_+$ and $\ell(\tilde{w})\geqslant6i+3$.
\end{enumerate}
\item If $b=\tau$, then the affine Deligne-Lusztig variety $X_{\tilde{w}}(b)\neq\emptyset$ if and only if $\tilde{w}$ satisfies one of the following conditions:
\begin{enumerate}
\item $\tilde{w}\in\mathbb{O}_{id,\tau}$
\item $\tilde{w}\in\mathbb{O}_{i,\tau}$, where $i\in\mathbb{N}_+$ and $\ell(\tilde{w})\geqslant6i-1$
\item $\tilde{w}\in\mathbb{O}_{1-i,\tau}$, where $\ell(\tilde{w})\geqslant6i+1$.
\end{enumerate}
\end{enumerate}
\end{thm}
\begin{proof}
Following the ``Dimension$=$Degree" Theorem \ref{DimDeg}, we go back to \S \ref{Class} to check those nonzero class polynomials.
\end{proof}
Once the affine Deligne-Lusztig variety $X_{\tilde{w}}(b)\neq\emptyset$, we do have a neat dimension formula.
\begin{thm}\label{dimen}
\begin{enumerate}
\item If $b=1$ and $X_{\tilde{w}}(b)\neq\emptyset$, then
\[\dim X_{\tilde{w}}(b)=\left\{
 \begin{array}{ll}
 0, & \tilde{w}=Id\\
 1, & \tilde{w}\in\mathbb{O}_1\ and\ \ell(\tilde{w})=1\\
 \frac{\ell(\tilde{w})}{2}+1, & \tilde{w}\in\mathbb{O}_2\\
 \frac{\ell(\tilde{w})+3}{2}, & \tilde{w}\in\mathbb{O}_1\ with\ \ell(\tilde{w})>1,\ or\\
 & \tilde{w}\in\mathbb{C}_{i}\ or\ \mathbb{C}'_{i}\ for\ i\in\mathbb{N}_+.
  \end{array}
  \right.\]
\item If $b=\tau$ and $X_{\tilde{w}}(b)\neq\emptyset$, then\[\dim X_{\tilde{w}}(b)=\left\{
 \begin{array}{ll}
 \frac{\ell(\tilde{w})}{2}, & \tilde{w}\in\mathbb{O}_{id,\tau}\\
 \frac{\ell(\tilde{w})+1}{2}, & \tilde{w}\in\mathbb{O}_{i,\tau}\ where\ i\in\mathbb{Z}.
  \end{array}
  \right.\]
\end{enumerate}
\end{thm}
\begin{proof}
By the ``Dimension$=$Degree" theorem, if $b=1$, then $$\dim X_{\tilde{w}}(1)= \frac{1}{2}\max\{\ell(\tilde{w})+\deg f_{\tilde{w},Id},\ \ell(\tilde{w})+1+\deg f_{\tilde{w},\mathbb{O}_1},\ \ell(\tilde{w})+2+\deg f_{\tilde{w},\mathbb{O}_2}\}.$$We check the degree of those class polynomials in \S \ref{split1}, and the theorem is proved.
\end{proof}
\begin{rmk}
Symmetrically, we have a similar description of the emptiness/nonemptiness pattern and dimension formula of $X_{\tilde{w}}(b)$ for $\tilde{w}$, $b=\tau^2$.
\end{rmk}
In the following, the proofs of theorems of the emptiness/nonemptiness pattern and dimension formula are similar to that of Theorem \ref{empty} and \ref{dimen}, and we will omit them.
Now we assume that $b\in U_3(L)$.
\begin{thm}
If $b$ is basic, then $X_{\tilde{w}}(b)\neq\emptyset$ if and only if $\tilde{w}\not \in\bigsqcup_{m\in\mathbb{N}_+}\mathbb{O}^{min}_{2m,\delta}$, where $\mathbb{O}^{min}_{2m,\delta}$ is the set of minimal length elements of $\mathbb{O}_{2m,\delta}$.
\end{thm}
\begin{thm}
If $b=1$ and if $X_{\tilde{w}}(b)\neq\emptyset$, then\[\dim X_{\tilde{w}}(b)=\left\{
 \begin{array}{ll}
 0, & \tilde{w}\in\mathbb{O}_{0,\delta}\ and\ \ell(\tilde{w})=0\\
\frac{1}{2}(\ell(\tilde{w})+1), & \tilde{w}\in\mathbb{O}'_{1,\delta}\ and\ lies\ in\ critical\ strips,\ or\ \tilde{w}\in\mathbb{O}_{1,\delta}\\
\frac{1}{2}(\ell(\tilde{w})+2), & \tilde{w}\in\mathbb{O}_{0,\delta}\ with\ \ell(\tilde{w})>0,\ or\ \tilde{w}\in \mathbb{O}_{2m,\delta}\ where\ m\in\mathbb{N}_+\\
\frac{1}{2}(\ell(\tilde{w})+3), & \tilde{w}\in\mathbb{O}'_{1,\delta}\ corresponds\ to\ shrunkun\ Weyl\ chambers,\ or\\
& \tilde{w}\in\mathbb{O}_{3,\delta}.
  \end{array}
  \right.\]
\end{thm}
\subsection{Affine Deligne-Lusztig varieties of nonbasic elements}\label{ADLVnonbasic}
By the ``Dimension$=$Degree" theorem, for any $b\in G(L)$ and $\tilde{w}\in\widetilde{W}$, the affine Deligne-Lusztig variety $X_{\tilde{w}}(b)\neq\emptyset$ if and only if the corresponding class polynomial is nonzero. Thus for nonbasic $b$, we check the corresponding class polynomial in Section \S\ref{Class} and we know the emptiness/nonemptiness pattern. For example, if $b\in PGL_3(L)$ corresponds to $\mathbb{O}_{\lambda_0}$ (i.e. $f(b)=f(\mathbb{O}_{\lambda_0})$) where $\lambda_0\in P_+\cap Q_{sh}$, then $X_{\tilde{w}}(b)\neq\emptyset$ if and only if $f_{\tilde{w},\mathbb{O}_{\lambda_0}}\neq0$. If $b\in U_3(L)$ corresponds to $\mathbb{O}_{2m_0,\delta}$ where $m_0\in\mathbb{N}_+$, then $X_{\tilde{w}}(b)\neq\emptyset$ if and only if $f_{\tilde{w},\mathbb{O}_{2m_0,\delta}}\neq0$.
\begin{thm}
\begin{enumerate}
\item If $b\in PGL_3(L)$ corresponds to $\mathbb{O}_{\lambda_0}$ where $\lambda_0\in P_+\cap Q_{sh}$ and $X_{\tilde{w}}(b)\neq\emptyset$, then
\[\dim X_{\tilde{w}}(b)=\left\{
 \begin{array}{ll}
 \frac{1}{2}(\ell(\tilde{w})+\ell(\mathbb{O}_{\lambda_0}))-\langle\lambda_0,\ 2\rho\rangle, & \tilde{w}\in\mathbb{O}_{\lambda_0}\\
 \frac{1}{2}(\ell(\tilde{w})+\ell(\mathbb{O}_{\lambda_0})+1)-\langle\lambda_0,\ 2\rho\rangle, & \tilde{w}\in\mathbb{O}_1\ and\ \lambda_0=\frac{\ell(\tilde{w})-1}{4}(\alpha_1+\alpha_2)\ or\\
 & \tilde{w}\in\mathbb{C}_{i}\ or\ \mathbb{C}'_{i}\ i\in\mathbb{N}_+\ with\ \ell(\tilde{w})\leqslant6i+1\\
 \frac{1}{2}(\ell(\tilde{w})+\ell(\mathbb{O}_{\lambda_0})+2)-\langle\lambda_0,\ 2\rho\rangle, & \tilde{w}\in\mathbb{O}_2\\
 \frac{1}{2}(\ell(\tilde{w})+\ell(\mathbb{O}_{\lambda_0})+3)-\langle\lambda_0,\ 2\rho\rangle, & \tilde{w}\in\mathbb{O}_1\ and\ \lambda_0\neq\frac{\ell(\tilde{w})-1}{4}(\alpha_1+\alpha_2)\ or\\
 & \tilde{w}\in\mathbb{C}_{i}\ or\ \mathbb{C}'_{i}\ i\in\mathbb{N}_+\ with\ \ell(\tilde{w})>6i+1\\
  \end{array}
  \right.\]
\item If $b\in PGL_3(L)$ corresponds to $\mathbb{O}_{i_0(\alpha_1+2\alpha_2)}$ where $i_0\in\mathbb{N}_+$ and $X_{\tilde{w}}(b)\neq\emptyset$, then
\[\dim X_{\tilde{w}}(b)=\left\{
 \begin{array}{ll}
 \frac{1}{2}(\ell(\tilde{w})+6i_0+1)-t_0, & \tilde{w}\in\mathbb{C}_{i}\ or\ \mathbb{C}'_{i}\ with\ \ell(\tilde{w})\leqslant6i+1,\ i\in\mathbb{N}_+\\
 \frac{1}{2}(\ell(\tilde{w})+6i_0+2)-t_0, & \tilde{w}\in\mathbb{O}_2\\
 \frac{1}{2}(\ell(\tilde{w})+6i_0+3)-t_0, & \tilde{w}\in\mathbb{O}_1\ or\ \tilde{w}\in\mathbb{C}_{i},\ \mathbb{C}'_{i}\ with\ \ell(\tilde{w})>6i+1.
  \end{array}
  \right.\]Where $t_0=i_0\langle\alpha_1+2\alpha_2,\ 2\rho\rangle$.\\
 Symmetrically, if $b$ corresponds to $\mathbb{O}_{i_0(2\alpha_1+\alpha_2)}$, then there is a similar dimension formula.\\
\item If $b\in PGL_3(L)$ corresponds to $\mathbb{C}_{i_0}$ where $i_0\in\mathbb{N}_+$ and $X_{\tilde{w}}(b)\neq\emptyset$, then\[\dim X_{\tilde{w}}(b)=\left\{
 \begin{array}{ll}
 \frac{1}{2}(\ell(\tilde{w})+\ell(\mathbb{C}_{i_0}))-t'_0, & \tilde{w}\in\mathbb{C}_{i}\ or\ \mathbb{C}'_{i}\ with\ \ell(\tilde{w})\leqslant6i+1,\ i\in\mathbb{N}_+\\
 \frac{1}{2}(\ell(\tilde{w})+\ell(\mathbb{C}_{i_0})+1)-t'_0, & \tilde{w}\in\mathbb{O}_2\\
 \frac{1}{2}(\ell(\tilde{w})+\ell(\mathbb{C}_{i_0})+2)-t'_0, & \tilde{w}\in\mathbb{O}_1\ or\ \tilde{w}\in\mathbb{C}_{i},\ \mathbb{C}'_{i}\ with\ \ell(\tilde{w})>6i+1,\ i\in\mathbb{N}_+.
  \end{array}
  \right.\]Where $t'_0=i_0\langle\alpha_1+2\alpha_2,\ \rho\rangle$.\\
Symmetrically, when $b$ corresponds to $\mathbb{C}'_{i_0}$ we have a similar dimension formula.
\end{enumerate}
\end{thm}
\begin{thm}
\begin{enumerate}
\item Let $b\in PGL_3(L)$. If $b$ corresponds to $\mathbb{O}_{\lambda_0,\tau}$ where $\lambda_0\in P_+\cap Q_{sh}$ but $\lambda_0\neq(2i-1)\alpha_1+i\alpha_2$ for $i\in\mathbb{N}_+$ and $X_{\tilde{w}}(b)\neq\emptyset$, then
\[\dim X_{\tilde{w}}(\dot{\tau})=\left\{
 \begin{array}{ll}
\ell(\tilde{w})-\langle\bar{\nu}_{\mathbb{O}_{\lambda_0,\tau}},\ 2\rho\rangle, & \tilde{w}\in\mathbb{O}_{\lambda_0,\tau}\\
 \frac{1}{2}(\ell(\tilde{w})+\ell(\mathbb{O}_{\lambda_0,\tau})+1)-\langle\bar{\nu}_{\mathbb{O}_{\lambda_0,\tau}},\ 2\rho\rangle, & \tilde{w}\in\mathbb{O}_{i,\tau}\ or\ \mathbb{O}_{1-i,\tau}\ and\ \ell(\tilde{w})\leqslant6i-3\\
 \frac{1}{2}(\ell(\tilde{w})+\ell(\mathbb{O}_{\lambda_0,\tau})+2)-\langle\bar{\nu}_{\mathbb{O}_{\lambda_0,\tau}},\ 2\rho\rangle, & \tilde{w}\in\mathbb{O}_{id,\tau}\\
 \frac{1}{2}(\ell(\tilde{w})+\ell(\mathbb{O}_{\lambda_0,\tau})+3)-\langle\bar{\nu}_{\mathbb{O}_{\lambda_0,\tau}},\ 2\rho\rangle, & \tilde{w}\in\mathbb{O}_{i,\tau}\ or\ \mathbb{O}_{1-i,\tau}\ and\ \ell(\tilde{w})>6i-3.
  \end{array}
  \right.\]
\item Let $b\in PGL_3(L)$. If $b$ corresponds to $\mathbb{O}_{i_0(\alpha_1+2\alpha_2),\tau}$ where $i_0\in\mathbb{N}_+$ and $X_{\tilde{w}}(b)\neq\emptyset$, then\[\dim X_{\tilde{w}}(b)=\left\{
 \begin{array}{ll}
 \frac{1}{2}(\ell(\tilde{w})+\ell(\mathbb{O}_{2i_0,\tau}))-l_0, & \tilde{w}\in\mathbb{O}_{i,\tau},\ i\in\mathbb{N}_+\ and\ \ell(\tilde{w})\leqslant6i-3\\
 \frac{1}{2}(\ell(\tilde{w})+\ell(\mathbb{O}_{2i_0,\tau})+1)-l_0, & \tilde{w}\in\mathbb{O}_{id,\tau}\\
 \frac{1}{2}(\ell(\tilde{w})+\ell(\mathbb{O}_{2i_0,\tau})+2)-l_0, & \tilde{w}\in\mathbb{O}_{i,\tau}\ and\ \ell(\tilde{w})>6i-3\ or\ \tilde{w}\in\mathbb{O}_{1-i,\tau}.
  \end{array}
  \right.\]Where $l_0=\langle(i_0-\frac{1}{3})(\alpha_1+2\alpha_2),\ 2\rho\rangle$.\\
\item Let $b\in PGL_3(L)$. If $b$ corresponds to $\mathbb{O}_{(2i_0-1)\alpha_1+i_0\alpha_2),\tau}$ where $i_0\in\mathbb{N}_+$ and $X_{\tilde{w}}(b)\neq\emptyset$, then\[\dim X_{\tilde{w}}(b)=\left\{
 \begin{array}{ll}
 \frac{1}{2}(\ell(\tilde{w})+\ell(\mathbb{O}_{2(1-i_0),\tau}))-l_1, & \tilde{w}\in\mathbb{O}_{1-i,\tau},\ i\in\mathbb{N}_+\ and\ \ell(\tilde{w})\leqslant6i-1\\
 \frac{1}{2}(\ell(\tilde{w})+\ell(\mathbb{O}_{2(1-i_0),\tau})+1)-l_1, & \tilde{w}\in\mathbb{O}_{id,\tau}\\
 \frac{1}{2}(\ell(\tilde{w})+\ell(\mathbb{O}_{2(1-i_0),\tau})+2)-l_1, & \tilde{w}\in\mathbb{O}_{1-i,\tau}\ and\ \ell(\tilde{w})>6i-1\ or\ \tilde{w}\in\mathbb{O}_{i,\tau}.
  \end{array}
  \right.\]Where $l_1=\langle(i_0-\frac{2}{3})(2\alpha_1+\alpha_2),\ 2\rho\rangle$.\\
\item Let $b\in PGL_3(L)$. If $b$ corresponds to $\mathbb{O}_{i_0,\tau}$ where $i_0\in\mathbb{N}_+$. If $X_{\tilde{w}}(b)\neq\emptyset$, then\[\dim X_{\tilde{w}}(b)=\left\{
 \begin{array}{ll}
 \frac{1}{2}(\ell(\tilde{w})+\ell(\mathbb{O}_{i_0,\tau}))-l_2, & \tilde{w}\in\mathbb{O}_{i,\tau}\ i\in\mathbb{N}_+\ and\ \ell(\tilde{w})\leqslant6i-3\\
 \frac{1}{2}(\ell(\tilde{w})+\ell(\mathbb{O}_{i_0,\tau})+1)-l_2, & \tilde{w}\in\mathbb{O}_{id,\tau}\\
 \frac{1}{2}(\ell(\tilde{w})+\ell(\mathbb{O}_{i_0,\tau})+2)-l_2, & \tilde{w}\in\mathbb{O}_{i,\tau}\ and\ \ell(\tilde{w})>6i-3\ or\ \tilde{w}\in\mathbb{O}_{1-i,\tau}.
  \end{array}
  \right.\]Where $l_2=\langle(\frac{i_0}{2}-\frac{1}{3})(\alpha_1+2\alpha_2),\ 2\rho\rangle$.\\
Similarly, if $b$ corresponds to $\mathbb{O}_{1-i_0,\tau}$ and $X_{\tilde{w}}(b)\neq\emptyset$, then\[\dim X_{\tilde{w}}(b)=\left\{
 \begin{array}{ll}
 \frac{1}{2}(\ell(\tilde{w})+\ell(\mathbb{O}_{1-i_0,\tau}))-l_3, & \tilde{w}\in\mathbb{O}_{1-i,\tau}\ i\in\mathbb{N}_+\ and\ \ell(\tilde{w})\leqslant6i-1\\
 \frac{1}{2}(\ell(\tilde{w})+\ell(\mathbb{O}_{1-i_0,\tau})+1)-l_3, & \tilde{w}\in\mathbb{O}_{id,\tau}\\
 \frac{1}{2}(\ell(\tilde{w})+\ell(\mathbb{O}_{1-i_0,\tau})+2)-l_3, & \tilde{w}\in\mathbb{O}_{1-i,\tau}\ and\ \ell(\tilde{w})>6i-1\ or\ \tilde{w}\in\mathbb{O}_{i,\tau}.
  \end{array}
  \right.\] Where $l_3=\langle(\frac{i_0}{2}-\frac{1}{6})(2\alpha_1+\alpha_2),\ 2\rho\rangle$.
\end{enumerate}
\end{thm}
\begin{thm}
Let $b\in U_3(L)$. If $b$ corresponds to $\mathbb{O}_{2m_0,\delta}$ where $m_0\in\mathbb{N}_+$, then $\bar{\nu}_{\mathbb{O}_{2m_0,\delta}}=\frac{m_0}{2}(\alpha_1+\alpha_2)$ and $\langle\bar{\nu}_{\mathbb{O}_{2m_0,\delta}},2\rho\rangle=2m_0$. If $X_{\tilde{w}}(b)\neq\emptyset$, then \[\dim X_{\tilde{w}}(b)=\left\{
 \begin{array}{ll}
 0, & \tilde{w}\in\mathbb{O}_{2m_0,\delta}\ and\ \ell(\tilde{w})=2m_0\\
 \frac{1}{2}(\ell(\tilde{w})+1)-m_0, & \tilde{w}\in\mathbb{O}_{1,\delta},\ or\\
 & \tilde{w}\in\mathbb{O}'_{1,\delta}\ and\ lies\ in\ critical\ strips,\ or \\
 & \tilde{w}\in\mathbb{O}'_{1,\delta}\ or\ \mathbb{O}_{3,\delta}\ with\ \ell(\tilde{w})=2m_0+1\\
 \frac{1}{2}(\ell(\tilde{w})+2)-m_0, & \tilde{w}\in\mathbb{O}_{0,\delta}\ or\ \tilde{w}\in\mathbb{O}_{2m_0,\delta}\ and\ \ell(\tilde{w})>2m_0,\ or\\
 & \tilde{w}\in\mathbb{O}_{2m,\delta}\ where\ m\neq m_0\\
\frac{1}{2}(\ell(\tilde{w})+3)-m_0, & \ell(\tilde{w})>2m_0+1,\ \tilde{w}\in\mathbb{O}_{3,\delta}\ or\\
& \tilde{w}\in\mathbb{O}'_{1,\delta}\ and\ corresponds\ to\ shrunkun\ Weyl\ chambers.
  \end{array}
  \right.\]
\end{thm}
\subsection{Affine Deligne-Lusztig varieties: $GL_3$ and $\mathbb{D}_3^{\times}$ cases}\label{ADLVdiv}
In general, there is a canonical projection $\pi:GL_n(L)\longrightarrow PGL_n(L)$. Let $I_1$ be the Iwahori subgroup of $GL_n(L)$ described in Chapter 2, and let $I_2$ be the corresponding Iwahori subgroup of $PGL_n(L)$. Let $\widetilde{W}_1$ (or $\widetilde{W}_2$, respectively) be the Iwahori-Weyl group. Let $\tilde{w}\in\widetilde{W}_1$ and $b\in GL_n(L)$, such that $\kappa(\tilde{w})=\kappa(b)$. Then the affine Deligne-Lusztig variety $X_{\tilde{w}}(b)$ in the affine flag variety of $GL_n(L)$ is isomorphic to the affine Deligne-Lusztig variety $X_{\pi(\tilde{w})}(\pi(b))$ in the affine flag variety of $PGL_n(L)$ (see \cite{GHN} for more details). Further more, we have $f_{\tilde{w}, \mathbb{O}}=f_{\pi(\tilde{w}),\pi(\mathbb{O})}$. Since we already know the emptiness/nonemptiness pattern and the dimension formulas of the affine Deligne-Lusztig variety for $PGL_3(L)$, we have the same pattern and dimension formulas for $GL_3(L)$.

In a different way, if the automorphism $\sigma: GL_3(L)\longrightarrow GL_3(L)$ is replaced by $\sigma'=Ad(\tau)\circ\sigma: GL_3(L)\longrightarrow GL_3(L)$, then $GL(L)_3^{\sigma'}$ is the group of units of the division algebra $\mathbb{D}_3$ (i.e. $\mathbb{D}_3^{\times}$). As it is described in \cite{GHN}, the affine Deligne-Lusztig variety $X_{\tilde{w}\tau}(\tau)$ for $PGL_3$ is isomorphic to the affine Deligne-Lusztig variety $X_{\tilde{w}}(1)$ for the group $\mathbb{D}_3^{\times}$. Thus we have the same pattern and dimension formulas for $\mathbb{D}_3^{\times}$.
\subsection{Rational points for some affine Deligne-Lusztig varieties}
In general, if $X_{\tilde{w}}(b)\neq\emptyset$, then it has infinitely many irreducible components. However, if $b\in G(L)$ is superbasic, then $X_{\tilde{w}}(b)$ contains only finite irreducible components and the number of rational points is finite.
\begin{prop}[He]\label{rational}
Let $G=PGL_n$ be split over $F$. If $x$ is a superbasic element in $\widetilde{W}$, then for any $\tilde{w}\in\widetilde{W}$\[\sharp X_{\tilde{w}}(x)(\mathbb{F}_q)=nq^{\frac{\ell(\tilde{w})}{2}}f_{\tilde{w},\mathbb{O}}|_{v=\sqrt{q}},\]where $x\in\mathbb{O}$ and it is the conjugacy class of $\widetilde{W}$.
\end{prop}
Applying Proposition \ref{rational} to the group $G=PGL_3$, we can obtain an explicit formula.
\begin{cor}
Let $G=PGL_3$ be split over $F$. Let $\tau\in G(L)$, then for any $\tilde{w}\in\widetilde{W}$ we have\[\sharp X_{\tilde{w}}(\tau)(\mathbb{F}_q)=\left\{
 \begin{array}{ll}
 3q^{\frac{\ell(\tilde{w})}{2}}, & \tilde{w}\in\mathbb{O}_{id,\tau}\\
 3\lceil\frac{\ell(\tilde{w})-6i+3}{4}\rceil q^{\frac{\ell(\tilde{w})-1}{2}}(q-1), & \tilde{w}\in\mathbb{O}_{i,\tau},\ i\in\mathbb{N}_+\ and\ \ell(\tilde{w})\geqslant6i-1\\
 3\lceil\frac{\ell(\tilde{w})-6i+1}{4}\rceil q^{\frac{\ell(\tilde{w})-1}{2}}(q-1), & \tilde{w}\in\mathbb{O}_{1-i,\tau}\ and\ \ell(\tilde{w})\geqslant6i+1\\
 0, & otherwise.
  \end{array}
  \right.\]
\end{cor}
\begin{proof}
It follows directly from Proposition \ref{rational} and Propositions \ref{cpoddtau}, \ref{cpodd1tau}.
\end{proof}
\subsection{The GHKR conjecture}
Before we recall a conjecture of G\"ortz-Haines-Kottwitz-Reuman, let give some notations first. Let $G(L)$ be as in Chapter 2. For any $b\in G(L)$, we denote by $J_b$ the $\sigma$-centralizer of $b$ (i.e. $J_b=\{g\in G(L)\mid gb\sigma(g)^{-1}=b\}$). By definition, the defect of $b$ ($def(b)$) is the $F$-rank of $G$ minus the $F$-rank of $J_b$ (i.e. $def(b)=rk_F G-rk_F J_b$). We follow \cite{Ha} Lemma 4.6 to calculate the defect of $b$.\\
The G\"ortz-Haines-Kottwitz-Reuman conjecture is stated as:
\begin{conj}\label{conjecture}
Let $b\in G(L)$ and $b'$ be a basic element in $G(L)$ such that $\kappa(b)=\kappa(b')$. Then for $\tilde{w}\in\widetilde{W}$ with sufficiently large length, $X_{\tilde{w}}(b)\neq\emptyset$ if and only if $X_{\tilde{w}}(b')\neq\emptyset$. In this case,\[\dim X_{\tilde{w}}(b)=\dim X_{\tilde{w}}(b')-\langle\nu_b,\ \rho\rangle+\frac{1}{2}(def(b')-def(b)).\]
\end{conj}
\begin{thm}\label{ghkr}
The Conjecture \ref{conjecture} is true for the following groups:\[GL_3(L),\ \ \ PGL_3(L),\ \ \ U_3(L),\ \ \ \mathbb{D}_3^{\times}.\]
\end{thm}
\begin{proof}
By \S \ref{ADLVdiv}, it is enough for us to check it for the group $PGL_3(L)$ and $U_3(L)$. By \S \ref{ADLVbasic}, we know the emptiness/nonemptiness pattern for basic $b'$. For nonbasic $b$, using the Dimension$=$Degree theorem and class polynomials in Chapter 3, if $\tilde{w}\in\widetilde{W}$ with sufficiently large length, we check case-by-case that $X_{\tilde{w}}(b)\neq\emptyset$ if and only if $X_{\tilde{w}}(b')\neq\emptyset$. If they are nonempty, it remains to check the comparison of their dimensions and we check case-by-case as well. For example, let $G(L)=PGL_3(L)$ and $b'=1$. If $b$ corresponds to $\mathbb{O}_{\lambda_0}$, where $\lambda_0\in P_+\cap Q_{sh}$. For those $\tilde{w}\in\widetilde{W}$ with sufficiently large length, and $X_{\tilde{w}}(b)\neq\emptyset$, we have\[\dim X_{\tilde{w}}(b)-\dim X_{\tilde{w}}(1)=\frac{1}{2}\ell(\mathbb{O}_{\lambda_0})-\langle\lambda_0,\ 2\rho\rangle.\]Since $\lambda_0\in P_+$, $\ell(\mathbb{O}_{\lambda_0})=\ell(t^{\lambda_0})=\langle\lambda_0,\ 2\rho\rangle$. By direct calculation, we have $def(1)=def(b)$. Thus\[\dim X_{\tilde{w}}(b)=\dim X_{\tilde{w}}(1)-\langle\nu_b,\ \rho\rangle+\frac{1}{2}(def(1)-def(b)).\]If $b'=\tau$ and $b$ corresponds to $\mathbb{O}_{\lambda_0,\tau}$, where $\lambda_0\in P_+\cap Q_{sh}$ and $\lambda_0\neq (2i-1)\alpha_1+i\alpha_2$ for all $i\in\mathbb{N}_+$. Thus we have\[\dim X_{\tilde{w}}(b)-\dim X_{\tilde{w}}(\tau)=\frac{1}{2}(\ell(\mathbb{O}_{\lambda_0,\tau})+2)-\langle\nu_{\mathbb{O}_{\lambda_0,\tau}},\ 2\rho\rangle.\]In this case, we have $\nu_b=\nu_{\mathbb{O}_{\lambda_0,\tau}}=\lambda_0-\frac{1}{3}(\alpha_1+2\alpha_2)$, $\ell(\lambda_0)=\ell(\mathbb{O}_{\lambda_0,\tau})+2$, and $def(\tau)-def(b)=2$, thus\[\dim X_{\tilde{w}}(b)=\dim X_{\tilde{w}}(\tau)-\langle\nu_b,\ \rho\rangle+\frac{1}{2}(def(\tau)-def(b)).\]For other cases, the methods are quite similar and we will omit them.
\end{proof}
\subsection{Affine Deligne-Lusztig varieties in the affine Grassmannian}
\ \ \ \ We recall the affine Deligne-Lusztig variety in the affine Grassmannian first. We keep the notations as in Section \S\ref{Prem}. Let $\mathbb{G}$ be the smooth affine group scheme associated with the special vetex of the Bruhat-Tits building of $G$. We denote by $L^+\mathbb{G}(R)=\mathbb{G}(R[[\epsilon]])$ the infinite dimensional affine group scheme. The twisted \emph{affine Grassmannian} is defined by the fpqc quotient $\textbf{Gr}=LG/L^+\mathbb{G}$. We have the Cartan decomposition\[G(L)=\bigsqcup_{\mu\in P_+}L^+\mathbb{G}(\mathbf{k})\epsilon^{\mu} L^+\mathbb{G}(\mathbf{k}),\ \ \ \ \ Gr(\mathbf{k})=\bigsqcup_{\mu\in P_+}L^+\mathbb{G}(\mathbf{k})\epsilon^{\mu} L^+\mathbb{G}(\mathbf{k})/L^+\mathbb{G}(\mathbf{k}).\]
\begin{defn}
Let $b\in G(L)$ and $\mu\in P_+$, the affine Deligne-Lusztig variety $X_{\mu}(b)$ in the affine Grassmannian $Gr$ is defined by\[X_{\mu}(b)(\mathbf{k})=\{g\in G(L)\mid gb\sigma(g)^{-1}\in L^+\mathbb{G}(\mathbf{k})\epsilon^{\mu} L^+\mathbb{G}(\mathbf{k})/L^+\mathbb{G}(\mathbf{k}).\}\]
\end{defn}
Let $w_0$ be the longest element in $W$. We notice that any $W\times W$-coset of $\widetilde{W}$ contains a unique maximal element and this element is of the form $w_0t^{\lambda}$ for some $\lambda\in P_+$. An element in this double coset is of the form $xt^{\lambda}y$ for $x\in W$ and $y\in^{I(\lambda)}W$. Here $^{I(\lambda)}W=\{s_i\in\mathbb{S}\mid \langle\lambda,\ \alpha_i\rangle=0\}$. Let $\tilde{w}=w_0t^{\lambda}$,  and $\lambda\in P_+$. For the special element $w_0t^\lambda$, we have
\begin{thm}[He]\label{dimw0}
Let $\lambda\in P_+$, $x\in W$ and $y\in^{I(\lambda)}W$. For any $b\in G(L)$,\[\dim X_{xt^\lambda y}(b)\leqslant\dim X_{w_0t^\lambda}(b)-\ell(w_0)+\ell(x).\]
\end{thm}
What's more, for the special element $w_0t^\lambda$, there is a complete answer for the emptiness/nonemptiness pattern and dimension formula (see \cite{He3} Theorem 6.1).
\subsection{Leading coefficient of $f_{w_0t^\lambda, \mathbb{O}}$}
Let $p: \textbf{Fl}\longrightarrow \textbf{Gr}$ be the projection. Each fiber of $p$ is isomorphic to $L^+\mathbb{G}/I$, which is of dimension $\ell(w_0)$. Since $p^{-1}X_{\lambda}(b)=\bigsqcup_{\tilde{w}\in Wt^{\lambda} W}X_{\tilde{w}}(b)$ and Theorem \ref{dimw0}, for any $\tilde{w}\in Wt^{\lambda}W$ we have $\dim X_{\tilde{w}}(b)\leqslant\dim X_{w_0t^\lambda}(b)$. Based on information of class polynomials and the reduction method, it is expected that the irreducible components of $X_{\tilde{w}}(b)$ of maximal dimension are controlled by the leading coefficient of the corresponding class polynomial $f_{w_0t^\lambda, \mathbb{O}}$. And we denote by $f_{\tilde{w},\mathbb{O}_b}$ the class polynomial corresponding to $X_{\tilde{w}}(b)$, and it is indicated by the ``Dimension$=$Degree" theorem.

Instead of considering the irreducible components of $X_{\tilde{w}}(b)$ of maximal dimension, we are going to consider the leading coefficient of the corresponding class polynomial.
\begin{prop}
Given $\tilde{w}\in\widetilde{W}$, we denote $N_0$ to be the leading coefficient of $f_{\tilde{w},\mathbb{O}_2}$. We denote $L(f_{\tilde{w},\mathbb{O}})$ to be the leading coefficient of $f_{\tilde{w},\mathbb{O}}$. Moreover, we assume that $\tilde{w}=w_0t^\lambda$, where $\lambda\in P_+$.

(1) If $\lambda=k_0(\alpha_1+\alpha_2)$ with $k_0\in\mathbb{N}$ and it is large enough, then\[L(f_{\tilde{w},\mathbb{O}_b})=\left\{
 \begin{array}{ll}
 N_0-2i, & b\longleftrightarrow\mathbb{O}_{i\alpha_1+2i\alpha_2}\ or\ \mathbb{O}_{2i\alpha_1+i\alpha_2};\\
  N_0-i, & b\longleftrightarrow\mathbb{O}_{\lambda},\ \lambda\in\{i(\alpha_1+\alpha_2)\}\sqcup E_{i(\alpha_1+\alpha_2)}\sqcup E'_{i(\alpha_1+\alpha_2)},\ or\\
 & b\longleftrightarrow\mathbb{C}_i\ or\ \mathbb{C}'_i.\\
  \end{array}
  \right.\]
 Here $i\in\mathbb{N}$ and $k_0$ relates to $i$ is quite large.

(2) If $\lambda=i_0(\alpha_1+2\alpha_2)+k_0(\alpha_1+\alpha_2)$ with $k_0\in\mathbb{N}$ and it is large enough, then\[L(f_{\tilde{w},\mathbb{O}_b})=\left\{
 \begin{array}{ll}
 N_0, & b\longleftrightarrow\mathbb{O}_{i_0\alpha_1+2i_0\alpha_2},\ or\ \mathbb{C}_i,\ \ \mathbb{C}'_i\ (i\leqslant i_0);\\
 N_0-j, & b\longleftrightarrow\mathbb{O}_{\lambda},\ \lambda\in\{(i_0+j)\alpha_1+(2i_0+j)\alpha_2\}\sqcup E_{(i_0+j)\alpha_1+(2i_0+j)\alpha_2}\\
 &\ \ \ \ \ \ \ \ \ \ \ \ \ \ \ \sqcup E'_{(i_0+j)\alpha_1+(2i_0+j)\alpha_2},\ or\\
 & b\longleftrightarrow\mathbb{C}_{i_0+j}\ or\ \mathbb{C}'_{i_0+j}.
  \end{array}
  \right.\]
(3) If $\lambda=i_0(2\alpha_1+\alpha_2)+k_0(\alpha_1+\alpha_2)$ with $k_0\in\mathbb{N}$ and it is large enough, then\[L(f_{\tilde{w},\mathbb{O}_b})=\left\{
 \begin{array}{ll}
 N_0, & b\longleftrightarrow\mathbb{O}_{2i_0\alpha_1+i_0\alpha_2},\ or\ \mathbb{C}_i,\ \ \mathbb{C}'_i\ (i\leqslant i_0);\\
 N_0-j, & b\longleftrightarrow\mathbb{O}_{\lambda},\ \lambda\in\{(2i_0+j)\alpha_1+(i_0+j)\alpha_2\}\sqcup E_{(2i_0+j)\alpha_1+(i_0+j)\alpha_2}\\
 &\ \ \ \ \ \ \ \ \ \ \ \ \ \ \ \sqcup E'_{(2i_0+j)\alpha_1+(i_0+j)\alpha_2},\ or\\
 & b\longleftrightarrow\mathbb{C}_{i_0+j}\ or\ \mathbb{C}'_{i_0+j}.\\
  \end{array}
  \right.\]
\end{prop}
\begin{proof}
We check (1) and the others are similar. If $b$ corresponds to $\mathbb{O}_{i\alpha_1+2i\alpha_2}$ (or $\mathbb{O}_{2i\alpha_1+i\alpha_2}$), where $i\in\mathbb{N}_+$. Then $\mathbb{O}_b=\mathbb{C}_{2i}$ (or $\mathbb{O}_b=\mathbb{C}'_{2i}$). In this case, $f_{\tilde{w},\mathbb{O}_2}=N_0(v-v^{-1})$ and $f_{\tilde{w},\mathbb{C}_{2i}}=f_{\tilde{w},\mathbb{C}'_{2i}}=(N_0-2i)(v-v^{-1})^2$. If $b$ corresponds to $\mathbb{O}_{\lambda_0}$ where $\lambda_0=i(\alpha_1+\alpha_2)$ (or $\lambda_0\in E_{i(\alpha_1+\alpha_2)}\sqcup E'_{i(\alpha_1+\alpha_2)}$), then $\mathbb{O}_b=\mathbb{O}_{\lambda_0}$, and $f_{\tilde{w},\mathbb{O}_{\lambda_0}}=(N_0-i)(v-v^{-1})^3+(v-v^{-1})$ (or $f_{\tilde{w},\mathbb{O}_{\lambda_0}}=(N_0-i)(v-v^{-1})^3$). If $b$ corresponds to $\mathbb{C}_i$ or$\mathbb{C}'_i$ , then $\mathbb{O}_b=\mathbb{C}_{i}$ or $\mathbb{O}_b=\mathbb{C}'_{i}$, and $f_{\tilde{w},\mathbb{C}_{i}}=f_{\tilde{w},\mathbb{C}'_{i}}=(N_0-i)(v-v^{-1})^2$. Thus (1) is proved.
\end{proof}

\section*{Acknowledgment}
I would like to express my deepest gratitude to my thesis supervisor Professor Xuhua He, who has supported my PhD study generously and kindly, given me careful guidance, and shared his brilliant ideas with me. Also I would like to thank Department of Mathematics, The Hong Kong University of Science and Technology for providing me with postgraduate studentships.

\end{document}